\documentclass[3p,12pt]{elsarticle}

\usepackage{mathtools}
\usepackage{amsthm}
\usepackage{amsmath}
\usepackage{amsfonts}
\usepackage{amssymb}
\usepackage{appendix}
\usepackage{epigraph}
\usepackage{graphicx}
\usepackage{color}
\usepackage{xcolor}
\usepackage{psfrag}
\usepackage{placeins}
\usepackage[colorlinks=true]{hyperref}
\hypersetup{urlcolor=blue,linkcolor=black,citecolor=black}

\def\WW{\mathcal W}
\def\RR{\mathbb R}
\def\NN{\mathbb N}

\def\eps{\varepsilon}

\def\supp{\mathrm{supp}}
\def\d{\mathrm{d}}
\def\dv{\mathrm{div}}
\def\de{\partial}

\newcommand{\nm}[1]{\left\| #1 \right\|}
\def\size{0.35\textwidth}

\theoremstyle{definition}
\newtheorem{defi}{Definition}
\newtheorem{hyp}{Assumption}
\newtheorem{rem}{Remark}
\newtheorem*{rem*}{Remark}
\theoremstyle{plain}
\newtheorem{lem}{Lemma}
\newtheorem{theor}[lem]{Theorem}
\newtheorem{cor}[lem]{Corollary}

\newtheorem*{conj*}{Conjecture}

\journal{Nonlinear Analysis}

\begin{document}

\begin{frontmatter}

\title{The role of a strong confining potential in a nonlinear
Fokker--Planck equation}

\author{Luca Alasio}
\ead{luca.alasio@gssi.it}

\address{Gran Sasso Science Institute, L'Aquila 67100, Italy.}

\author{Maria Bruna}
\ead{bruna@maths.ox.ac.uk}

\address{Mathematical Institute, University of Oxford, Oxford OX2 6GG, United Kingdom.}

\author{Jos\'e Antonio Carrillo}
\address{Department of Mathematics, Imperial College London, London SW7 2AZ, United Kingdom.}
\ead{carrillo@imperial.ac.uk}

\begin{abstract}
We show that solutions of nonlinear nonlocal Fokker--Planck equations in a bounded domain with no-flux boundary conditions can be approximated by Cauchy problems with increasingly strong confining potentials defined in the whole space. Two different approaches are analyzed, making crucial use of uniform estimates for $L^2$ energy functionals and free energy (or entropy) functionals respectively. In both cases, we prove that the weak formulation of the problem in a bounded domain can be obtained as the weak formulation of a limit problem in the whole space involving a suitably chosen sequence of large confining potentials. The free energy approach extends to  the case degenerate diffusion.
\end{abstract}

\begin{keyword}
Fokker--Planck equations \sep nonlocal models \sep strong confinement \sep degenerate diffusion

\MSC 35B30 \sep 35K61 \sep 35Q84

\end{keyword}

\end{frontmatter}


\section{Introduction} \label{sec:intro}

In this paper, we consider a nonlinear Fokker--Planck equation of the form 
\begin{align}
	\label{eq_bounded}
	\begin{aligned}
	\de_{t}u &=  \dv\left[ \nabla\phi(u) + u\nabla V_0 + u\nabla (W*u) \right], \quad  x\in \Omega, t>0, \\
	\end{aligned}
\end{align}
where $u(x, t) \ge 0$ satisfies no-flux boundary conditions on a bounded and connected domain $\Omega \subset \mathbb R^d$ of class $C^2$, for dimension $d\geq 1$ in $\NN$, and a suitable initial condition that we will specify later. The function $\phi(\cdot)$ represents nonlinear diffusion, $W$ is a symmetric interaction potential, and $V_0$ is an external potential.

Equation \eqref{eq_bounded} is often used to describe a system of interacting particles at the macroscopic level and explain how individual-level mechanisms give rise to population-level or collective behavior. Systems of interacting particles play a key role in many physical and biological applications, including granular materials \cite{BCP97,BCCP98}, self-assembly of nanoparticles \cite{HP06}, colloidal systems \cite{Glanz:2016jz}, ionic transport \cite{Horng:2012ioa}, cell motility \cite{Hillen:2008ita}, animal swarms \cite{Carrillo:2009hu}, pedestrian dynamics \cite{burger2011continuous}, and social sciences \cite{PT13,T06}. For example, equation \eqref{eq_bounded} with $\phi = u$ and $W = 0$ can be used to describe a system of noninteracting Brownian particles under the influence of an external field $V_0$, representing a chemical concentration in the case of chemotaxis. Deviations from Brownian motion, such as in the case of transport through porous media, can be modelled changing the diffusion term to $\phi = u^m$ with $m>1$. Interactions between particles may arise in the macroscopic model in two forms: either as a modification of the diffusion term $\phi$, for example $\phi = u + \beta u^2$, or as a nonlocal convolution. The former typically arises from short-range repulsive interactions between particles (such as excluded-volume interactions) \cite{Bruna:2012cg,CC06,CCH17}, whereas the latter is used to model long-range attractive-repulsive interactions (such as electrostatic or chemoattractive interactions) \cite{Bruna:2017vr,CCH17,TBL06}. 

The goal of this paper is to understand how the solutions of \eqref{eq_bounded} in the bounded domain $\Omega$ relate to the solutions of the following equation in the whole space, as $k\to \infty$:
\begin{align}
	\label{eq_whole}
	\begin{aligned}
	\de_{t}u_k &=  \dv\left[ \nabla\phi(u_k)  + u_k\nabla V_k + u_k\nabla (W*u_k) \right], \quad  x\in \mathbb R^d, t>0,\\
    	\end{aligned}
\end{align}
where the confinement potential is fixed in the bounded domain, that is, $V_k(x) = V_0(x)$  for $x\in \Omega$ and it becomes stronger outside $\Omega$ as $k\to\infty$.

\def \scc {0.8}
\def \scl {1.0}
\begin{figure}
\unitlength=1cm
\begin{center}
\vspace{3mm}
\psfrag{a}[][][\scl]{(a)}
\psfrag{b}[][][\scl]{(b)}
\psfrag{k}[][][\scl]{$k$} \psfrag{O}[][][\scl]{$\Omega$} 
\psfrag{OK}[][][\scl]{$\Omega_k$} 
\includegraphics[width=.45\textwidth]{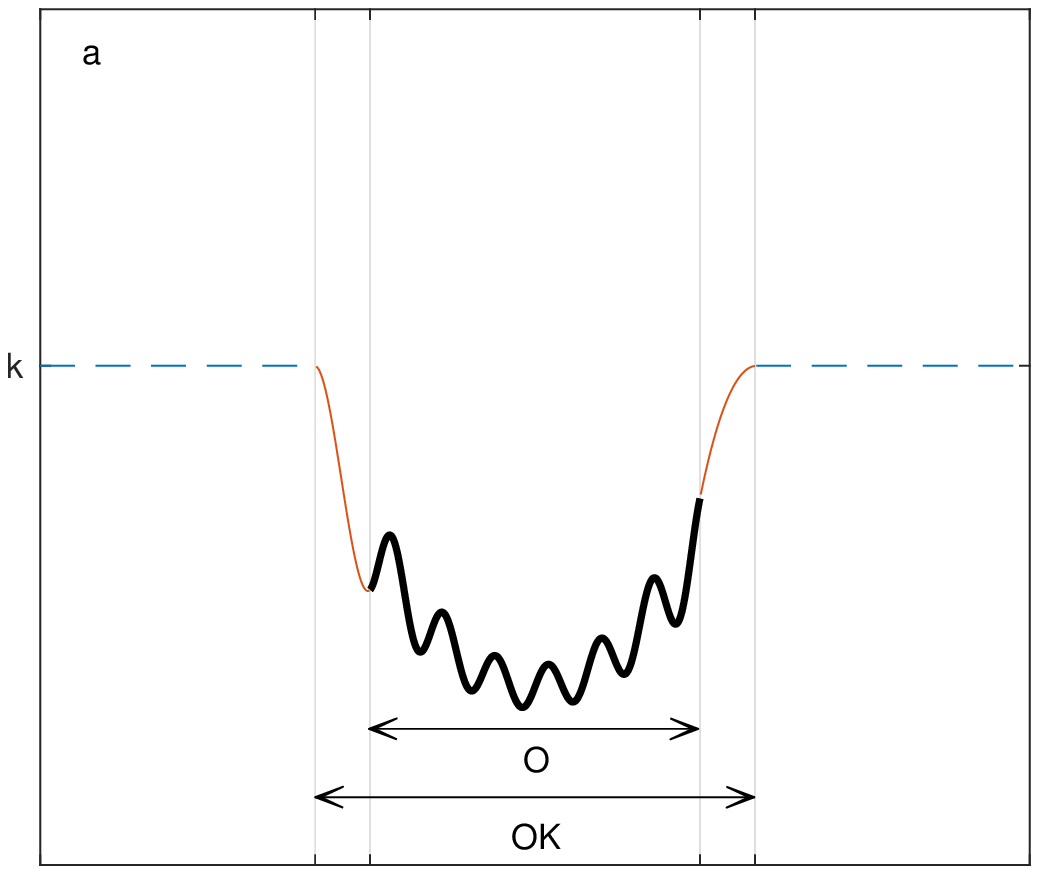}\hfill 
\includegraphics[width=.45\textwidth]{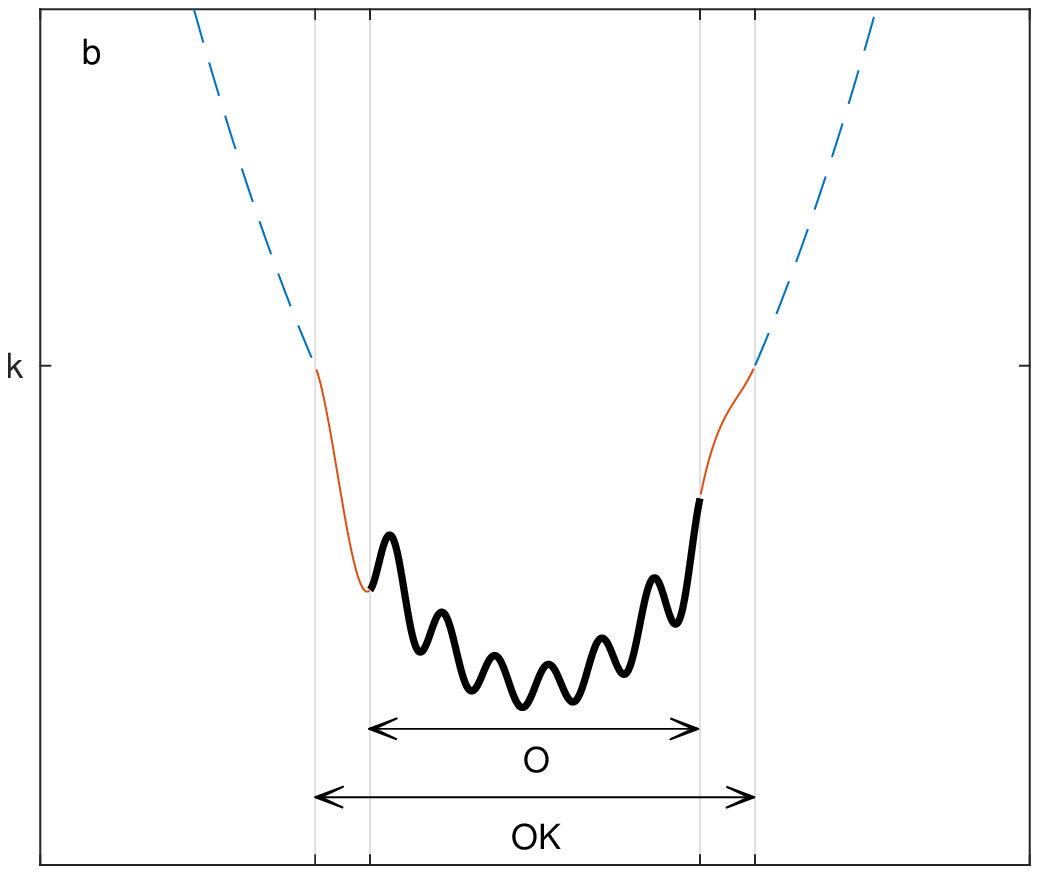}
\caption{Sketch of the global potential $V_k(x)$ for (a) the $L^2$ setting (see Definition \ref{defi:vk}) and (b) the free energy setting (see Definition \ref{defi:vk-ent}). } 
\label{fig:potential}
\end{center}
\end{figure}

In particular, we consider  potentials $V_k$ of the form depicted in Figure \ref{fig:potential} (see Definitions \ref{defi:vk} and \ref{defi:vk-ent}). The parameter $k$ determines the level of confinement ($k\to \infty$ for an infinite confinement). Our aim is to control the behavior of the solution $u_k$ to \eqref{eq_whole} when $V_k$ becomes a strong confinement potential outside $\Omega$. More precisely, we show that infinite confinement is equivalent to solving equation \eqref{eq_bounded} in a bounded domain with no-flux boundary conditions for certain class of initial data supported in $\Omega$. A similar set-up was used in \cite{alexander2016fokker} to study the limit of a local, non-degenerate, fully nonlinear problem with a method based on super- and subsolutions.

Our approximation by strong confinement has not only a theoretical interest but also a practical added value from the numerical viewpoint. Sometimes solving \eqref{eq_bounded} in a bounded domain is hindered by the geometry of $\Omega$. The strong confinement approximation is potentially useful to approximate numerically problems of the form \eqref{eq_bounded} by problems in the whole space \eqref{eq_whole} in square geometries large enough for the domain of interest. Solving in Cartesian grids is always much easier than producing good meshes for the approximated domains.

Problem \eqref{eq_bounded} is sometimes referred to as \emph{aggregation-diffusion equation} or \emph{drift-diffusion-interaction} equation and it has been recently studied by different authors. Free energy methods have been successfully applied to study long time behaviour of the solutions of problem \eqref{eq:main} in the setting of gradient flows and Wasserstein spaces. For example, long time asymptotics to nonlinear nonlocal Fokker--Planck equations have been studied in
\cite{carrillo2003kinetic,carrillo2006contractions}. Many results concerning convergence to equilibrium can be found in particular cases (with or without diffusive terms) in
\cite{arnold2001convex,Bonnaillie,carrillo2001entropy,carrillo1998exponential}.
Gradient flows in the Wasserstein sense were first described for the linear Fokker--Planck equation in \cite{jordan1998variational}; an extensive theory of the subject was later given in  \cite{ambrosio2008gradient}.
For classical results on parabolic equations we refer to
\cite{alt1983quasilinear,ladyzhenskaia1988linear} and more recent results on well-posedness for the aggregation-diffusion equations are presented in \cite{bertozzi2009existence}. 
A comparison between free energy and classical $L^2$ methods can be found in \cite{otto2001geometry}.
Some of the techniques we use in Section \ref{sec:L2} involving the weight $\exp(V)$ were also employed in \cite{alasio2018stability,markowich2013stationary}. 
In this paper we consider the problem of strong confinement using both an $L^2$ approach and a free energy approach. The $L^2$ approach is restricted to non-degenerate diffusion, while the free energy approach allows us to consider also degenerate diffusion but provides weaker estimates. 

The paper is organized as follows. In Section \ref{sec:outline} we introduce the key definitions and give an outline of the main results. Section \ref{sec:L2} is devoted to the special cases in which $\phi(u) = u$ or $W = 0$, and the problem is treated with classical energy methods in the $L^2$ setting. We consider the fully nonlinear nonlocal problem in the free energy setting by free energy methods in Section \ref{sec:entropy}. Finally, we conclude in Section \ref{sec:numerics} with some numerical examples illustrating the main results of this paper together with some numerical exploration for initial data supported outside the confined domain $\Omega$.

\section{Outline of the results} \label{sec:outline}

Let us consider the nonlinear nonlocal Fokker--Planck equation with confinement potential $V$ and interaction potential $W$ given by
\begin{align}\label{eq:main}
	\begin{aligned}
	\de_{t}u &=  \dv\left[ \nabla\phi(u)  + u\nabla V + u\nabla (W*u)  \right], \quad  x\in \RR^d, t>0, \\
	u(x,0) & = u_0(x).
	\end{aligned}
\end{align}
We denote by $Q_T$ the space $\RR^d \times [0,T]$ and by $\Omega_T$ the space $\Omega \times [0,T]$. The main results of this work are obtained under two different set of assumptions on the potentials $V$, $W$, the nonlinearity of the diffusion $\phi$, and the initial data. We will label them as the $L^2$ setting and the free energy setting.

\subsection{\texorpdfstring{$L^2$}{} setting, main results}\label{subsec1}

\begin{hyp}[$L^2$ setting]\label{assumptions}

We consider the following set of assumptions:

\begin{enumerate}
\item $0\leq u_0(x)\leq M_0$, $\supp ( u_0 ) \subseteq \Omega$.
\item $W\in \WW^{1,\infty}(\RR^d)$ is symmetric, $\nabla W \in L^1(\RR^d)$ and, without loss of generality, $W\geq 0$.
\item $V\in \WW^{1,\infty}(\RR^d)$ and, without loss of generality, we also assume $V\geq 0$.
\item $\phi\in C^1([0,\infty))$ has the form $\phi(s) = s + \sigma(s)$, $\phi(0) = 0$, and is increasing. Moreover, there exist constants $\mu >0$ and $b\geq a > 0$ such that
\begin{equation}\label{eq:growth0}
\mu s^{a} \leq \sigma'(s) \leq  \frac{1}{\mu} s^{b}.
\end{equation}
\end{enumerate}
\end{hyp}

In the $L^2$ setting, weak solutions to \eqref{eq:main} are defined as follows.

\begin{defi}[$L^2$ solution]\label{def:weaksol1} 
Suppose that Assumption \ref{assumptions} holds. We say that $u$ is an $L^2$ weak solution of \eqref{eq:main} if
$$
u\in L^2(0,T;H^1(\RR^d))\cap C^0(0,T;L^2(\RR^d)),
\;
\nabla \phi(u) \in L^2(Q_T),
\;
\de_t u\in L^2(0,T;H^{-1}(\RR^d)),
$$
and, for all test functions $\eta\in C^\infty(Q_T)$,
\begin{equation}\label{eq:weakform_def}
\left.\int_{\RR^d} u(t)\eta(t) \d x \right|_{t=0}^{t=T} - 
\int_0^T\int_{\RR^d}    
\left[
(\nabla \phi(u) 
+  u \nabla V
+  u \nabla W*u
 )
 \cdot \nabla \eta  
 -u \de_t\eta 
\right] 
 \,
\d x \d t
=0.
\end{equation}
The initial datum is satisfied in the $L^2$ sense.
\end{defi}

We now define the sequences of confinement potentials associated to a fixed confinement potential $V_0\in \WW^{1,\infty}(\Omega)$, $V_0\geq 0$, that we will use in Section \ref{sec:L2} (see Figure \ref{fig:potential}(a)).

\begin{defi}[Sequence of potentials, $L^2$ setting]\label{defi:vk}
We define the following sequence of  potentials $V_k\in \WW^{1,\infty}(\mathbb R^d)$ for $k\geq 1$ as
\begin{equation} \label{potential_Vk}
V_{k}(x)=\begin{cases}
V_0(x) & x\in\Omega,\\
\psi_k(x) & x\in\Omega_k\setminus \Omega,\\
k & x\in \mathbb R^d \setminus \Omega_k,
\end{cases}
\end{equation}
where $\Omega_k$ is an extended domain around $\Omega$,
$$
\Omega_k = \left\{ x+\frac{1}{k}e \, \big| \, x\in \Omega, e\in S^d  \right\},
$$
so that $\Omega_{k}\searrow\Omega$ as $k\to\infty$, and $\psi_k$ is a suitable $C^1$ extension of $V$.
\end{defi}

Our main result in the $L^2$ setting concerns the convergence of the sequence of solutions $u_k$, which are defined in $\RR^d$, to a limit function $u$ that solves a problem in the bounded domain $\Omega$.

\begin{theor}[Main result - $L^2$ setting]\label{thm:main1}
Assume that the conditions in Assumption \ref{assumptions} are satisfied and, additionally, suppose that one of the following conditions holds: either $\phi(s)= s$ or $W = 0$.

Consider a solution $u_k$ of the Cauchy problem \eqref{eq:main} in the sense of Definition \ref{def:weaksol1} with $V=V_k$ satisfying the conditions in Definition \ref{defi:vk}.
Then $u_k$ converges for $k\to\infty$ to a function $u$ such that 
$$
u\in L^2(0,T;H^1(\Omega))\cap C^0(0,T;L^2(\Omega)),
\;
\nabla \phi(u) \in L^2(\Omega_T),
\;
\de_t u\in L^2(0,T;H^{1}(\Omega)'),
$$
and, for all $\eta\in C^\infty(\Omega_T)$, we have: 
\begin{equation}\label{eq:limitweakform1}
\left.\int_{\Omega} u(t)\eta(t) \d x \right|_{t=0}^{t=T}
- 
\int_0^T\int_{\Omega}    
\left[
(\nabla \phi(u) 
+  u \nabla V_0
+  u \nabla W*u
 )
 \cdot \nabla \eta  
 -u \de_t\eta 
\right] 
 \,
\d x \d t
=0.
\end{equation}
In other words, $u$ is a weak solution of the Cauchy problem to \eqref{eq_bounded} with no-flux boundary conditions and initial datum $u_0$.
\end{theor}

\begin{rem}[Convolution on a bounded domain]
Given $f : \Omega \subseteq \RR^d \to\RR $, we use the following convention:
$$
(W*f)(x) = \int_\Omega W(x-y) f(y) \d y.
$$
\end{rem}

\subsection{Free energy setting, main results}

We now introduce a different set of assumptions and we prove a convergence result corresponding to Theorem \ref{thm:main1} in a new setting. This allows us to consider confinement and interaction potentials diverging as $|x|\to\infty$, as well as degenerate diffusion terms.

\begin{hyp}[Free energy setting]\label{assumptions-ent}

We consider the following set of assumptions:

\begin{enumerate}
\item $u_0(x)\geq 0$, $\int_{\RR^d} u_0 \d x = 1$ and $\supp ( u_0 ) \subseteq \Omega$.
\item $W\in \WW_{loc}^{1,\infty}(\RR^d)$ is symmetric and, without loss of generality, $W\geq 0$. 
\item $V\in \WW_{loc}^{1,\infty}(\RR^d)$ and we assume that $V\geq c |x|^2$ for $|x|\to\infty$ and for some $c>0$.
\item $\phi\in C^1([0,\infty))$ has the form $\phi(s) = s + \sigma(s)$, $\phi(0) = 0$, and is increasing. We suppose that there exist constants $\mu >0$ and $b\geq a > 0$ such that relation \eqref{eq:growth0} holds.
\end{enumerate}
\end{hyp}
Notice that we do not assume boundedness or decay of $V$ and $W$ at infinity, but only local regularity. The initial datum can also be unbounded as we only assume non-negativity and integrability. The key ingredient is the gradient flow structure of equations \eqref{eq_bounded} and \eqref{eq_whole}.
We denote the $2$-Wasserstein space of probability measures by $\mathcal{P}_2(\RR^d)$ endowed by the $2$-Wasserstein distance $d_2$.

\begin{defi}[Free energy solution]\label{def:weaksol2}
Suppose that Assumption \ref{assumptions-ent} is satisfied.
We say that $u\in C([0,T],\mathcal{P}_2(\RR^d))\cap L^\infty(0,T;L^1(\RR^d))$ is an entropy solution of \eqref{eq:main} if $u$ is a distributional solution in $\mathcal{D}'(\RR^d)$ and it satisfies 
$$
u(t,\cdot) \to u_0 \text{ in } \mathcal{P}_2(\RR^d)
\; \text{ as } t\to 0 ,
$$
for  all $T>0$,
$$
\int_{Q_T} u \left| 
\frac{1}{u}\nabla\phi(u) + \nabla V + \nabla W * u 
\right|^2 \d x \d t
 < \infty,
$$
and, for $\Theta(u) = \int_0^u\int_1^s \frac{\phi'(r)}{r} \d r\d s$, it is a gradient flow in $\mathcal{P}_2(\RR^d)$ for the entropy functional:
$$
E[u(t)] = \int_{\RR^d} \left[\Theta(u) +  uV + \frac{1}{2}u(W*u)\right]  \d x.
$$
\end{defi}

In Section \ref{sec:entropy} we will consider the following sequence of potentials (see Figure \ref{fig:potential}(b)).

\begin{defi}[Sequence of potentials, free energy setting]\label{defi:vk-ent}
We define the following sequence of  potentials $V_k\in \WW^{1,\infty}_{loc}(\mathbb R^d)$:
\begin{equation} \label{potential_Vk-ent}
V_{k}(x)=\begin{cases}
V_0(x) & x\in\Omega,\\
\psi_k(x) & x\in\Omega_k\setminus \Omega,\\
\zeta_k(x) & x\in \mathbb R^d \setminus \Omega_k,
\end{cases}
\end{equation}
where $\Omega_k$ is an extended domain around $\Omega$,
$$
\Omega_k = \left\{ x+\frac{1}{k}e \, \big| \, x\in \Omega, e\in S^d  \right\},
$$
so that $\Omega_{k}\searrow\Omega$ as $k\to\infty$.
Here $\zeta_k(x) \ge k$ is such that $V_k(x)\geq c|x|^2$ for $|x|$ sufficiently large
(in particular $\int_{\mathbb R^d} e^{-\zeta_k(x)} \d x < \infty$), and $\psi_k(x)$ is a $C^1$ interpolant between the values of $V_0$ on $\partial \Omega$ and $\zeta_k$ outside $\Omega_k$.
\end{defi}

\begin{theor}[Main result - Free energy setting]\label{thm:main2}
Assume that the conditions in Assumption \ref{assumptions-ent} are met. Consider a solution $u_k$ of problem \eqref{eq:main} in $\RR^d$ in the sense of Definition \ref{def:weaksol2} with $V=V_k$ satisfying the conditions in Definition \ref{defi:vk-ent}. Then $u_k$ converges for $k \to \infty$ to a function $u$ satisfying  the following weak formulation in $\Omega$:
$$
\left.\int_{\Omega} u(t)\eta(t) \d x \right|_{t=0}^{t=T}
- 
\int_0^T\int_{\Omega}   \left[
\left(  
\nabla \phi(u) 
+  u \nabla V_0 
+  u \nabla W*u 
\right)\cdot \nabla \eta  
- u \de_t \eta 
\right]
\d x \d t = 0,
$$
for any $\eta\in C_0^\infty(\Omega_T)$. The initial datum is also satisfied in $\mathcal{P}_2(\Omega)$.

The same result holds for $\phi(\cdot) = \sigma(\cdot)$, that is, in the case of degenerate diffusion.
\end{theor}

\section{Analysis via \texorpdfstring{$L^2$}{L2}  estimates} \label{sec:L2}

The well-posedness of problem \eqref{eq:main} in the $L^2$ setting can be deduced from the existence and uniqueness results in \cite{bertozzi2009existence}. Let us discuss some of the properties of the solutions. Without loss of generality we assume that $\phi$ is the restriction of a $C^1$ function on the whole line to the positive semi-axis.

\begin{lem}[Non-negativity and conservation of mass]\label{lem:mass-pos}
Let Assumption \ref{assumptions} hold. Any weak solution $u$ of problem \eqref{eq:main} in the sense of Definition \ref{def:weaksol1} is non-negative and furthermore it preserves mass, that is,
$$
\int_{\RR^d} u\, \d x= \int_{\RR^d} u_0 \,\d x.
$$
\end{lem}

\begin{proof}
To obtain conservation of mass, it is sufficient to test with the function $\eta=1$ and integrate by parts  to obtain $0=\langle \de_t u, 1\rangle = \de_t \int_{\RR^d} u \d x$.

In order to obtain non-negativity, we consider a solution $u$ of problem \eqref{eq:main} (in the sense of Definition \ref{def:weaksol1}) and we test the equation against $\theta=(u)_-$ (which is non-negative and supported in the set $\{ u \leq 0 \}$). 
In particular, for a.e. $t$ we have 
$$
 \frac{d}{dt} \int_{\RR^d} \frac{1}{2} \theta^2 \d x
+ \int_{\RR^d} \left[\phi'(u)| \nabla \theta|^2 
+  \theta \nabla ( V +W*u) \cdot \nabla \theta \right]
\d x
=0,
$$
which implies, since $\phi'(s)\geq 1$,
$$
\int_{\RR^d} \frac{1}{2} \theta^2 \d x 
+\int_{Q_T} | \nabla \theta|^2 \d x 
\leq 
\nm{\nabla ( V +W*u)}_{L^\infty(Q_T)} \int_{Q_T} \left[ \frac{\eps}{2} | \nabla \theta|^2 + \frac{1}{2\eps} | \theta|^2 \right] \d x,	
$$
for $\eps= \nm{\nabla ( V +W*u)}_{L^\infty(Q_T)}^{-1}$.
Notice that we have already established that $\int_{\RR^d} u \d x  = \int_{\RR^d} u_0 \d x$ and that the quantity $\nm{\nabla ( V +W*u)}_{L^\infty(Q_T)}$ is finite thanks to Assumption \ref{assumptions}.
Thus we have
$$
\int_{\RR^d} \theta^2 \d x 
+\int_{Q_T}  | \nabla \theta|^2 \d x
\leq 
\nm{\nabla ( V +W*u)}_{L^\infty(Q_T)}^2 \int_{Q_T}  |  \theta|^2 \d x.
$$
Using Gronwall's inequality we obtain $\theta = (u)_-=0$ a.e. $(x,t)\in Q_T$.
\end{proof}

\subsection{Linear Fokker--Planck equation} \label{sec:linear}

We begin with the simplest case with non-interacting particles ($W=0$, $\phi(s)=s$); in this case it is possible to work in an $L^2$ setting.

\begin{lem}[Energy identity and boundedness]\label{lem:energy}
Consider the scalar equation
\begin{align} \label{linear_problem}
\begin{aligned}
\de_{t}u &=  \Delta u + \dv(u\nabla V), \quad x\in \RR^d, t>0,\\
 u(x,0) & =u_0(x).
 \end{aligned}
\end{align}
For every weak solution of problem \eqref{linear_problem} (in the sense of Definition \ref{def:weaksol1}), the following identity holds
\begin{equation}\label{eq:energy}
\int_{\RR^d} u(T)^{2}e^{V}\d x
+2\int_{0}^{T}\int_{\RR^d} 
e^{V}|\nabla u+u\nabla V|^{2}\d x\d t=
\int_{\RR^d} u_0{}^{2}e^{V}\d x,
\end{equation}
for a.e. $T>0$.
In addition, $u$ has exponential tails in $L^\infty(Q_T)$ in the following sense
\begin{equation}\label{eq:boundedness}
u_0\leq me^{-V} \Rightarrow u\leq me^{-V} \text{ a.e. } (t,x)\in Q_T,
\end{equation}
for a fixed $m\geq 0$.
\end{lem}

\begin{proof}
We test the equation against $ue^{V}$. We obtain
\begin{equation}\label{linear_en1}
	\int_{\RR^d} \de_{t}u\left(ue^{V}\right)\d x
=\int_{\RR^d} \left[\Delta u+\dv(u\nabla V)\right]ue^{V}\d x.
\end{equation}
Integrating the left-hand side of \eqref{linear_en1} in time, we deduce
\begin{align*}
\int_{0}^{T}\int_{\RR^d}
\de_{t}u\left(ue^{V}\right)\d x\d t  = 
 \int_{0}^{T}\int_{\RR^d}
\de_{t}\left(\frac{1}{2}u^{2}e^{V}\right)\d x\d t =  \int_{\RR^d}
 \frac{1}{2}u(T)^{2}e^{V}\d x-
 \int_{\RR^d}\frac{1}{2}u_0{}^{2}e^{V}\d x.
\end{align*}
Using integration by parts in the right-hand side of \eqref{linear_en1} we get
\begin{align*}
\int_{\RR^d} \left[\Delta u-\dv(u\nabla V)\right]\cdot\left(ue^{V}\right)\d x & =  
-\int_{\RR^d}\left(\nabla u+u\nabla V\right)\cdot\nabla\left(ue^{V}\right)\d x\\
 & =  
 -\int_{\RR^d} e^{V}|\nabla u+u\nabla V|^{2}\d x
 \leq  0.
\end{align*}
Collecting terms yields
\[
\int_{\RR^d}\frac{1}{2}e^{V} u(T)^{2}\d x
=\int_{\RR^d}\frac{1}{2}u_0{}^{2}e^{V}\d x
-\int_{0}^{T}\int_{\RR^d} e^{V}|\nabla u+u\nabla V|^{2}\d x\d t,
\]
as required.
In order to prove boundedness, we consider the function $\tilde{u} = ue^V$.
Let us rewrite equation \eqref{linear_problem} in terms of $\tilde{u}$:
$$
e^{-V} \de_t \tilde{u} =  \dv\left( e^{-V}\nabla \tilde{u} \right).
$$
We integrate the equation above against the test function $(\tilde{u}-m)_+$:
$$
\int_{\RR^d}\frac{1}{2} e^{-V} (\tilde{u}(t)-m)_+^{2}\d x
+
\int_{0}^{t}\int_{\RR^d}e^{-V}|\nabla (\tilde{u}-m)_+|^{2}\d x\d \tau
=
\int_{\RR^d}\frac{1}{2} e^{-V} (\tilde{u}(0)-m)_+^{2}\d x.
$$
Notice that $\tilde{u}(0) = u_0 e^V$, thus the right-hand side in the equality above vanishes. This means that $\tilde{u}(t)\leq m$ for a.e. 
$(t,x)\in Q_T$ and the proof is complete.
\end{proof}

\begin{cor} \label{cor:mass1}
Consider the assumptions of Lemma \ref{lem:energy} and recall that $\supp ( u_0 ) \subseteq \Omega$. 
Let $V=V_k$ as in Definition \ref{defi:vk} and consider $u=u_k$ as in Lemma \ref{lem:energy}.
Then $u_k(x,t)\to 0$  in $L^\infty([0,T]\times \Omega_k^{c})$ for $k\to\infty$.
\end{cor}
\begin{proof}
It is a direct consequence of \eqref{eq:boundedness} (recalling that $V_k\to\infty$ on $\Omega_k^c$).
\end{proof}

\subsection{Nonlocal Fokker--Planck equation} \label{sec:nonlocal}

Here we consider the extension of the previous linear case to include a nonlinear interaction potential $W$. 

\begin{lem}[$L^2$ energy estimate, case $\phi(s)=s$]\label{lem:l2est}
Let $u$ be a weak solution of problem \eqref{eq:main} with $\phi(s)=s$ in the sense of Definition \ref{def:weaksol1}. 
Let $\bar{u}_0$ be the (constant) mass of $u$ and recall that $\supp (u_0) \subset \Omega$.
The following inequality holds:
\begin{equation}\label{eq:energyw}
\int_{\RR^d} u(T)^{2}e^{V}\d x
+\int_{0}^{T} \int_{\RR^d} e^{-V} |\nabla( e^Vu )|^{2} \, \d x\d t
\leq 
C(\nabla W, \bar{u}_0, T) \int_{\Omega} u_0{}^{2}e^{V}\d x.
\end{equation}
\end{lem}

\begin{proof}
We are going to use $u\exp(V)$
as our test function.
\[
\int_{\RR^d} u(T)^{2}e^{V}\d x
+2\int_{0}^{T}\int_{\RR^d} \left[ e^{V} |\nabla u+u\nabla V|^{2}
+ u (\nabla W* u) \cdot\nabla(e^Vu) \right]\d x\d t
=\int_{\Omega}u_0{}^{2}e^{V}\d x,
\]
and
$$
\int_{0}^{T}\int_{\RR^d} 
 u e^{V/2} (\nabla W* u) \cdot e^{-V/2}\nabla(e^Vu)\d x\d t
 \leq
 \| u e^{V/2} \|_{L^2} \| e^{-V/2}\nabla(e^Vu) \|_{L^2}
 \| \nabla W \|_{L^\infty}  \| u \|_{L^1},
$$
resulting into
\begin{align*}
&	\int_{\RR^d} u(T)^{2}e^{V}\d x
+\int_{0}^{T}\int_{\RR^d} e^{-V} |\nabla( e^Vu )|^{2}\d x\d t \\
&\leq \int_{\Omega}u_0{}^{2}e^{V}\d x
+ \| \nabla W \|_{L^\infty}^2  \| u_0 \|_{L^1}^2 \int_{0}^{T}\int_{\RR^d}  u^{2}e^{V}\d x\d t.
\end{align*}
Using Gronwall's lemma we obtain
\[
\int_{\RR^d} u(T)^{2}e^{V}\d x
+\int_{0}^{T}\int_{\RR^d} e^{-V} |\nabla( e^Vu )|^{2}\d x\d t
\leq \exp\left(
T \| \nabla W \|_{L^\infty}^2  \| u_0 \|_{L^1}^2 
\right)
\int_{\RR^d}u_0{}^{2}e^{V}\d x,
\]
as announced.
\end{proof}

\begin{rem}
Unlike in the linear case, boundedness of (weak) solutions is not so straightforward. It can be shown arguing as in \cite{ladyzhenskaia1988linear}, Chapter 3, Theorem 7.1 or \cite{Stampacchia1965}. Since we do not need boundedness in our analysis, we do not discuss it further.
\end{rem}

\begin{cor}
Under the same assumptions of Lemma \ref{lem:l2est}, we have the following decay estimate for $V=V_k$:
\begin{equation}\label{eq:decay1}
\int_{\Omega_k^c} u(T)^{2}\d x
+\int_{0}^{T}\int_{\Omega_k^c} |\nabla u  + u\nabla V|^{2}
\d x\d t
\leq \exp(-k) 
C\left(T, \nabla W, u_0 \right)
\int_{\Omega}u_0{}^{2}e^{V_0}\d x.
\end{equation}
\end{cor}
\begin{proof}
This is a direct consequence of Lemma \ref{lem:l2est}, recalling that fact that $\supp(u_0)\subseteq \Omega$ and that $V_k\geq k$ on $\Omega_k^c$.
\end{proof}

\subsection{Nonlinear local Fokker--Planck equation} \label{sec:local}

In the case of nonlinear diffusion, we generalize the familiar procedure often used to obtain energy estates and we introduce a new quantity, indicated by $P(u)$, which coincides with $u$ in the linear case.

\begin{lem}[Energy inequality, $W=0$] \label{lem:en2_nl}
Let $u$ be a weak solution the following equation
\begin{align} \label{nonlinear_local}
\begin{aligned}
\de_{t}u &=  \Delta \phi(u) + \dv(u\nabla V), \quad x\in \RR^d, t>0,\\
 u(x,0) & =u_0(x),
 \end{aligned}
\end{align}
where $\phi(s) = s + \sigma(s)$ satisfies Assumptions \ref{assumptions}.
Then:
\begin{equation}\label{eq:energyphi}
\int_{\RR^d} e^{V} Q(u(T))\d x
+
 \int_{\RR^d}
 e^{V}\frac{P(u)}{u}|\nabla \phi(u) + u\nabla V(x)|^{2}
 \d x
\leq
\int_{\Omega} e^{V} Q(u_0) \d x,
\end{equation}
for a.e. $T>0$, with 
$$
P(u)=\exp\left(\int^{u}_1\frac{\phi'(s)}{s}\d s\right),
$$
 and 
$$
Q(u)=\int^{u}_0P(s)\d s.
$$
\end{lem}

\begin{rem}\label{rem:nonlinloc}
Notice that, thanks to Assumption \ref{assumptions} (point 4), we have
$$
\frac{\phi'(s)}{s} = \frac{1 + \sigma'(s)}{s} \geq \frac{1}{s} + \mu s^{a-1},
$$
for any $s\geq 0$ and some given $\mu>0$, $a > 0$.
From the definition of $P$, for $u\geq 0$, we obtain
$$
P(u) = 
\exp\left(\int^{u}_1\frac{1 + \sigma'(s)}{s}\d s\right)
\geq
\exp\left(\log(u) + \mu\int^{u}_1  s^{a-1} \d s\right)
=
u\exp\left( \frac{\mu}{a} (u^a -1 ) \right).
$$
Consequently, $P(u)$ is well defined and, using the lower bound just obtained, we observe that it satisfies
\begin{equation}\label{eq:uexpu}
P(u) \geq c(\mu,a) u \exp(u^a),
\end{equation}
for $c(\mu,a)= \exp(-\frac{\mu}{a})$. It follows that, for any $u\geq 0$,
$$
\frac{P(u)}{u}\geq c(\mu,a)
\quad \text{
 and 
}\quad
\lim_{s\to 0^+} \frac{P(s)}{s} = c(\mu,a).
$$
In addition, notice that since $P(u)\geq c(\mu,a)u$, we also have 
$$
Q(u)\geq \frac{1}{2} c(\mu,a) u^2.
$$
These facts will be useful in the proof of Theorem \ref{thm:main1}, providing useful lower bounds for the left-hand side of inequality \eqref{eq:energyphi}.
\end{rem}

\begin{proof}[Proof of Lemma \ref{lem:en2_nl}.]
First of all, we notice that
\begin{align*}
\nabla ( e^{V}P(u) ) & =  e^{V}P(u)
\nabla\left(V(x)+\int^{u}\frac{\phi'(s)}{s}\d s\right) 
= e^{V}P(u)\left(\nabla V(x)+\frac{\phi'(u)}{u}\nabla u\right)
\\
&
= e^{V}\frac{P(u)}{u}\left(\nabla \phi(u) + u\nabla V(x)\right).
\end{align*}
We test equation \eqref{nonlinear_local} against $e^{V}P(u)$ and obtain
\begin{equation}
	\label{enen_1}
\int_{\RR^d}\de_{t}u
e^{V}P(u)
\d x
=
\int_{\RR^d}
\dv \left[ \nabla\phi(u)
+u\nabla V 
\right]
e^{V}P(u)\d x.
\end{equation}
Considering the left-hand side of \eqref{enen_1}, we deduce
\begin{align*}
\int_{0}^{T}\int_{\RR^d}\de_{t}u\left[e^{V}P(u)\right]\d x\d t & =  \int_{0}^{T}\int_{\RR^d}\de_{t}\left[Q(u)e^{V}\right]\d x\d t =  \int_{\RR^d}Q(u(T))e^{V}\d x-\int_{\RR^d}Q(u_0)e^{V} \, \d x,
\end{align*}
where $Q$ is a primitive of $P$.

Notice that, since $P(u)>0$, $Q(u)$ is increasing and we can choose $Q(0)=0$. 
Using integration by parts in the right-hand side of \eqref{enen_1} we
obtain
\begin{align*}
\int_0^T\int_{\RR^d}
\dv \left[ \nabla\phi(u)
+u\nabla V 
\right]
e^{V}P(u)\d x
& =  -\int_0^T\int_{\RR^d}\left[\nabla\phi(u)
+u\nabla V 
\right]
\cdot\nabla\left[e^{V}P(u)\right]\d x \d t
\\
 & = 
 -  \int_0^T\int_{\RR^d}
 e^{V}\frac{P(u)}{u}|\nabla \phi(u) + u\nabla V(x)|^{2}
 \d x \d t.
\end{align*}
Collecting terms we finally conclude that
\begin{equation}
\int_{\RR^d} e^{V} Q(u(T))\d x \d t
+
\int_0^T \int_{\RR^d}
 e^{V}\frac{P(u)}{u}|\nabla \phi(u) + u\nabla V(x)|^{2}
 \d x
\leq
\int_{\Omega} e^{V} Q(u_0) \d x \d t
.
\end{equation}

\end{proof}

The following Corollary will help us pass to the limit in the proof of Theorem 1, case $W=0$.
\begin{cor}
Under the same assumptions of Lemma \ref{lem:en2_nl}, we have the estimate
\begin{equation}\label{eq:beoundW}
\int_{\RR^d} \frac{1}{2} e^{V} 
u(T)^2
\d x
+
\int_0^T \int_{\RR^d}
 e^{V}
 |\nabla \phi(u) + u\nabla V(x)|^{2}
 \d x \d t
\leq
\int_{\Omega} \frac{1}{c(\mu,a)} e^{V} Q(u_0) \d x.
\end{equation}
Additionally, for $V=V_k$, the following decay estimate holds:
\begin{equation}\label{eq:decayW}
\int_{\Omega_k^c} 
u_k(T)^2
\d x
+
\int_0^T \int_{\Omega_k^c}
 |\nabla \phi(u_k) + u_k\nabla V_k(x)|^{2}
 \d x \d t
\leq
\exp(-k)
\int_{\Omega} \frac{2 e^{V_0}}{c(\mu,a)} Q(u_0) \d x.
\end{equation}
\end{cor}
\begin{proof}
Combining estimate \eqref{eq:energyphi} and the lower bounds for $P$ and $Q$ from Remark \ref{rem:nonlinloc} we have
$$
\int_{\RR^d} e^{V} 
\frac{1}{2} c(\mu,a) u(T)^2
\d x
+
\int_0^T \int_{\RR^d}
 e^{V}
 c(\mu,a) \exp(u^a)
 |\nabla \phi(u) + u\nabla V(x)|^{2}
 \d x \d t
\leq
\int_{\Omega} e^{V} Q(u_0) \d x,
$$
which directly implies \eqref{eq:beoundW}. This implies the decay recalling that $V_k\geq k$ on $\Omega_k^c$.
\end{proof}

We are now going to exploit the properties of the function $Q$ in order to obtain a bound in $L^p$. First, we need the following simple result.

\begin{lem}\label{lem:optimisation}
Let $p>1$, $a>0$. 
The following inequality holds for any $s>0$:
\begin{equation}\label{eq:optimisation}
s^p \leq \theta_{p,a} s \exp( s^a ), 
\qquad \text{ where } \qquad
\theta_{p,a} = \left(\frac{p-1}{ea}\right)^{\frac{p-1}{a}}.
\end{equation}
\end{lem}
\begin{proof}
Maximizing the function
$
{s^{p-1}}{\exp( -s )}
$
we obtain the optimal value for $\theta_{p,a}$, which is attained at $s = p-1$ and hence $\theta_{p,a} = 
\left(\frac{p-1}{ea}\right)^{\frac{p-1}{a}}$.
\end{proof}

\begin{cor}[$L^p$ estimate]\label{cor:lpest}
Under the hypotheses of Lemma \ref{lem:en2_nl}, for any $1\leq p < \infty$ and $\theta_{p,a}$ as in Lemma \ref{lem:optimisation}, it holds
\begin{equation}
\int_{\RR^d} e^{V} u(t)^p \d x 
\leq 
\theta_{p,a} \int_{\Omega} e^{V} Q(u_0) \d x, \quad \text{ for a.e. } t\in [0,T].
\end{equation}
\end{cor}
\begin{proof}
The result follows combining the inequalities \eqref{eq:energyphi}, \eqref{eq:uexpu} and \eqref{eq:optimisation}.
\end{proof}

\begin{rem}
Notice that the term $\dv(u\nabla V)$ can be generalized to the form $\dv(\alpha(u)\nabla V)$, where 
$\alpha$ is such that $\alpha(s) = s(1+\rho(s))$ and there exists a constant $\mu \in [0,1)$ such that
\begin{equation}\label{eq:growth}
\mu \sigma'(s) \leq \rho(s) \leq  \sigma'(s).
\end{equation}
In this case we can define $P$ as follows: $$
P(u)=\exp\left(\int^{u}_1\frac{\phi'(s)}{\alpha(s)}\d s\right),
$$
and we can obtain a result analogous to Lemma \ref{lem:en2_nl}.
\end{rem}

\subsection{Proof of Theorem \ref{thm:main1}} \label{sec:linear_passage}

We are going to present the two cases, $W=0$ or $\phi(s)=s$, separately.

\begin{rem}\label{rem:bicont}
Suppose that the function $ue^{V/2}$ belongs to $L^2(0,T;H^1(\Omega))$, then $u$ belongs to the same space. Indeed, $u$ belongs to $L^2(0,T;L^2(\Omega))$ since $V \in \WW^{1,\infty}(\Omega)$ and we have $\nabla(ue^{V/2}) = e^{V/2}(\nabla u + u\nabla V/2)$, so that 
$\nabla u = e^{-V/2} \nabla(ue^{V/2}) - u\nabla V/2$, where both terms on the right-hand side are square-integrable.
\end{rem}

\begin{proof}[Proof of Theorem \ref{thm:main1}, case $\phi(s) = s$.]

The weak formulation \eqref{eq:weakform_def} gives us
\begin{equation}\label{eq:weakform}
\int_{\RR^d} u_k(t)\eta(t)\d x \big|_{t=0}^{t=T} - 
\int_0^T\int_{\RR^d}    
\left[
(\nabla u_k 
+  u_k \nabla 
   ( V_k
  + W*u_k  )
 )
 \cdot \nabla \eta  
 -u_k \de_t\eta 
\right] 
 \,
\d x \d t
=0.
\end{equation}
for all test functions $\eta\in H^1(\RR^d)$. 
We notice that, from the energy inequality \eqref{eq:energyw}, the function $u_ke^{\frac{V_k}{2}}$ is bounded in $L^2(0,T;H^1(\RR^d))\cap  L^\infty(0,T;L^2(\RR^d)) $.
Since the term $e^{\frac{V_k}{2}}$ is bounded from below, we also have that $u_k$ belongs to the same space (see Remark \ref{rem:bicont}).

We divide $\RR^d$ into three parts, namely $\Omega$, $\Omega_k^c$ and $\Omega_k\setminus \Omega$.
Consequently, we split \eqref{eq:weakform} as $I_{\Omega} + I_{\Omega_k^c} +  I_{\Omega_k\setminus \Omega} + J_{\Omega} + J_{\Omega_k^c} +  J_{\Omega_k\setminus \Omega}$, where
\begin{align*}
I_{A} &= 
\int_0^T\int_{A} 
\left(  
\nabla u_k
+  u_k \nabla V_k 
+  u_k \nabla W*u_k 
\right)\cdot \nabla \eta  
\d x \d t,
\\
J_A 
& = 
\int_A [u_k(T)\eta(T)
- u_0 \eta(0)]\d x - 
\int_0^T\int_A    
u_k \de_t\eta \,
\d x \d t.
\end{align*}
We want to show that all terms but $I_\Omega$ and $J_\Omega$ vanish in the limit $k\to \infty$. Then $I_\Omega + J_\Omega$ will characterize the limit problem defined in $\Omega$. 
\begin{itemize}
	\item {\boldmath $I_\Omega$.}
Restricting our attention to $I_\Omega$, from \eqref{eq:energyw} we obtain
\begin{equation}
\int_{\Omega} u_k(T)^{2}e^{V_k}\d x
+\int_{0}^{T} \int_{\Omega} e^{-V_k} |\nabla( e^{V_k} u_k )|^{2} \, \d x\d t
\leq 
C \int_{\Omega} u_0{}^{2}e^{V}\d x.
\end{equation}
This implies that
$$
u_k\in L^\infty(0,T;L^2(\Omega)), \qquad \nabla (e^{V_0}u_k)\in L^2(0,T;L^2(\Omega)),
$$ 
uniformly in $k$. 
Since $V_0$ is bounded and sufficiently smooth, thanks to Remark \ref{rem:bicont} we obtain $u_k\in L^2(0,T;H^1(\Omega))$ uniformly in $k$, hence we can extract a subsequence $u_{k_n}$ that converges weakly in $L^2(0,T;H^1(\Omega))$ to a limit denoted by $u$.

We now estimate the time derivative $\de_t u_k$ in order to prove a strong convergence result for $u_k$.
Notice that, rewriting \eqref{eq:weakform_def}, we obtain the bound
\begin{align*}
\langle \de_t u_k , \eta \rangle_{H^1(\Omega)', H^1(\Omega)}
&=
\int_{\Omega} [u_k(T)\eta(T)
- u_k(0) \eta(0)]\d x - 
\int_0^T\int_{\RR^d}    
u_k \de_t\eta \,
\d x \d t,
\\
&=
\int_{\Omega_T}    
\left[
(\nabla u_k 
+  u_k \nabla V_0
+  u_k \nabla W*u_k
 )
 \cdot \nabla \eta  
\right] 
 \,
\d x \d t
\\
&\leq 
C\left(\nabla V, \nabla W, \int_\Omega u_0 \d x\right)
 \nm{u_k}_{L^2(0,T;H^1(\Omega))}  \nm{\eta}_{L^2(0,T;H^1(\Omega))},
\end{align*}
where $C$ is a constant independent of $k$. 
We also know that $\nm{u_k}_{L^2(0,T;H^1(\Omega))}$ is bounded uniformly with respect to $k$, we deduce that $\de_t u_k \in  L^2(0,T;H^1(\RR^d)')$.
We can now apply Aubin--Lions Lemma and deduce that,
by compactness, $u_{k_n}$ converges strongly (up to a subsequence) in $L^2(\Omega_T)$ as well. To simplify the notation, in what follows we write $u_{k}$ instead of $u_{k_n}$.
In particular, thanks to the strong convergence of $u_k$ in $L^2$, we have
$$
u_k \nabla V_k \to u \nabla V_0 \text{ in } L^2(\Omega_T),
$$
and
$$
u_k \nabla (W * u_k ) \to u \nabla (W * u ) \text{ in } L^1(\Omega_T).
$$
In fact, for $k\to\infty$,
$$
\nm{ \nabla W * (u_k - u) }_{L^2(Q_T)} \leq \nm{\nabla W}_{L^1(\Omega)} \nm{ (u_k - u) }_{L^2(Q_T)} \to 0.
$$

Thus we have obtained
$$
I_\Omega \to 
\int_{\Omega}   \left[
\left(  
\nabla u 
+  u \nabla V 
+  u \nabla W*u 
\right)\cdot \nabla \eta  
\right]
\d x,
$$
as $k\to\infty$.

\item {\boldmath$I_{\Omega_k^c}$.}
Considering $I_{\Omega_k^c}$ and using again \eqref{eq:energyw}, we have
\[
\int_{\Omega_k^c} u_k(T)^{2}e^{k}\d x
+\int_{0}^{T}\int_{\Omega_k^c} e^{k}|\nabla u_k |^{2}\d x\d t\leq
C\int_\Omega u_0{}^{2}e^{V_0}\d x,
\]
and therefore, for $k\to\infty$, we have that $I_{\Omega_k^c} \to 0$ since, by the decay estimate \eqref{eq:decay1},
\begin{equation}\label{eq:convergeneto0}
\nm{u_k}_{  L^2(0,T;H^1(\Omega_k^c))\cap  L^\infty(0,T;L^2(\Omega_k^c)) } \to 0.
\end{equation}

\item {\boldmath$I_{\Omega_k \setminus \Omega}$.}
It remains to be checked that $I_{\Omega_k \setminus \Omega}$ also vanishes.
Once more, from the energy identity \eqref{eq:energyw}, we obtain
$$
\int_{0}^{T}\int_{\Omega_k \setminus\Omega} e^{\psi_k}|\nabla u_k + u_k\nabla (\psi_k + W *u_k )|^{2}\d x\d t 
\leq C \int_{\Omega} u_0{}^{2}e^{V_0}\d x.
$$
Since $\exp(V_k)\geq 1$, we deduce that
\begin{align*}
\int_0^T\int_{\Omega_k \setminus\Omega}    &
(\nabla u_k 
+  u_k \nabla (\psi_k + W *u_k )) \cdot \nabla \eta  \,
\d x \d t
\\
&
\leq
\left( 
\int_0^T\int_{\Omega_k \setminus\Omega}    
|
\nabla u_k 
+  u_k \nabla (\psi_k + W *u_k )
|^2 \,
\d x \d t
\right)^{1/2}
\left( 
\int_0^T\int_{\Omega_k \setminus\Omega}   
 |\nabla \eta  |^2 \,
\d x \d t
\right)^{1/2}
\\
&
\leq
\left( 
C 
\int_{\Omega} u_0{}^{2}e^{V_0}\d x
\right)^{1/2}
\left( 
\int_0^T\int_{\Omega_k \setminus\Omega}   
 |\nabla \eta  |^2
\d x \d t
\right)^{1/2}
\to 0 \qquad \text{ as } k\to\infty.
\end{align*}
Indeed, the first term in the last line is bounded and the second one vanishes since $\nabla \eta \in L^2(Q_T)$ and $|\Omega_k \setminus\Omega| \to 0$ as $k\to\infty$.

\item {\boldmath$J_{\Omega}$.}
The sequence $u_k$ converges weakly in $H^1(\Omega)$ to a limit $u$, hence
\begin{equation}\label{eq:convergence2}
J_\Omega \to 
\int_{\Omega} [u(T)\eta(T)
- u_0 \eta(0)]\d x - 
\int_0^T\int_{\Omega}    
u \de_t\eta \,
\d x \d t
\quad \text{ as } k\to\infty,
\end{equation}

\item {\boldmath$J_{\Omega_k^c}$.}
Notice that $u_0 = 0$ in $\Omega_k^c$. The remaining terms in $J_{\Omega_k^c}$ vanish thanks to \eqref{eq:convergeneto0}.

\item {\boldmath$J_{\Omega_k \setminus \Omega}$.}
The integral $J_{\Omega_k \setminus \Omega}$ goes to zero 
because the integrand is uniformly bounded in $L^1$ (thanks to the conservation of mass) and $|\Omega_k \setminus\Omega| \to 0$. 
\end{itemize}

The weak formulation we obtain in the limit is the following:
$$
\int_{\Omega} [u(T)\eta(T)
- u(0) \eta(0)]\d x - 
\int_0^T\int_{\Omega}    
\left[
(\nabla \phi(u) 
+  u \nabla V_0
+  u \nabla W*u
 )
 \cdot \nabla \eta  
 -u \de_t\eta 
\right] 
 \,
\d x \d t
=0.
$$
Notice that the initial datum is still satisfied in the $L^2$ sense and that since the test function can be any element of $H^1(\Omega)$ this implies that no-flux conditions on $\de\Omega$ are implicitly enforced.
It is easy to see that the initial datum is satisfied in the $L^2$ sense.
\end{proof}

\begin{proof}[Proof of Theorem \ref{thm:main1}, case $W = 0$.] 

We now consider the case $W = 0$ and $\phi$ generic.
We can repeat most of the steps above using the energy inequality \eqref{eq:energyphi} instead of \eqref{eq:energyw}, therefore we will focus only on the main differences.
First of all, we observe that the quotient $\frac{P(u)}{u} \geq c(\mu,a)$ for any $u\geq 0$, see Remark \ref{rem:nonlinloc}.
In particular, just like in the previous proof, we can split the weak formulation into terms denoted by $I_A$ and $J_A$.
The terms denoted by $J_A$ above are unchanged and are treated in the same way.
The terms $I_{\Omega_k^c}$ and $I_{\Omega_k \setminus \Omega}$ vanish in an analogous way since, also in this case we have exponential decay with respect to $k$, as shown in \eqref{eq:decayW}.

Considering $I_{\Omega_k}$, from \eqref{eq:energyphi}, Remark \ref{rem:nonlinloc} and recalling that $Q(0) = 0$, it follows that
\begin{equation}\label{eq:energyestnonlin}
\int_{\RR^d} e^{V_k} \frac{1}{2}u_k(T)^2\d x
+
 \int_{\RR^d}
 e^{V_k}|\nabla \phi(u_k) + u_k\nabla V_k(x)|^{2}
 \d x
\leq
\int_{\Omega} e^{V_0} Q(u_0) \d x.
\end{equation}
We now proceed as in the proof of the case $\phi(s) = s$, in the sense that, thanks to inequalities \eqref{eq:decayW} and \eqref{eq:energyestnonlin}, we have that, up to a subsequence, $u_k$ converges to $0$ outside of $\Omega$ and it converges strongly in $L^2(Q_T)$ (hence almost everywhere) and weakly in $L^2(0,T;H^1(\Omega))$ to a function $u$ satisfying 
the limit weak formulation \eqref{eq:limitweakform1}.
In particular, let us focus on the nonlinear diffusion term.
We know that $u_k\to u$ a.e. in $\Omega$, $\phi'$ is continuous and it has polynomial growth by assumption \eqref{eq:growth0}. Furthermore, thanks to Corollary \ref{cor:lpest} (with $V=V_k$) we know that $u_k\in L^\infty(0,T;L^p(\Omega))$, for any $p\geq 1$. Combining such bounds and a.e. convergence, we deduce that $\phi'(u_k)\to\phi'(u)$ strongly in $L^2(\Omega_T)$.
Additionally, we have $\nabla u_k \rightharpoonup \nabla u$ weakly in $L^2(\Omega_T)$ as $k\to\infty$, thus the product $\phi'(u_k)\nabla u_k$ converges to $\phi'(u)\nabla u$ weakly in $L^1(\Omega_T)$ and we have
$$
\int_{\Omega}
\nabla \phi(u_k)
\cdot \nabla \eta  
\d x=
\int_{\Omega}
\phi'(u_k)\nabla u_k
\cdot \nabla \eta  
\d x
\to
\int_{\Omega}
\phi'(u)\nabla u
\cdot \nabla \eta  
\d x
\quad \text{ as }
k\to\infty,
$$
The remaining term $I_\Omega$ is treated just like in the previous case.
\end{proof}

\begin{rem}
We do not treat the more general case with nonlinear $\phi$ as well as $W\ne 0$ in the present section since it is not clear how to obtain a suitable energy estimate that is uniform with respect to $k$. However, the general equation can be studied using entropy techniques as shown in the next section.
\end{rem}

\section{Analysis via free energy estimates}\label{sec:entropy}

Let us consider the full problem with nonlinear diffusion and nonlocal interaction terms: 
\begin{align}\label{eq:uk}
	\begin{aligned}
	\de_{t}u &=  \dv\left[ \nabla\phi(u) + u\nabla V + u\nabla (W*u)  \right], & & \\
	u(x,0) & = u_0(x), & & x\in \mathbb R^d, t>0.
	\end{aligned}
\end{align}
We consider a solution $u_k$ to \eqref{eq:uk} with $V = V_k$ given in Definition \ref{defi:vk-ent}.
We have to prove that the sequence $u_k$ converges to a limit function $u$ solving problem \eqref{eq:main} in $\Omega$. 
The main steps involved are:
\begin{enumerate}
	\item finding estimates for $u_k$ independent of $k$,
	\item showing that $u_k\to 0$ outside $\Omega$,
	\item passing to the limit in the weak formulation of \eqref{eq:uk}.
\end{enumerate}
Concerning well-posedness of problem \eqref{eq:uk}, the following result has been proven in \cite{ambrosio2008gradient}, Theorem 11.2.8.

\begin{theor}[Existence and uniqueness of solutions]
\label{thmAGS}
Let Assumption \ref{assumptions-ent} hold and suppose that $V$ and $W$ are strictly convex. For every $u_0\in \mathcal{P}_2(\RR^d)$ there exists a unique distributional solution $u\in \mathcal{P}_2(\RR^d)$ of \eqref{eq:uk} in $\RR^d$ satisfying 
$$
u(t,\cdot) \to u_0 \text{ in } \mathcal{P}_2(\RR^d)
\; \text{ as } t\to 0 ,
$$
$\phi(u) \in L^1_{loc}(0,\infty;\WW^{1,1}_{loc}(\RR^d))$, and, for  all $T>0$,
\begin{equation}\label{dissipation}
\int_{Q_T} u \left| 
\frac{1}{u}\nabla\phi(u) + \nabla V + \nabla W * u 
\right|^2  \d x \d t
 < \infty.
\end{equation}
\end{theor}

\begin{rem}[Existence theory]
Theorem \ref{thmAGS} holds even if $\phi=\sigma$ and the initial datum is a measure in $\mathcal{P}_2(\RR^d)$. However, given the assumptions in this paper, we do not have to consider measure-valued solutions.
For further results concerning existence, uniqueness and asymptotic properties of gradient flow/free energy solutions, we refer to \cite{ambrosio2008gradient,carrillo2003kinetic}.
\end{rem}

\subsection{Uniform bounds} \label{sec:step2}

\begin{lem} \label{lem:entropy_ineq0}
Let $u$ be a weak solution of \eqref{eq:uk} and let Assumption \ref{assumptions-ent} hold. Problem \eqref{eq:uk} is a gradient flow (in the Wasserstein sense) for the following associated free energy functional
$$
E[u(t)] = \int_{\RR^d} \left[ u\log u + \Xi(u) +  uV + \frac{1}{2}u(W*u) \right] \d x,  
$$
where 
$$
\Xi(u)=\int^{u}_0\xi(s)\d s, \qquad 
\xi(s)=\int^{s}_1\frac{\sigma'(r)}{r}\d r.
$$
More specifically, $\forall t \geq 0$ we have
\begin{equation}\label{eq:entropydecay}
E[u(t)]+\int_0^t D[u(\tau)] \d \tau = E[u_0],
\end{equation}
where
\begin{equation}
D[u(t)] = \int_{\mathbb R^d}
u \left| 
\nabla(\log(u) + \xi(u) + V + W*u )
\right|^2
\d x.
\end{equation}
Moreover, for any $t\geq 0$, we have that $u\geq 0$ and that
$$
\int_{\RR^d} u \d x = \int_{\RR^d} u_0 \d x.
$$
\end{lem}
\begin{proof}
Identity \eqref{eq:entropydecay} is obtained differentiating $E[u(t)]$ with respect to time and noticing that equation \eqref{eq:main} can be rewritten as follows:
$$
\de_{t}u =  \dv\left[ u \nabla ( \log(u) + \xi(u) + V + W*u) \right].
$$
See for example Theorem 11.2.1, Chapter 11 in \cite{ambrosio2008gradient} for further details.
\end{proof}

\begin{rem} Notice that, for $V=V_k$, the estimates for the $L^1$ norm and for the free energy in Lemma \ref{lem:entropy_ineq0} are uniform with respect to $k$ and $t$.
\end{rem}

We now state a useful technical Lemma, for its proof we refer the reader to \cite{blanchet2012functional}.
Recall that $(f(x))_- = \max\{0,-f(x)\}$.

\begin{lem}[Carleman estimate \cite{blanchet2012functional}]\label{carleman}
Consider two functions $\rho\in L^1_+(\RR^d)$ and $\gamma:\RR^d\to\RR$, with $\gamma(x)\geq 0$,  $e^{-\gamma} \in L^1(\RR^d)$, such that the moment $\int_{\RR^d} \gamma(x) \rho(x) \d x $ is bounded.
Then
\begin{equation}
\int_{\RR^d} \rho(x)(\log \rho(x))_- \d x \leq \int_{\RR^d} \gamma(x) \rho(x) \d x + \frac{1}{e} \int_{\RR^d} e^{-\gamma(x)} \d x.
\end{equation}
\end{lem}
Thanks to Lemma \ref{carleman}, the negative part of $u_k\log u_k$ is bounded and Lemma \ref{lem:uk_to_0} (below) provides an estimate for $u_k$ that involves only positive terms on the left-hand side.

\begin{lem} \label{lem:uk_to_0}
Let $u_k$ be the solution of \eqref{eq:uk} where $V= V_k$ is defined in \eqref{potential_Vk}. 
Then the following inequality holds
\begin{equation}\label{eq:uniformbound}
\int_{\RR^d} \left[ 
u_k|\log u_k| 
+ u_kV_k 
+ u_k(W*u_k)
\right] \d x 
\leq 
C_0,
\end{equation}
where $C_0$ is a constant depending on $\Omega$, $u_0$ and $V_0$ only, given by 
$$
C_0 = 2E(u_0) 
+ \frac{2}{e} 
\left( 
 \int_{\Omega} e^{-\frac{1}{2}V_0} \d x 
+ \eps_k
\right),
$$
with $\eps_k = \int_{\RR^d\setminus\Omega} e^{-\frac{1}{2}V_k}\d x \to 0$ as $k\to\infty$.
Additionally, we have that
\begin{equation}\label{eq:decay}
\int_{\RR^d\setminus \Omega_k}  
 u_k  \d x
\leq \frac{2C_0}{k},
\end{equation}
and $u_k \to 0$ in $L^1(\RR^d\setminus \Omega_k)$ as $k\to\infty$, uniformly in $t\in [0,T]$.
\end{lem}

\begin{proof}
Let us first observe that, from \eqref{eq:growth0}, we have
$$
\frac{\phi'(s)}{s} =  \frac{1}{s} + \frac{\sigma'(s)}{s} \geq \frac{1}{s} + \mu s^{a-1} > 0,
$$
for $s>0$. It follows that all the terms involving $\frac{\sigma'(s)}{s}$ appearing in the free energy functional $E$ are non-negative.
Before using the entropy inequality of Lemma \ref{lem:entropy_ineq0}, we have to ensure that the term involving $u_k\log(u_k)$ is non-negative.
We have
$$
\int_{\RR^d} \left[(u_k\log u_k ) 
+ u_kV_k 
+ \frac{1}{2}u_k(W*u_k)\right] \d x \leq 
E(u_0).
$$
In order to estimate the negative part of $u_k\log u_k$ we will apply Lemma \ref{carleman} with $\rho = u_k$ and $\gamma = \frac{1}{2} V_k$ (notice that $\int_{\RR^d} V_k u_k \d x$ is bounded but we do not know that the bound is uniform in $k$ at this stage). More specifically, we have
$$
\int_{\RR^d} u_k(\log u_k)_- \d x \leq 
\frac{1}{2} \int_{\RR^d} V_k u_k \d x 
+ \frac{1}{e} \int_{\RR^d} e^{-\frac{1}{2}V_k} \d x.
$$
This implies that
$$
\int_{\RR^d} \left[ 
u_k|\log u_k| 
+ \frac{1}{2} u_kV_k 
+ \frac{1}{2}u_k((W)*u_k)
\right] \d x 
\leq 
E(u_0) 
+ \frac{1}{e} \int_{\RR^d} e^{-\frac{1}{2}V_k} \d x,
$$
and, in turn,
$$
\frac{1}{2}\int_{\RR^d\setminus \Omega_k} 
u_kV_k 
 \d x 
\leq 
E(u_0) 
+ \frac{1}{e} \int_{\RR^d} e^{-\frac{1}{2}V_k} \d x.
$$
We subsequently obtain
$$
0\leq \int_{\RR^d\setminus \Omega_k}  
 u_k  \d x
\leq 2\frac{C_0}{k} \to 0 \quad \text{ as } k\to\infty.
$$
\end{proof}

\begin{rem}
Notice that the assumption $W\geq 0$ is not restrictive and the same argument applies if $W$ has any lower bound of the type $W\geq - q$, in particular
$$
\int_{\RR^d} u(W*u) \d x = \int_{\RR^d} u((W+q)*u) \d x - q\left(\int_{\RR^d} u \d x\right)^2.
$$
\end{rem}

The following estimate will be used extensively in the next subsection.
\begin{cor} \label{lem:bound-step3}
Let $u_k$ be the solution to \eqref{eq:uk} when $V= V_k$ is given by \eqref{potential_Vk-ent}. Then the following estimate (uniform in $k$) holds
\begin{align} \label{eq:entropybound}
 \int_{\RR^d} \left[ 
u_k|\log u_k|  + \Xi(u_k)
+ u_kV_k 
+ u_k((W+q)*u_k)
\right] \d x 
& 
\nonumber \\ 
  +
\int_0^T\int_{\RR^d} 
u_k \left| 
\nabla (\log u_k + \xi(u_k) + V_k + W*u_k)
\right|^2
\d x \d t
&\leq 
C_0,
\end{align}
where $C_0$ is specified in Lemma \ref{lem:uk_to_0}.
\end{cor}

\begin{proof}
It is a direct consequence of \eqref{eq:entropydecay} and  \eqref{eq:uniformbound}. 
\end{proof}

\subsection{Passage to the limit and proof of Theorem \ref{thm:main2}.} \label{sec:step5}
Before proving Theorem \ref{thm:main2}, we state a simple interpolation lemma.
\begin{lem}\label{lem:interp}
Consider a function $f:\Omega\to\RR$ such that
$$
f\in L^{p_0}(0,T;L^{q_0}(\Omega)) \cap L^{p_1}(0,T;L^{q_1}(\Omega)),
$$
where $1 \leq p_i, q_i \leq \infty$, for $i=0,1$.
Then
$$
f\in L^{p_\theta}(0,T;L^{q_\theta}(\Omega)),
$$
for any $\theta\in (0,1)$ and the following relations hold:
$$
\frac{1}{p_\theta} = \frac{1-\theta}{p_0} + \frac{\theta}{p_1},
\quad
\frac{1}{q_\theta} = \frac{1-\theta}{q_0} + \frac{\theta}{q_1}.
$$
\end{lem}
\begin{proof}
The proof consists in applying H\"older's inequality to $f^{1-\theta}f^\theta = f$ twice, first with respect to $x\in\Omega$, then with respect to $t\in[0,T]$. Namely we have:
\begin{align*}
\left(
\int_0^T \nm{f}_{L^{q_\theta}(\Omega)}^{p_\theta} \d t
\right)^{\frac{1}{p_\theta}}
&\leq
\left(
\int_0^T \nm{f}_{L^{q_0}(\Omega)}^{1-\theta}
\nm{f}_{L^{q_1}(\Omega)}^{\theta}\d t
\right)^{\frac{1}{p_\theta}}
\\
&\leq
\left(
\int_0^T \nm{f}_{L^{q_0}(\Omega)}^{p_0} \d t
\right)^{\frac{1-\theta}{p_0}}
\left(
\int_0^T \nm{f}_{L^{q_1}(\Omega)}^{p_1} \d t
\right)^{\frac{\theta}{p_1}}.
\end{align*}
\end{proof}

Similarly to the $L^2$ case, we consider the weak formulation  \eqref{eq:weakform} and we divide $\RR^d$ into three parts, namely $\Omega$, $\Omega_k^c$ and $\Omega_k\setminus \Omega$, and hence we split the weak formulation as follows:
\begin{align} 
\nonumber
\int_{\RR^d}  u_k\eta\Big|_0^T \d x
&+ 
\int_0^T\int_{\RR^d}   \left[
\left(  
\nabla \phi(u_k) 
+  u_k \nabla V_k 
+  u_k \nabla W*u_k 
\right)\cdot \nabla \eta  
- u_k \de_t \eta 
\right]
\d x \d t
\\ \label{eq:weaksplit}
&
= I_\Omega + I_{\Omega_k \setminus \Omega} + I_{\Omega_k^c} + J_\Omega + J_{\Omega_k \setminus \Omega} + J_{\Omega_k^c},
\end{align}
for any test function $\eta\in C_0^\infty(Q_T)$. Namely we have defined
\begin{align*}
I_{A} &= 
\int_0^T\int_{A}   
\left(  
\nabla \phi(u_k) 
+  u_k \nabla V_k 
+  u_k \nabla W*u_k 
\right)\cdot \nabla \eta  \,
\d x \d t,\\
J_A & =  - 
\int_0^T\int_{A}    
u_k \de_t\eta \,
\d x \d t.
\end{align*}

\begin{proof}[Proof of Theorem \ref{thm:main2}, non-degenerate case]
We are going to show that all terms except $I_\Omega$ and $J_\Omega$ vanish in the limit $k\to \infty$. It follows that the term $I_\Omega + J_\Omega$ characterizes the limit problem defined in $\Omega$. 

\begin{itemize}
	\item {\boldmath$I_\Omega$.}
First, we restrict our attention to $I_\Omega$. We notice that restricting  \eqref{eq:entropybound} in $\Omega$ we have
\begin{equation} \label{bound0}
\int_0^T\int_{\Omega} 
u_k \left| 
\nabla (\log u_k + \xi(u_k) + V_0 + W*u_k)
\right|^2
\d x \d t
\leq 
C_0.
\end{equation}
Hence, we have that 
$\sqrt{u_k} \,
\nabla ( \log u_k
+ \xi(u_k)
+  V_0 
+  W*u_k )
 \in L^2(\Omega_T)$. We now proceed to show that the last two terms in this expression are bounded in $L^2(\Omega_T)$. We have
\begin{equation} \label{bound1}
\int_0^T\int_{\Omega} 
u_k 
|
\nabla V_0
|^2
\d x \d t
\leq 
T \nm{\nabla V_0}_{L^\infty(\Omega)}^2
\int_\Omega u_0 \d x,
\end{equation}
and
\begin{equation} \label{bound2}
\int_0^T\int_{\Omega} 
u_k 
|
\nabla W*u_k
|^2
\d x \d t
\leq T \nm{\nabla W}_{L^\infty(\Omega)}
\left( \int_\Omega u_0 \d x \right)^2.
\end{equation}
We rewrite the diffusion terms as
\begin{equation} \label{diffterms}
\int_0^T\int_{\Omega} 
u_k \left| 
\nabla (\log u_k + \xi(u_k))
\right|^2
\d x \d t
=
\int_0^T\int_{\Omega} 
(1+\sigma'(u_k))
| \nabla (\sqrt{ u_k })
|^2
\d x \d t.
\end{equation}
Recalling that $\sigma' \geq 0$ and combining \eqref{bound0}, \eqref{bound1} and \eqref{bound2}, from \eqref{diffterms} it follows that that $\nabla( \sqrt{u_k} )\in L^2(\Omega_T)$ uniformly in $k$. 
We know that the mass of $u_k$ is constant and $u_k$ is non-negative, therefore $\sqrt{u_k}\in L^\infty(0,T;L^2(\Omega))$. Additionally, we know that $\sqrt{u_k}\in L^2(0,T;H^1(\Omega))$, which implies $\sqrt{u_k}\in L^2(0,T;L^{2^*}(\Omega))$. By Lemma \ref{lem:interp}, we deduce that $\sqrt{u_k}\in L^r(\Omega_T)$ for any $r\in ( 2 , 2(1+2/d) )$, when $d\geq 2$ (if $d=1$ we choose $r=4$).
We can now estimate $\de_tu_k$ as follows:
\begin{align*}
\langle \de_tu_k, \eta \rangle 
&=
\int_{\Omega_T}
\sqrt{u_k}\sqrt{u_k}\left[  
\nabla (\log u_k + \xi(u_k) + V_0 + W*u_k)
\right]\cdot \nabla \eta  
\d x \d t
\\
&\leq
C_0^\frac{1}{2} \nm{\sqrt{u_k}}_{L^r(\Omega_T)} \nm{\nabla\eta}_{L^s(\Omega_T)},
\end{align*}
where $r$ and $s$ satisfy $\frac{1}{r}+\frac{1}{s}=\frac{1}{2}$.
This implies that $\de_tu_k$ is bounded in $L^{s}(0,T;W^{1,s}(\Omega))'$ uniformly with respect to $k$.
In order to obtain compactness for the sequence $\sqrt{u_k}$, we need a modified version of Aubin--Lions Lemma. In particular, if $a>1$ we apply Theorem 1 of \cite{moussa2016some} (with $\Phi(\cdot) = \sqrt{|\cdot|}$, so that $\sqrt{|\cdot|}\in W^{1,1}(\RR)$ and $\mathrm{meas}(\{  \sqrt{|\cdot|}>\delta  \}) \to 0$ in $L^1$ as $\delta\to 0$), otherwise we apply Theorem 3 of \cite{chen2014note} (with $m=\frac{1}{2}$).
We deduce that the sequence $\sqrt{u_k}$ is compact in $L^2(\Omega_T)$.
Therefore we can extract a subsequence (still denoted by $u_k$) that converges strongly in the same space, that is,
\begin{equation}\label{l1strong}
\sqrt{u_k} \to \sqrt{u} \quad \text{ in } L^2(\Omega_T).
\end{equation}

We now rewrite $I_\Omega$ as follows
\begin{align} \label{IntOmega}
I_\Omega & = \int_0^T\int_{\Omega}  
\sqrt{u_k} \sqrt{u_k}\, 
\nabla ( \log u_k
+ \xi(u_k)
+  V_k 
+  W*u_k )
\cdot \nabla \eta \,  
\d x \d t,
\end{align}
and we notice that the integrand is the product of a strongly converging sequence and a weakly converging sequence in $L^2(\Omega_T)$, in particular
$\sqrt{u_k}\to \sqrt{u}$ and 
$ F_k \rightharpoonup F\in L^2(\Omega_T)$, where $F_k = \sqrt{u_k}\nabla ( \log u_k
+ \xi(u_k)
+  V_k 
+  W*u_k )$. 
This means that
$$
\int_0^T\int_{\Omega}  
\sqrt{u_k} F_k
\cdot \nabla \eta \,  
\d x \d t 
\to 
\int_0^T\int_{\Omega}  
\sqrt{u} F
\cdot \nabla \eta \,  
\d x \d t .
$$
Furthermore, again combining \eqref{bound0}, \eqref{bound1} and \eqref{bound2}, it follows that $F = \nabla \phi(u) 
+  u \nabla V_0 
+  u \nabla W*u $ and
\begin{equation}\label{eq:convergence1}
I_{\Omega} \to 
\int_0^T\int_{\Omega}   \left[
\left(  
\nabla \phi(u) 
+  u \nabla V_0 
+  u \nabla W*u 
\right)\cdot \nabla \eta  
\right]
\d x \d t,
\quad \text{ as } k\to\infty.
\end{equation}

\item {\boldmath$I_{\Omega_k^c}$.}
The term $I_{\Omega_k^c}$ is dealt with in an analogous way to $I_\Omega$ since equation \eqref{IntOmega} holds replacing $\Omega$ by $\Omega_k^c$. In particular, this term then vanishes since $u_k\to 0$ strongly in $L^1(\Omega_k^c)$ (using Lemma \ref{lem:uk_to_0}).

\item {\boldmath$I_{\Omega_k \setminus \Omega}$.}
We now check that $I_{\Omega_k \setminus \Omega}$ also vanishes. Indeed we have
\begin{align*}
I_{\Omega_k \setminus \Omega} \leq \|u_k\|_{L^\infty (0, T; L^1(\Omega_k \setminus \Omega))}^{1/2} 
\bigg(
\int_0^T \! \int_{\Omega_k \setminus\Omega}
u_k
\left|  
\nabla (\log u_k + \xi(u_k)
+   V_k 
+   W*u_k )
\right|^2
\d x \d t
\bigg)^{\frac{1}{2}}.
\end{align*}
The second factor in the right hand side is bounded by the Lemma \ref{lem:bound-step3}. The first factor is bounded and converges to zero as $k\to \infty$ using the following argument. By Jensen's inequality we obtain, for a region $R$ such that $|R|<1$,
\begin{equation}\label{equintegrability}
\int_R g(x) \d x \leq
 \left( - \log |R|) 
 \right)^{-1}
   \left(
   \int_R g(x) \log g(x) \d x + e^{-1} 
   \right).
\end{equation}
We use inequality \eqref{equintegrability} with $R=\Omega_k \setminus \Omega$ and $g=u_k$. Recalling that $|\Omega_k \setminus \Omega | \to 0$ as $k\to \infty$, we obtain the desired result, $I_{\Omega_k \setminus \Omega} \to 0$. 

\item {\boldmath$J_{\Omega}$.}
Up to a subsequence, $u_k$ converges weakly in $L^1(\Omega)$ to a limit function $u$, hence
\begin{equation}\label{eq:convergence2b}
J_\Omega \to  - 
\int_0^T\int_{A}    
u \de_t\eta \,
\d x \d t,
\quad \text{ as } k\to\infty.
\end{equation}

\item {\boldmath$J_{\Omega_k^c}$.}
Notice that $u_0 = 0$ in $\Omega_k^c$. The remaining terms in $J_{\Omega_k^c}$ vanish due to Lemma \ref{lem:uk_to_0}.

\item {\boldmath$J_{\Omega_k \setminus \Omega}$.}
The integral $J_{\Omega_k \setminus \Omega}$ goes to zero 
because the integrand is uniformly bounded in $L^1$ (thanks to the conservation of mass, see Lemma \ref{lem:mass-pos}) and $|\Omega_k \setminus\Omega| \to 0$. 
\end{itemize}

Thanks to \eqref{eq:convergence1} and \eqref{eq:convergence2b}, the weak formulation we obtain in the limit is the following:
$$
\int_0^T\int_{\Omega}   \left[
\left(  
\nabla \phi(u) 
+  u \nabla V_0 
+  u \nabla W*u 
\right)\cdot \nabla \eta  
- u \de_t \eta 
\right]
\d x \d t = 0.
$$
Notice that since the test function can take arbitrary values on $\de\Omega$, the no-flux conditions on $\de\Omega$ are implicitly enforced.

We now show that initial datum is satisfied in $\mathcal{P}_2(\RR^d)$.
To do so, we use the characterization of $\mathcal{P}_2(\RR^d)$ convergence given in Proposition 7.1.5, in \cite{ambrosio2008gradient}.
In particular, the the second moment $\int_{\RR^d} u_k |x|^2 \d x$ is bounded uniformly ($V_k(x) \geq c|x|^2$ at infinity) 
and the so-called narrow convergence is implied by the $L^1$ bounds obtained in \eqref{eq:decay}, \eqref{l1strong} and \eqref{equintegrability}. 
We deduce that  $u_k$ converges to $u$  in $\mathcal{P}_2(\RR^d)$ as well.
As a consequence, we obtain that the initial datum is satisfied  in the $\mathcal{P}_2(\RR^d)$ sense, indeed, since our $L^1$ bounds are uniform in time, we have that
$$
\lim_{t\to 0} \lim_{k\to\infty} d_{2}(u(t,\cdot),u_0)
\leq \lim_{t\to 0} \lim_{k\to\infty} 
(
d_{2}(u_k(t,\cdot),u(t,\cdot))+d_{2}(u_k(t,\cdot),u_0)
)
=0,
$$
where $d_2$ denotes the 2-Wasserstein distance. Hence the result is proven.
\end{proof}

\begin{proof}[Proof of Theorem \ref{thm:main2}, degenerate case]
Let us now drop the linear diffusion term and consider $\phi = \sigma$.
Almost all the results above remain unchanged, therefore we will only highlight the steps that differ.
Notice that, as before, $u_k$ converges to $0$ in the complement of $\Omega$.
Let us discuss the passage to the limit in the term $I_\Omega$ in the proof of Theorem \ref{thmAGS}.
More specifically, we have to obtain 
strong $L^1$ convergence in $\Omega$ finding an alternative to \eqref{diffterms}.
Similarly to \eqref{bound0}, we have
$$
\int_0^T\int_{\Omega} 
u_k \left| 
\nabla (\xi(u_k) + V_0 + W*u_k)
\right|^2
\d x \d t
\leq 
C_0.
$$
Recall that, by assumption, we have $\sigma'(s)\geq \mu s^a$ for some $a>0$. 
Since $u_k$ has constant mass (and hence it is uniformly bounded in $L^\infty(0,T;L^1(\Omega))$), we deduce that
$$
\int_0^T\int_{\Omega} 
\mu u^{2a-1}\left| 
\nabla u_k
\right|^2
\d x \d t
\leq
\int_0^T\int_{\Omega} 
\frac{\sigma'(u)^2}{u}\left| 
\nabla u_k
\right|^2
\d x \d t
=
\int_0^T\int_{\Omega} 
u_k \left| 
\nabla \xi(u_k)
\right|^2
\d x \d t
\leq 
C_0.
$$
This implies that $\nabla \left(u_k^{a + \frac{1}{2}}\right) \in L^2(\Omega_T)$. 
We observe that,
since $\Xi(u_k)$ and $\nabla\xi(u_k)$ are bounded in $L^2(\Omega)$ independently of $k$, and since $u_k$ is non-negative,
inequality \eqref{eq:entropybound} implies, for a constant $C$ depending only on $a$ and $\mu$,
\begin{equation}\label{eq:est-deg}
 \int_\Omega u_k(t)^{1+a} \d x \leq C(a, \mu)\int_\Omega \Xi(u_k(t)) \d x <\infty,
 \quad \forall t\geq 0.
\end{equation}
It follows that $\sqrt{u_k}\in L^\infty(0,T;L^{2+2a}(\Omega))$, which implies $\sqrt{u_k}\in L^r(\Omega_T)$ for any $r\in (2, 2+2a )$.
We can estimate $\de_tu_k$ as follows:
\begin{align*}
\langle \de_tu_k, \eta \rangle 
&=
\int_{\omega_T}
\sqrt{u_k}\sqrt{u_k}\left(  
\nabla (\xi(u_k) + V_0 + W*u_k)
\right)\cdot \nabla \eta  
\d x \d t
\\
&\leq
C_0^\frac{1}{2} \nm{\sqrt{u_k}}_{L^r(\Omega_T)} \nm{\nabla\eta}_{L^s(\Omega_T)},
\end{align*}
where $r$ and $s$ satisfy $\frac{1}{r}+\frac{1}{s}=\frac{1}{2}$.
This implies that $\de_tu_k$ is bounded in $L^{s}(0,T;W^{1,s}(\Omega))'$ uniformly with respect to $k$.
We can now apply either the modified version of Aubin--Lions Lemma presented in Theorem 1 of \cite{moussa2016some} (with $\Phi(\cdot) = |\cdot|^{a + \frac{1}{2}}$) if $a\geq 1$, or Theorem 3 in \cite{chen2014note} (with $m=a + \frac{1}{2}$) otherwise.
Hence, by compactness, the sequence $u_k$ converges a.e. to a limit $u$. Combining a.e. convergence with inequality \eqref{eq:est-deg} and uniqueness of weak limits, we obtain strong $L^1$ convergence of $u_k$ to $u$ for $k\to\infty$. This allows to conclude the proof.
\end{proof}

\begin{rem}[Moments]
We have used the hypothesis of quadratic growth of $V$ at infinity only to ensure that the second moment $\int_{\RR^d} u_k |x|^2 \d x$ is bounded and therefore that the initial datum is satisfied in  $\mathcal{P}_2(\RR^d)$. It is possible to make less restrictive assumptions, for example, if $V$ grows linearly at infinity we get control over the first moment and the initial datum is satisfied in $\mathcal{P}_1(\RR^d)$.
\end{rem}

\section{Numerical exploration} \label{sec:numerics}

We will now illustrate our main results on the approximation of no-flux boundary value problems by large confinement with some numerical results. Here, we make use of a numerical scheme with excellent properties such as semidiscrete free energy decay, and positivity under a CFL condition. The numerical scheme is based on a finite volume discretization with upwinding and second order reconstruction. We refer to \cite{BF12,CCH15} and the references therein for further details. This numerical strategy has been successfully used in many similar gradient flow type equations and systems \cite{CHS18}, and it has recently been generalized to high order DG-approximations in \cite{sCS18}. All our numerical results are obtained with the original second-order version in \cite{CCH15}. In the next subsection we will showcase our results in one dimension and then we will explore the behavior of the solutions when the initial is not necessarily supported on the limiting domain. Finally, the last subsection explores these issues in two dimensions.

\subsection{One-dimensional examples}
We first run simulations of the one-dimensional problem, with $\Omega = [-1, 1]$ and a computational domain $B = [-4,4]$ (we choose $B$ large enough so that its size does not affect the solution in $\Omega$). 

Figure \ref{fig:linear} shows an example with the linear Fokker--Planck equation ($\phi = u$ and $W = 0$). 
Figure \ref{fig:local} shows an example with a nonlinear local Fokker--Planck equation (with $\phi = u + \beta u^2$ and $W = 0$). 
Figure \ref{fig:nonlocal} shows an example with a nonlinear nonlocal Fokker--Planck equation (with $\phi = u + \beta u^2$ and $W  = -(1-|x|)_+$). In all three figures, the subplot (a) shows the confining potentials $V_k$ in thin colored lines, and the potential $V_0$ in $\Omega$ in a thick black line. The subplot (b) shows the solutions $u_k$ at the final simulation time in thin colored lines and the solution of the limit problem $u$ (only defined in $\Omega$) as a thick black line. The subplot (c) shows the $l^2$-norm between $u_k$ and $u$ in $\Omega$ as a function of time, for various values of $k$. The subplot (d) shows again the $l^2$-norm between $u_k$ and $u$ in $\Omega$ but only at the final time (circles), as well as the norm of $u_k$ in $B\setminus \Omega$ at the final time. As expected, we observe that $u_k$ get closer to $u$ in $\Omega$ and that the norms of the errors decrease as $k$ increases. 
\def \scc {0.7}
\def \scl {0.8}
\begin{figure}
\unitlength=1cm
\begin{center}
\vspace{3mm}
\psfrag{a}[l][][\scl]{(a)} \psfrag{b}[l][][\scl]{(b)}
\psfrag{c}[l][][\scl]{(c)} \psfrag{d}[l][][\scl]{(d)}
\psfrag{x}[][][\scl]{$x$} \psfrag{k}[][][\scl]{$k$} 
\psfrag{Vk}[][][\scl]{$V_k$} \psfrag{uk}[b][][\scl]{$u_k$} 
\psfrag{normt}[b][][\scc]{$\|u-u_k\|$} \psfrag{norm}[b][][\scc]{$\|u-u_k\|, \|u_k\|$} \psfrag{t}[][][\scl]{$t$}
	\includegraphics[height = \size]{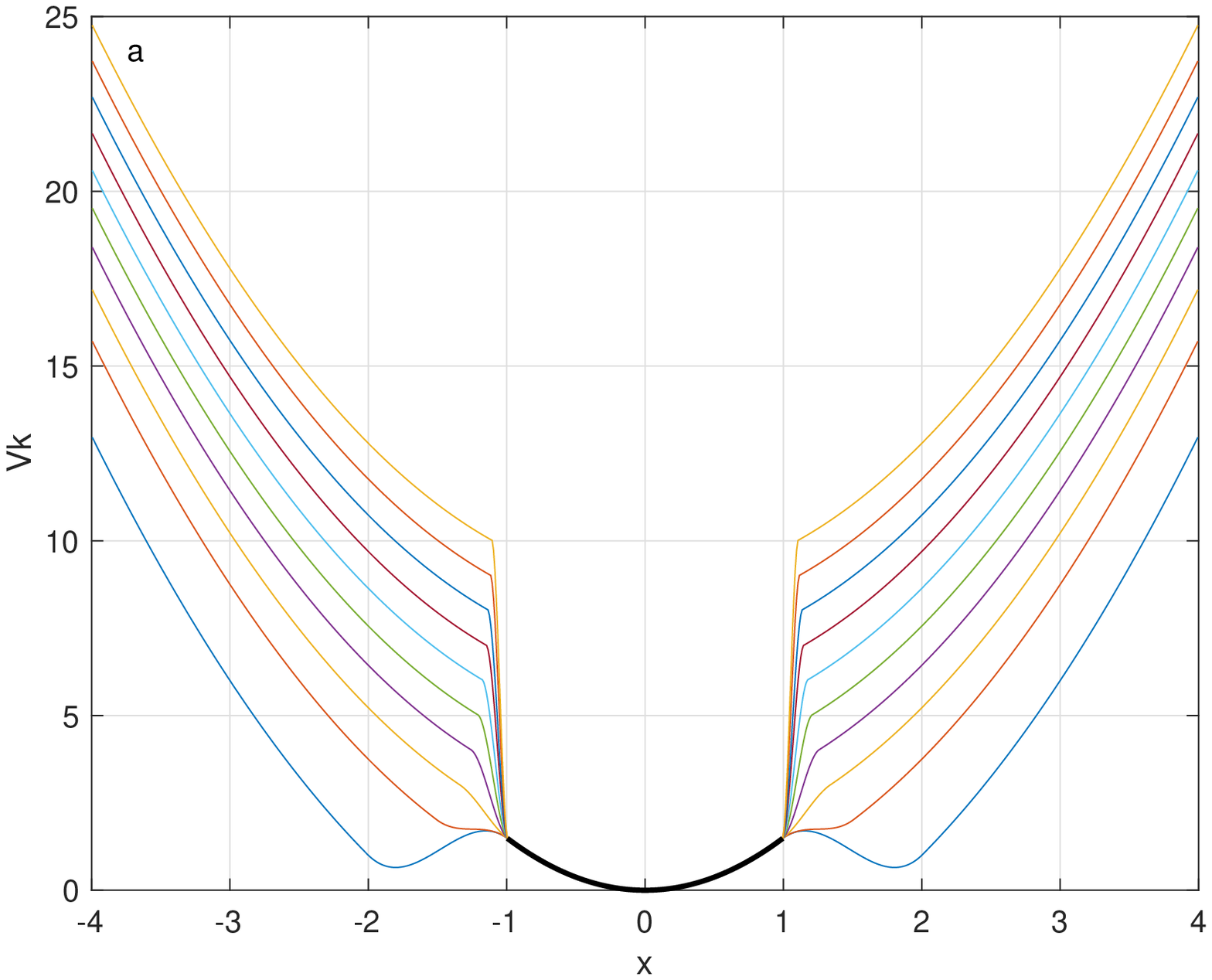} \qquad 
	\includegraphics[height = \size]{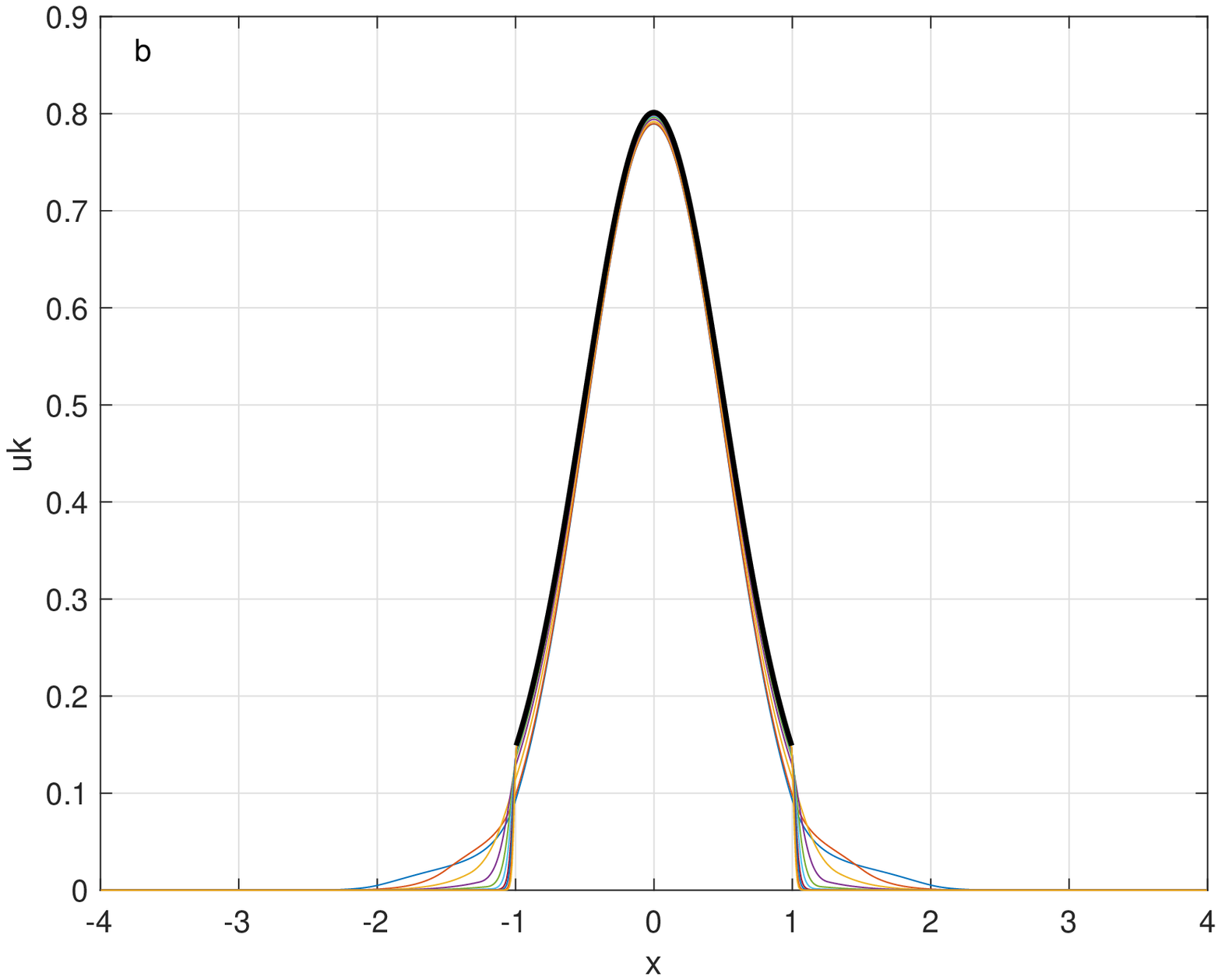}\\ \vspace{5mm}
	\includegraphics[height = \size]{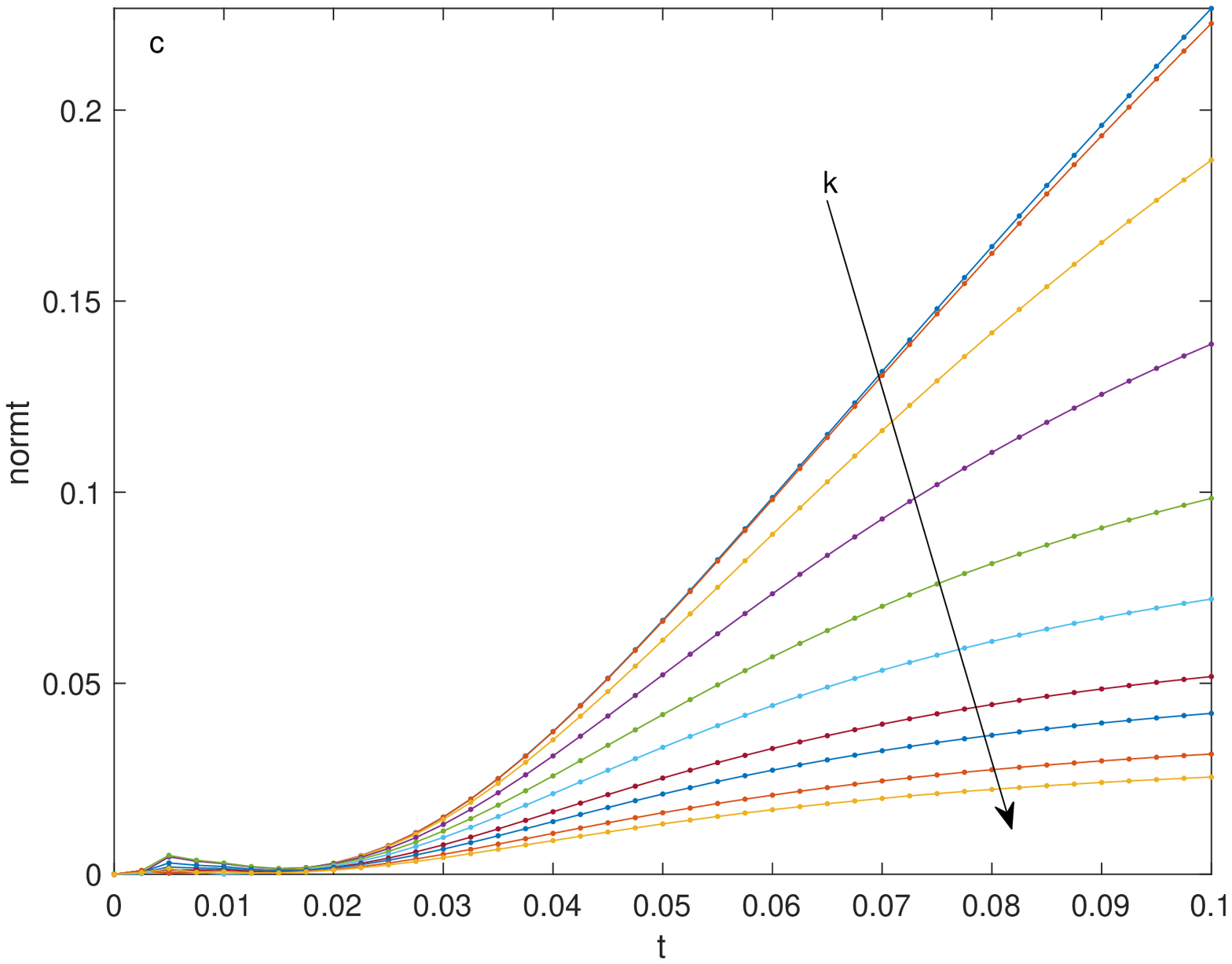} \qquad
	\includegraphics[height = \size]{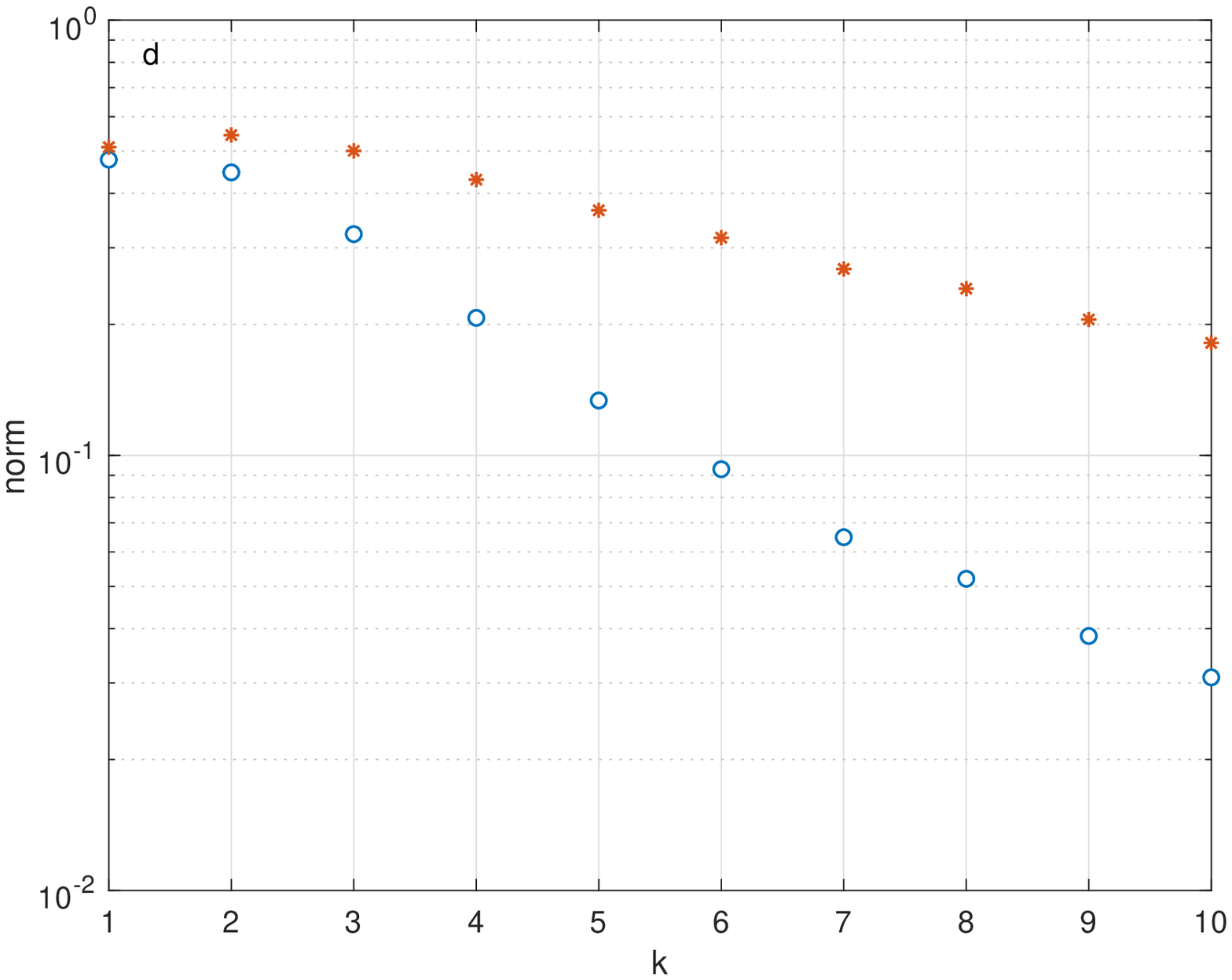} 
\caption{Example with linear diffusion, local. External potential $V(x) = 1.5 x^2$, linear diffusion $\phi (u) = u$, zero nonlocal term $W \equiv 0$. Simulation with $\Omega = [-1, 1]$, computational domain for $u_k$ is $B = [-4, 4]$ ($k \ge 1$), initial condition $u_0 = 1$, final time $T_f = 0.2$. Grid spacing $\Delta x = 0.01$ and initial time-step $\Delta t = 10^{-5}$. (a) Confining potential $V_k$ for $k = 1,2, \dots, 10$ (potential $V$ shown in thick black line). (b) Solutions $u$ (thick black line) and $u_k$ (colored thin lines) at $t = T_f$. (c) Norm between $u$ and $u_k$ in $\Omega$ at times $t \in [0, 0.1]$. (d) Norm between $u$ and $u_k$ in $\Omega$ at $t = T_f$ (circles) and norm of $u_k$ in $B \setminus \Omega$ at $t = T_f$ (asterisks).}
\label{fig:linear}
  \end{center}
\end{figure}

\def \scc {0.7}
\def \scl {0.8}
\begin{figure}
\unitlength=1cm
\begin{center}
\vspace{3mm}
\psfrag{a}[l][][\scl]{(a)} \psfrag{b}[l][][\scl]{(b)}
\psfrag{c}[l][][\scl]{(c)} \psfrag{d}[l][][\scl]{(d)}
\psfrag{x}[][][\scl]{$x$} \psfrag{k}[][][\scl]{$k$} 
\psfrag{Vk}[][][\scl]{$V_k$} \psfrag{uk}[b][][\scl]{$u_k$} 
\psfrag{normt}[b][][\scc]{$\|u-u_k\|$} \psfrag{norm}[b][][\scc]{$\|u-u_k\|, \|u_k\|$} \psfrag{t}[][][\scl]{$t$}
	\includegraphics[height = \size]{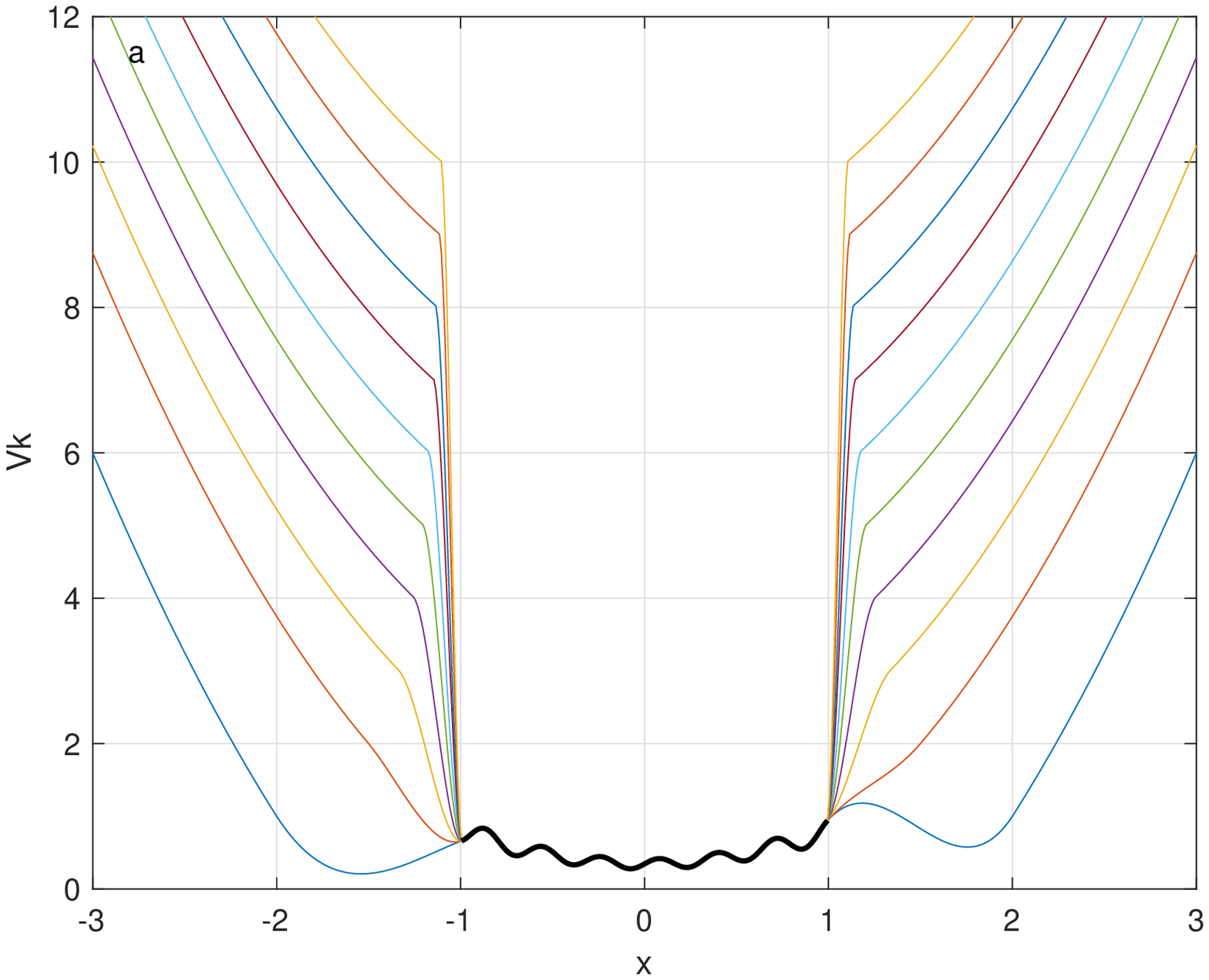} \qquad 
	\includegraphics[height = \size]{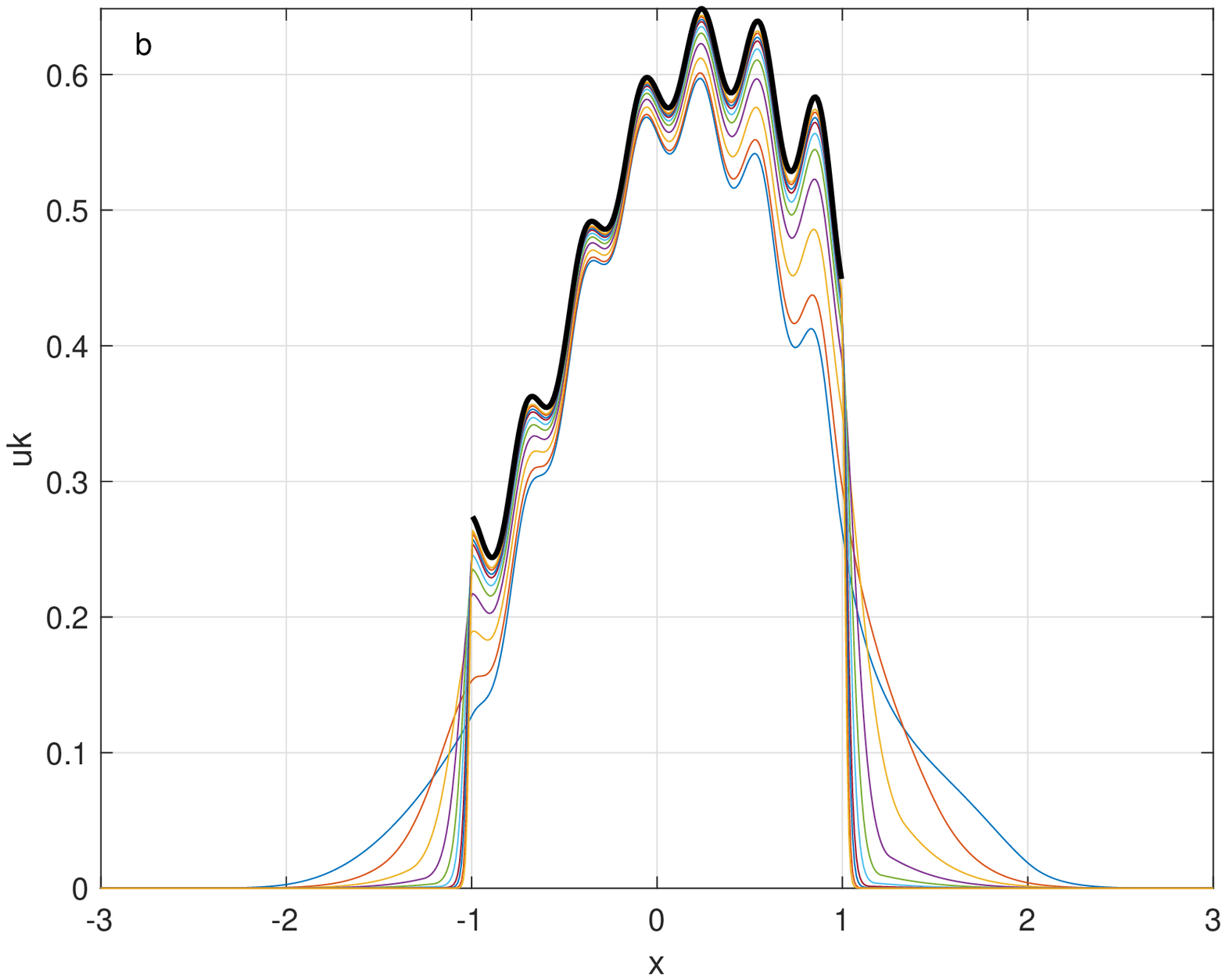}\\ \vspace{5mm}
	\includegraphics[height = \size]{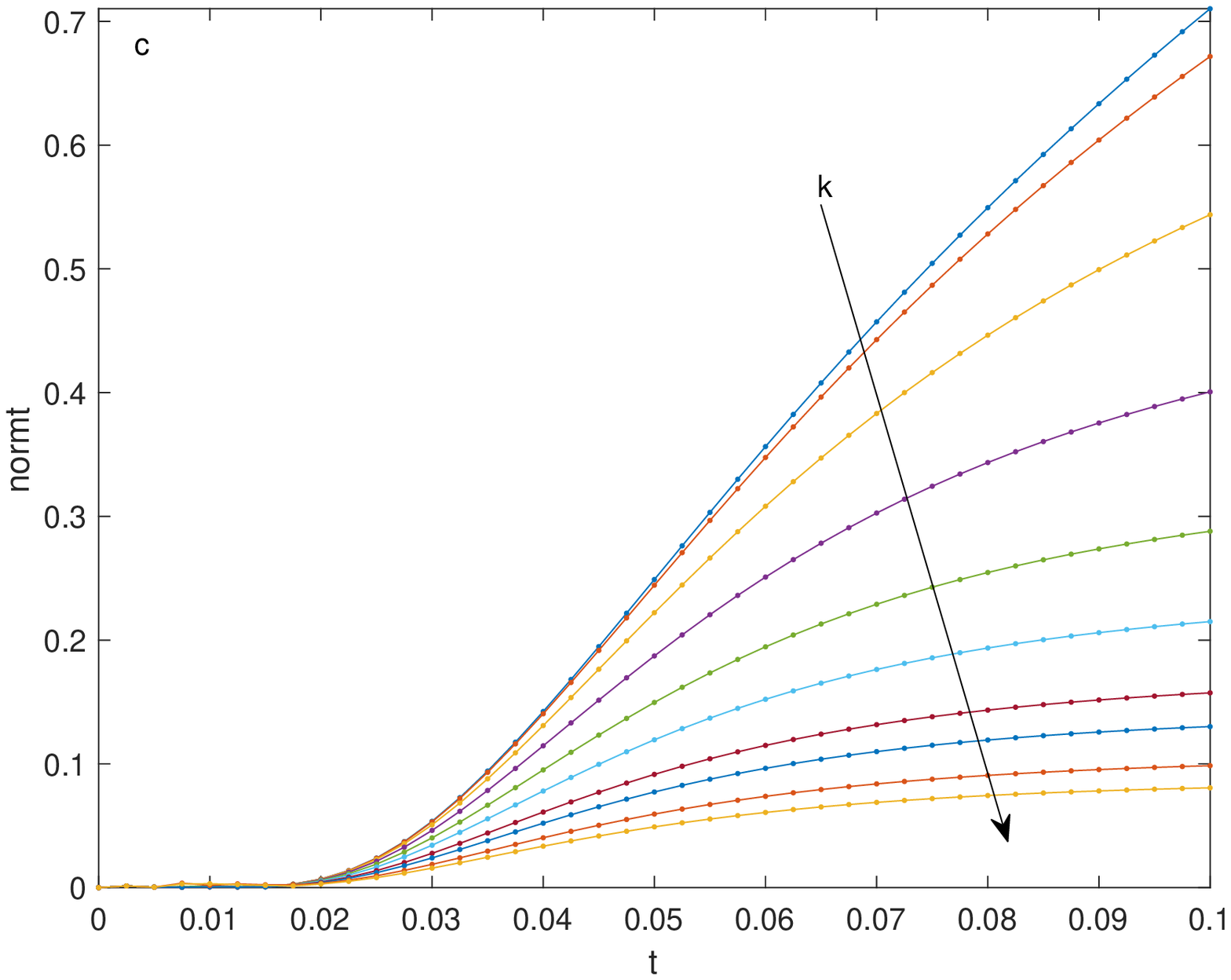} \qquad
	\includegraphics[height = \size]{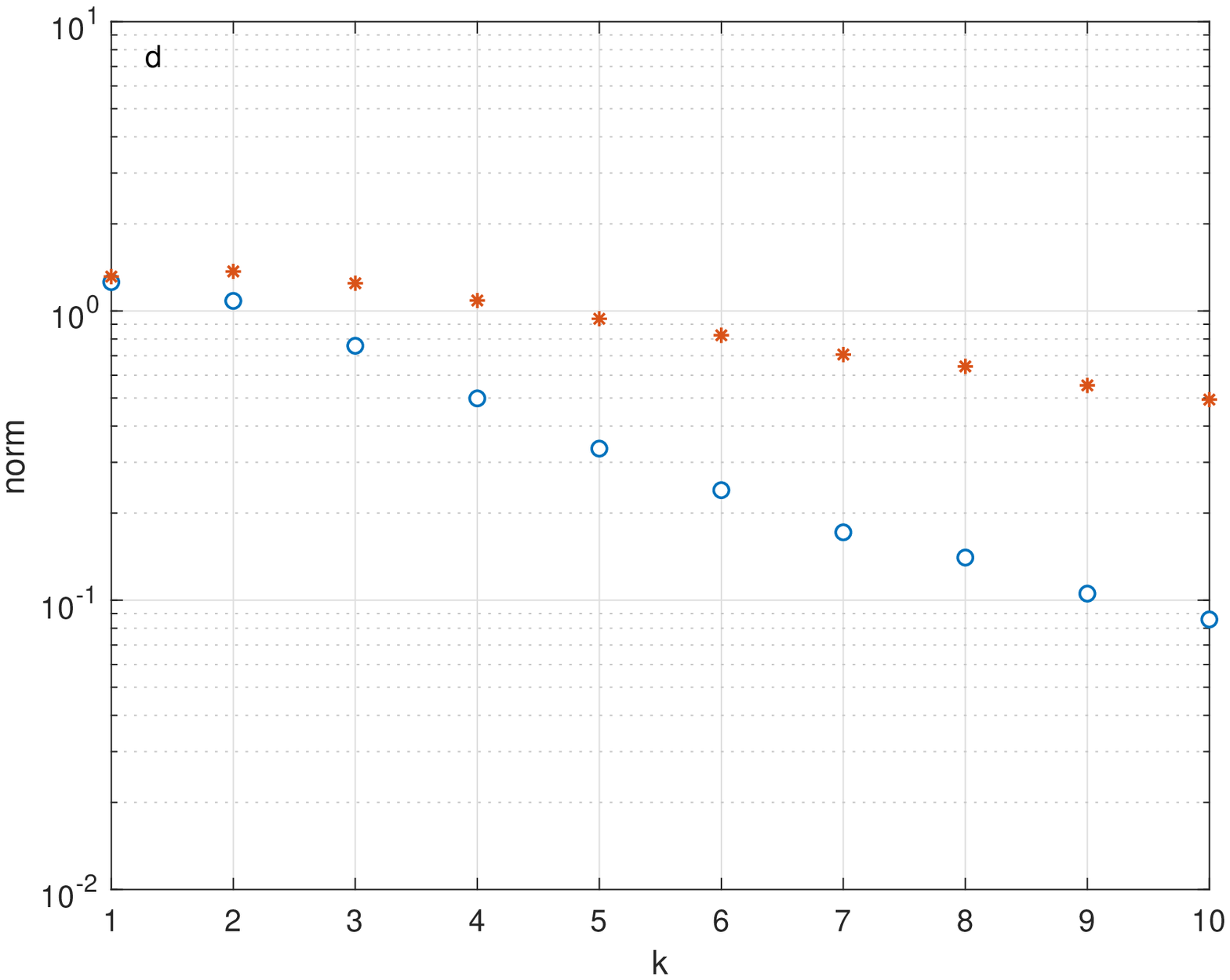} 
\caption{Example with nonlinear diffusion, local. External potential $V(x) = 0.46(1+0.2 \sin(20x))(x^2 + 0.75)$, nonlinear diffusion $\phi (u) = u + \beta u^2$, with $\beta = 0.49$, zero nonlocal term $W \equiv 0$. Simulation with $\Omega = [-1, 1]$, computational domain for $u_k$ is $B = [-4, 4]$ ($k \ge 1$), initial condition $u_0 = \chi_{[0.1, 0.3]}$, final time $T_f = 0.2$. Grid spacing $\Delta x = 0.005$ and initial time-step $\Delta t = 10^{-5}$. (a) Confining potential $V_k$ for $k = 1, 2,\dots, 10$ (potential $V$ shown in thick black line). (b) Solutions $u$ (thick black line) and $u_k$ (colored thin lines) at $t = T_f$. (c) Norm between $u$ and $u_k$ in $\Omega$ at times $t \in [0, 0.1]$. (d) Norm between $u$ and $u_k$ in $\Omega$ at $t = T_f$ (circles) and norm of $u_k$ in $B \setminus \Omega$ at $t = T_f$ (asterisks).}
\label{fig:local}
  \end{center}
\end{figure}

\def \scc {0.7}
\def \scl {0.8}
\begin{figure}
\unitlength=1cm
\begin{center}
\vspace{3mm}
\psfrag{a}[l][][\scl]{(a)} \psfrag{b}[l][][\scl]{(b)}
\psfrag{c}[l][][\scl]{(c)} \psfrag{d}[l][][\scl]{(d)}
\psfrag{x}[][][\scl]{$x$} \psfrag{k}[][][\scl]{$k$} 
\psfrag{Vk}[][][\scl]{$V_k$} \psfrag{uk}[b][][\scl]{$u_k$} 
\psfrag{normt}[b][][\scc]{$\|u-u_k\|$} \psfrag{norm}[b][][\scc]{$\|u-u_k\|, \|u_k\|$} \psfrag{t}[][][\scl]{$t$}
	\includegraphics[height = \size]{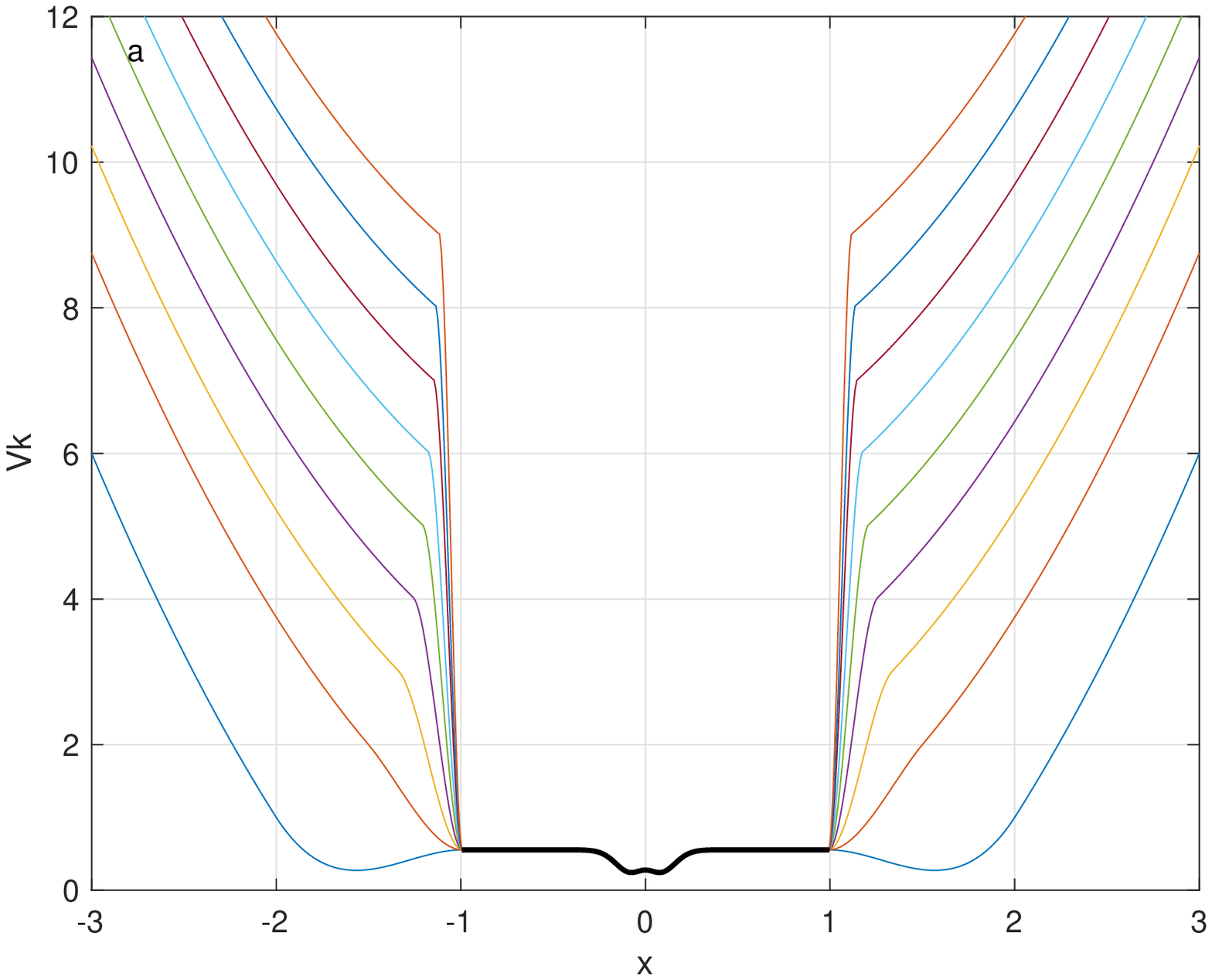} \qquad 
	\includegraphics[height = \size]{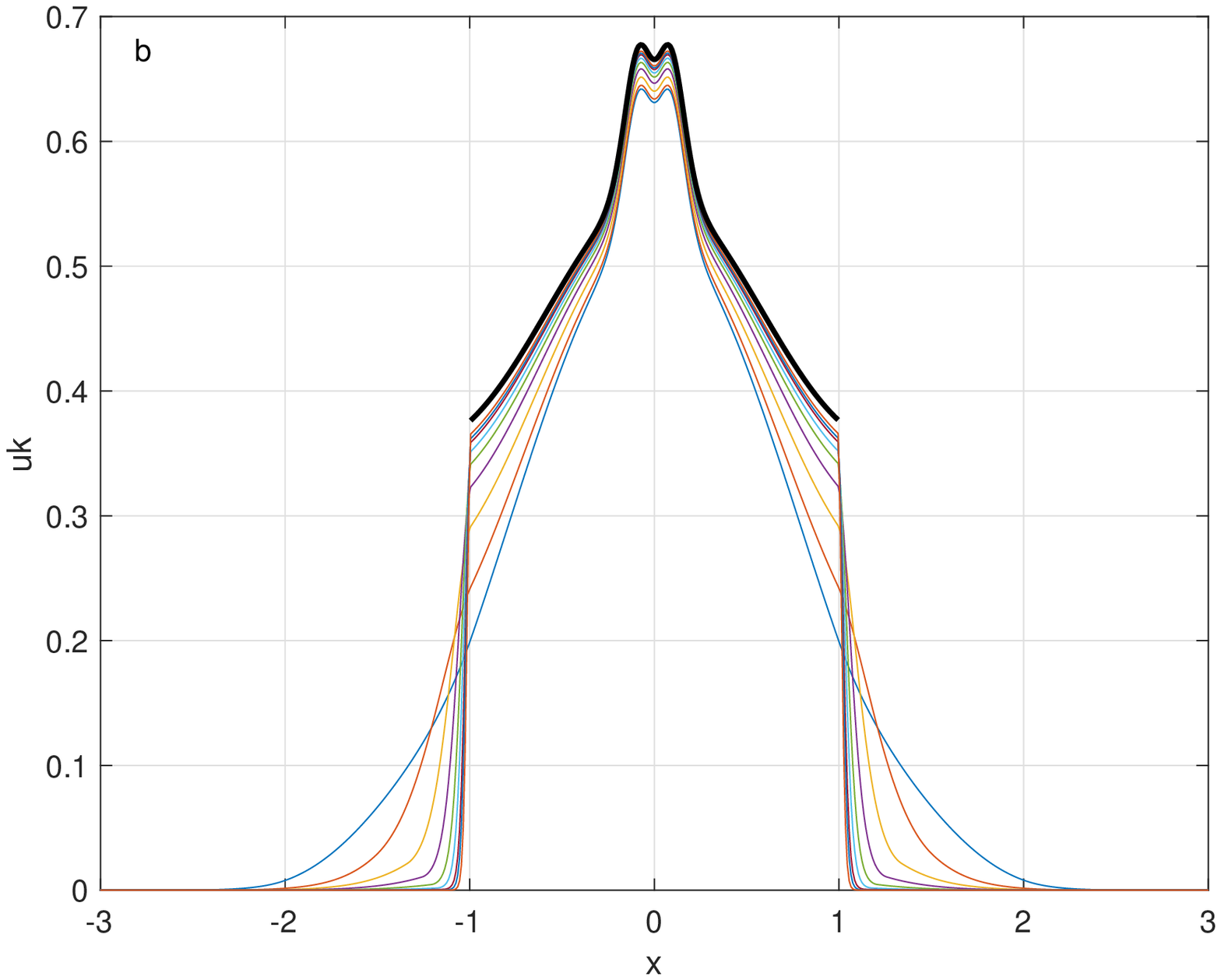}\\ \vspace{5mm}
	\includegraphics[height = \size]{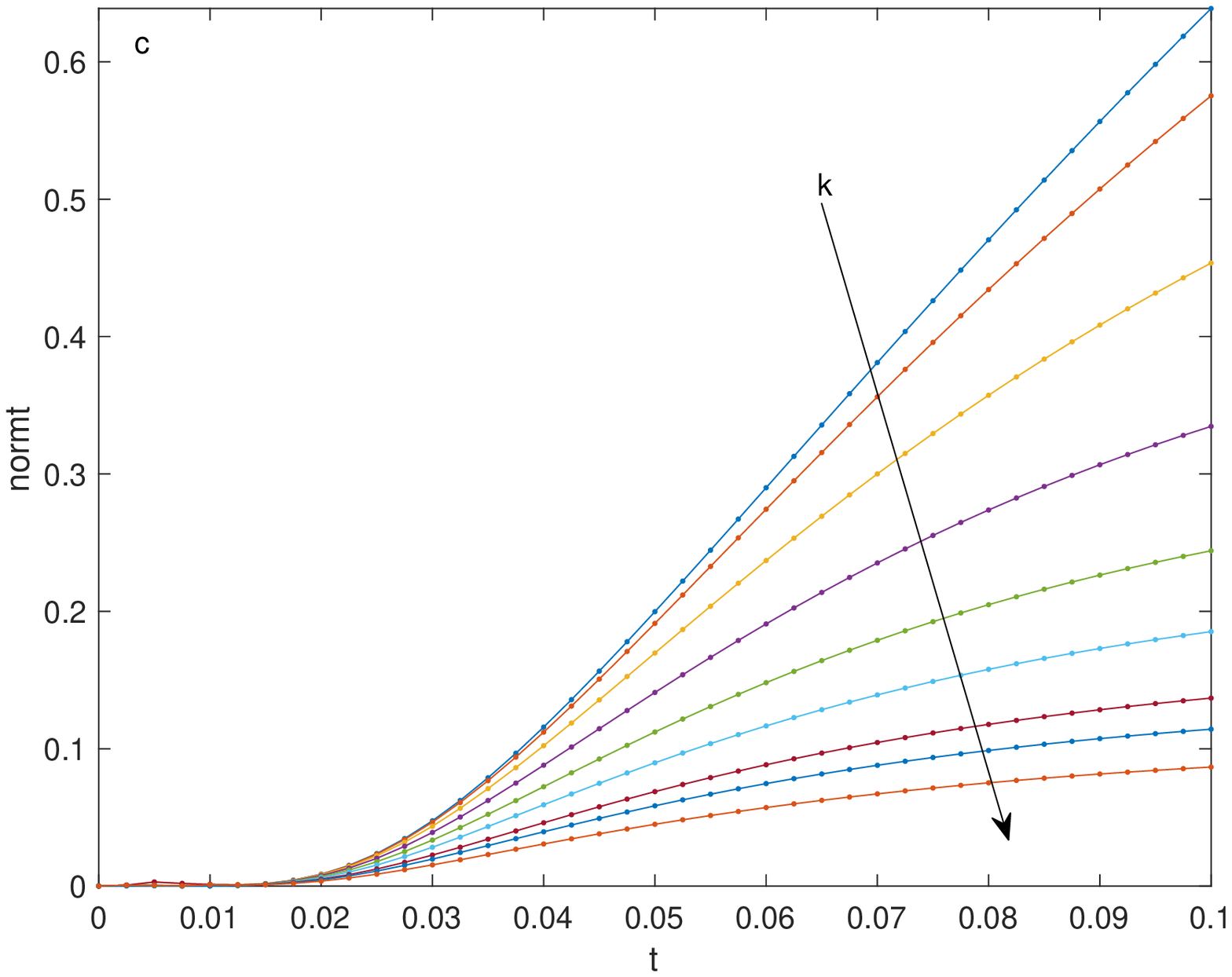} \qquad
	\includegraphics[height = \size]{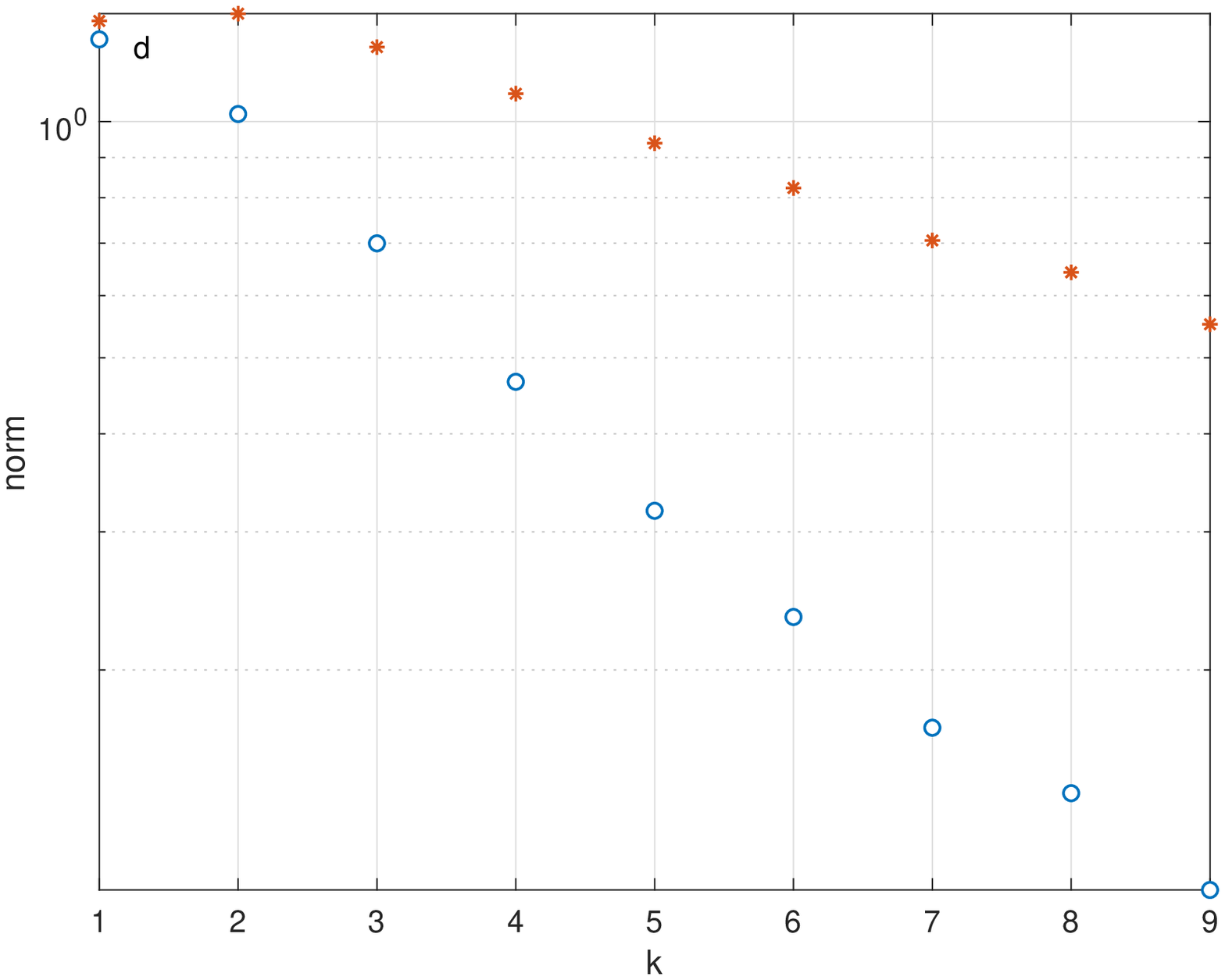} 
\caption{Example with nonlinear diffusion, nonlocal. External potential $V(x) = 0.56[ \exp(-100 x^2) - 1.5 \exp(-50 x^2) +1]$ , nonlinear diffusion $\phi (u) = u + \beta u^2$, with $\beta = 0.49$, nonlocal term $W  = -(1-|x|)_+$. Simulation with $\Omega = [-1, 1]$, computational domain for $u_k$ is $B = [-4, 4]$ ($k \ge 1$), initial condition $u_0 = 1$, final time $T_f = 0.2$. Grid spacing $\Delta x = 0.01$ and initial time-step $\Delta t = 10^{-5}$. (a) Confining potential $V_k$ for $k = 1, 2,\dots, 9$ (potential $V$ shown in thick black line). (b) Solutions $u$ (thick black line) and $u_k$ (colored thin lines) at $t = T_f$. (c) Norm between $u$ and $u_k$ in $\Omega$ at times $t \in [0, 0.1]$. (d) Norm between $u$ and $u_k$ in $\Omega$ at $t = T_f$ (circles) and norm of $u_k$ in $B \setminus \Omega$ at $t = T_f$ (asterisks).}
\label{fig:nonlocal}
  \end{center}
\end{figure}

In Figure \ref{fig:cas4} 
we consider an example with the same initial condition $u_0$ and potential $V_0$ for all three cases (linear, nonlinear local and nonlinear nonlocal Fokker--Planck equation) so that we can compare the effects that the different terms have in the solution. We consider a simple case with no external potential in $\Omega$, $V_0 = 0$ (see Figure \ref{fig:cas4}(a)) and initial data $u_0 = \chi_{[-1,-0.7]\cup [0.7,1]}$. Figures \ref{fig:cas4}(b-d) show the solutions $u_k$ at $T_f = 0.2$ and for $k = 1, \dots, 10$ (colored lines) and the limit problem solution $u$ (thick black line) in the linear, nonlinear, and nonlocal cases, respectively.

\def \scc {0.7}
\def \scl {0.8}
\begin{figure}
\unitlength=1cm
\begin{center}
\vspace{3mm}
\psfrag{a}[l][][\scl]{(a)} \psfrag{b1}[l][][\scl]{(b)}
\psfrag{b2}[l][][\scl]{(c)} \psfrag{b3}[l][][\scl]{(d)}
\psfrag{x}[][][\scl]{$x$} \psfrag{k}[][][\scl]{$k$} 
\psfrag{Vk}[][][\scl]{$V_k$} \psfrag{uk}[b][][\scl]{$u_k$} 
\psfrag{normt}[b][][\scc]{$\|u-u_k\|$} \psfrag{norm}[b][][\scc]{$\|u-u_k\|, \|u_k\|$} \psfrag{t}[][][\scl]{$t$}
	\includegraphics[height = \size]{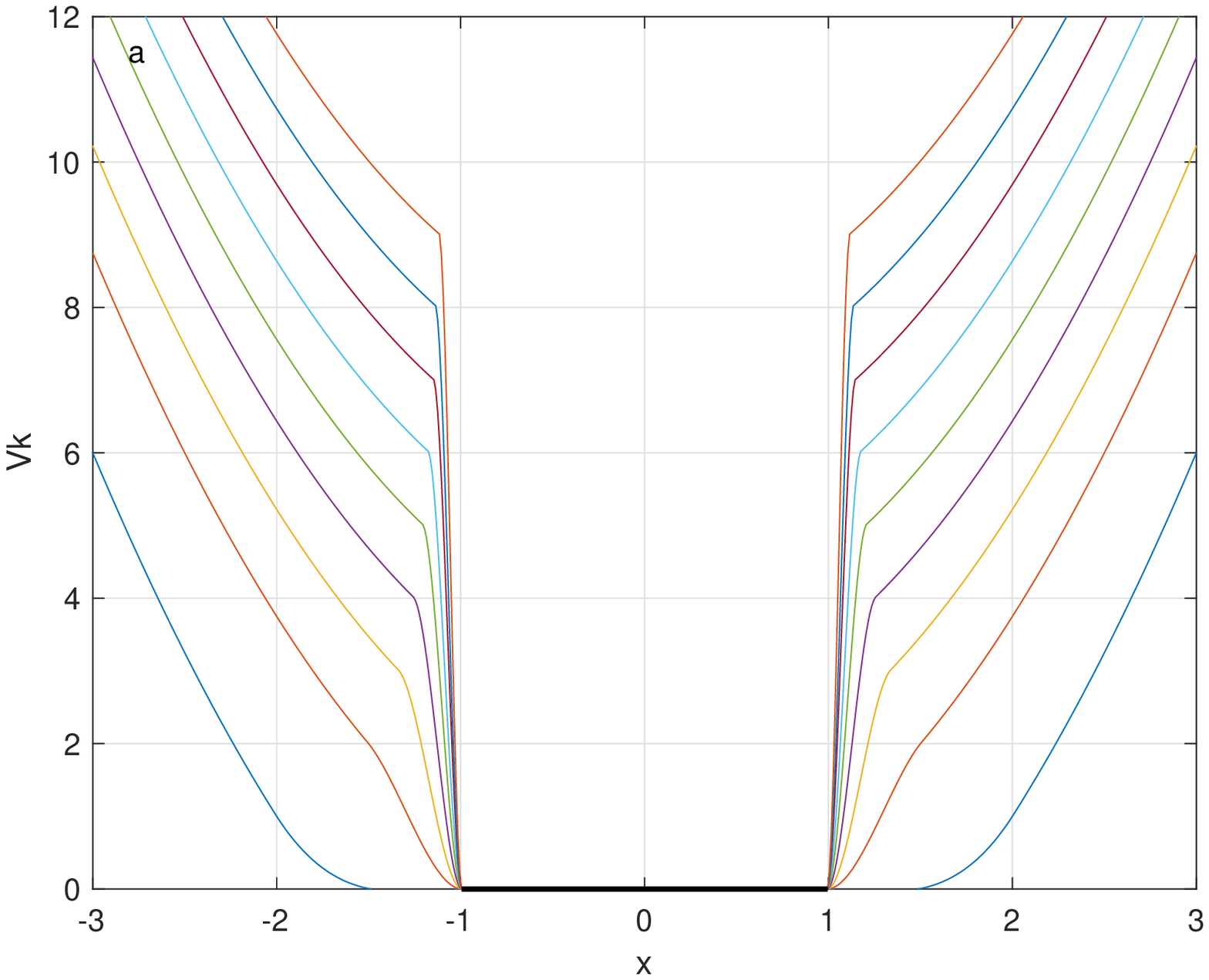} \qquad 
	\includegraphics[height = \size]{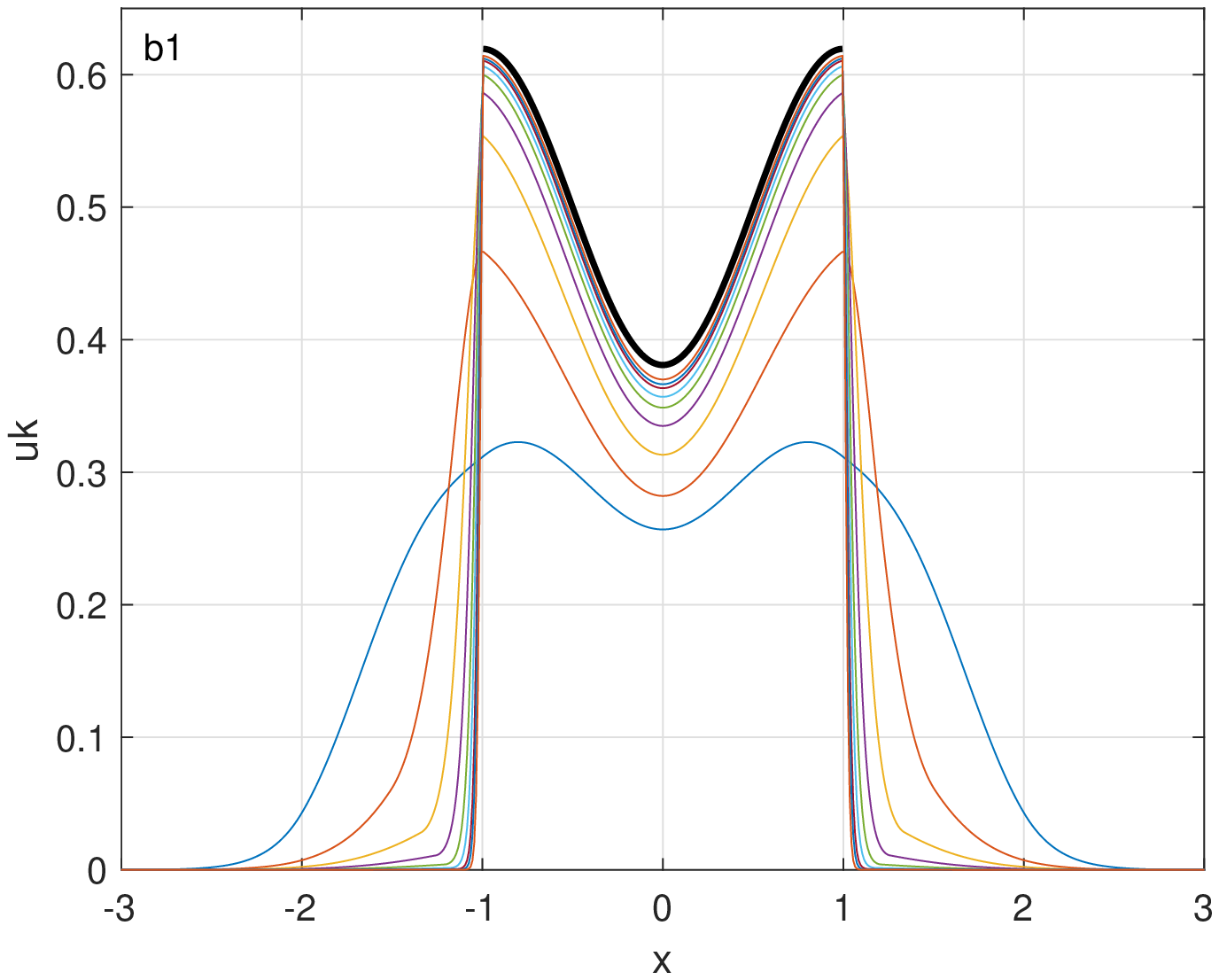}\\ \vspace{5mm}
	\includegraphics[height = \size]{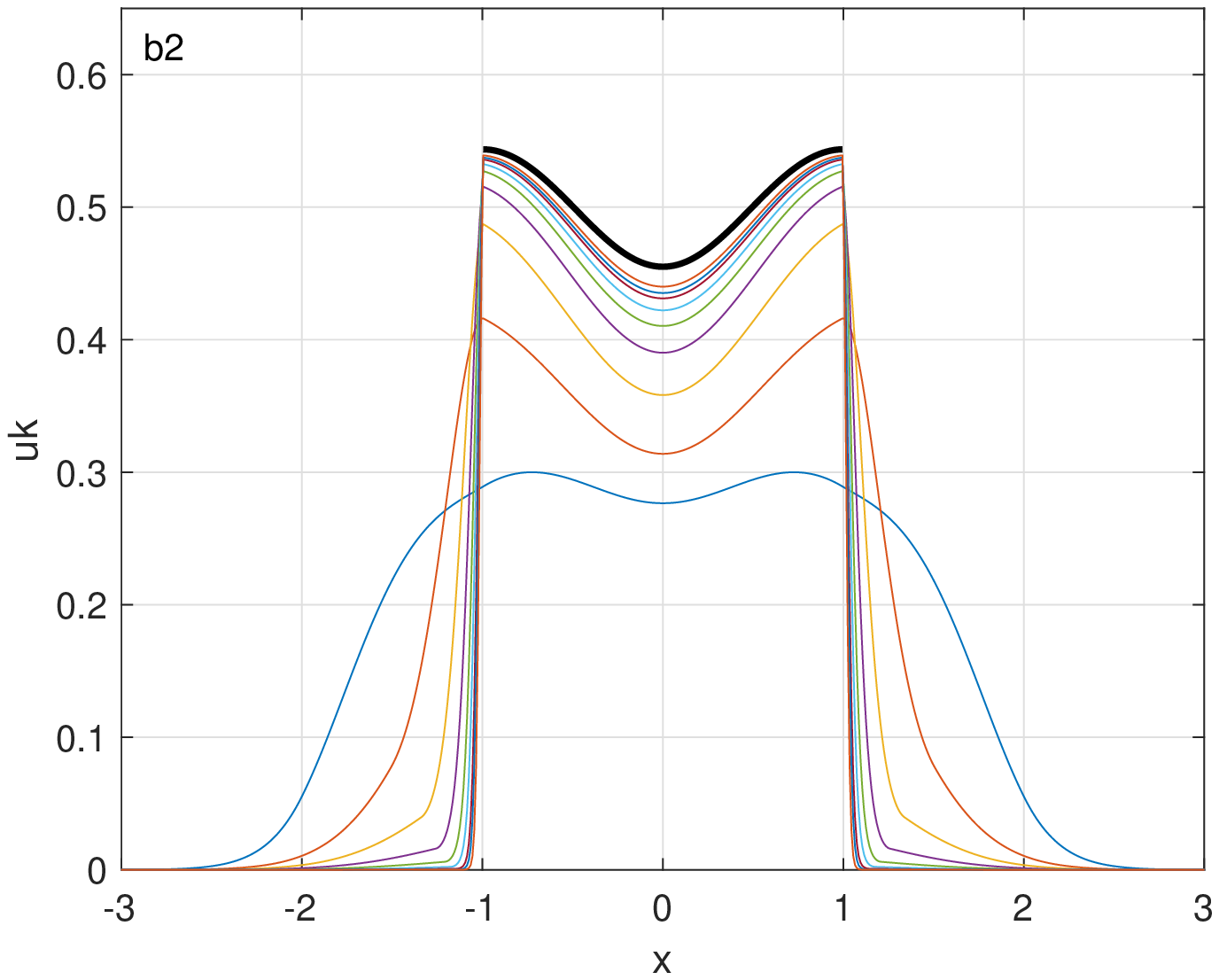} \qquad
	\includegraphics[height = \size]{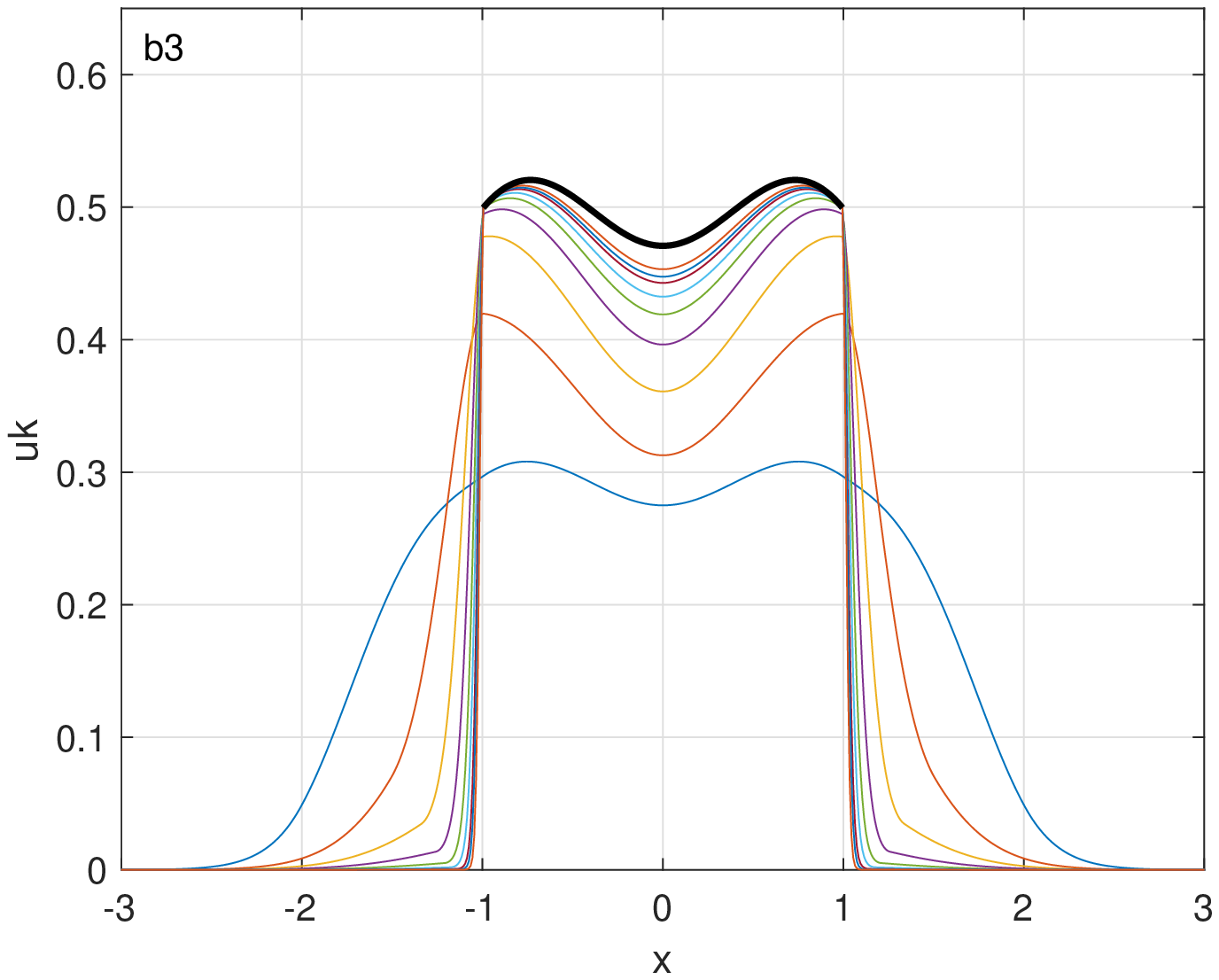} 
\caption{Comparison between models in an example with no external potential in $\Omega$, $V_0 = 0$, and initial data $u_0 = \chi_{[-1,-0.7]\cup [0.7,1]}$. (a) Confining potential $V_k$ for $k = 1, \dots, 10$ (potential $V$ shown in thick black line). (b-d) Solutions $u$ (thick black line) and $u_k$ (colored thin lines) at $t = 0.2$: (b) $\phi = u$ and $W = 0$, (c) $\phi (u) = u + \beta u^2$, with $\beta = 0.49$, and $W= 0$, (d) $\phi (u) = u + \beta u^2$, with $\beta = 0.49$ and $W  = -(1-|x|)_+$. Simulation with $\Omega = [-1, 1]$, computational domain for $u_k$ is $B = [-4, 4]$ ($k \ge 1$). Grid spacing $\Delta x = 0.01$ and initial time-step $\Delta t = 10^{-5}$.}
\label{fig:cas4}
  \end{center}
\end{figure}

For our final one-dimensional simulation, we consider a case (not allowed in our analysis) where part of the support of the initial data $u_0$ lies outside $\Omega = [-1,1]$. In particular, the initial condition for $u_k$ is:
\begin{equation}
	u_0 (x) = \frac{C}{\sqrt{2\pi \sigma^2}} e^{-x^2/(2\sigma^2)}, \qquad x \in B = [-4, 4],
\end{equation}
with $\sigma = 2$ and where $C$ is a constant such that $\int_B u_0 \d x = 1$. The initial condition for the limit problem is $u_0$ constrained in $\Omega$ and the mass that lies outside $\Omega$ placed on the  $x = \pm 1$:
\begin{equation} \label{u0_shifted}
	\bar u_0(x)  = u_0(x) + M_l \delta_{-1}(x) + M_r \delta_{1}(x), \qquad x \in \Omega = [-1, 1],
\end{equation}
where $M_l = \int_{-4}^{-1} u_0 \d x$ and $M_r = \int_1^4 u_0 \d x$. 
We consider again a zero external potential, $V_0 = 0$ (Figure \ref{fig:cas6}(a)), and linear diffusion,  $\phi = u$ and $W = 0$. 
Figure \ref{fig:cas6}(b) shows the solutions $u_k$ at $T_f =2$ for $k$ up to 10, and the solution of the limit problem $u$. We observe nice convergence as $k$ increases, see Figures \ref{fig:cas6}(c) and (d) for the evolution of the error. Figure \ref{fig:cas6time} shows the dynamics up to $t = 1$ of the limit problem, a weak confinement case ($k=2$) and a strong confinement case ($k=10$). 
To implement the initial condition \eqref{u0_shifted}, we placed a Dirac delta on the boundary nodes inside $\Omega=[-1,1]$. Let us denote the grid points in $B=[-4,4]$, regularly spaced by $\Delta x$, as 
$
[l_1, l_2, \dots, l_M, x_1, x_2, \dots, x_N, r_1, r_2, \dots, r_M].
$
Here $l_i$ and $r_i$ are points outside of $\Omega$ (respectively to the left or to the right), and $x_i$ are inside $\Omega$. The points are chosen so that $x_1 = -1 + \Delta x/2$ and $x_N= 1-\Delta x/2$ (as standard in finite-volume schemes).
The initial datum for the limit problem in $\Omega$ is set as follows:
\begin{align*}
    M_l &= \Delta x \sum_{i=1}^M u_0(l_i), \qquad M_r = \Delta x \sum_{i=1}^M u_0(r_i),\\
    \bar u_0^i &= u_0(x_i), \quad i = 2, \dots, N-1,\\
    \bar u_0^1 &= u_0(x_1) + M_l/\Delta x, \qquad \bar u_0^N = u_0(x_N) + M_r/\Delta x.
\end{align*}
Note that the factor $1/\Delta x$ is required as we are spreading a punctual mass over a compartment of length $\Delta x$.

\def \scc {0.7}
\def \scl {0.8}
\begin{figure}
\unitlength=1cm
\begin{center}
\vspace{3mm}
\psfrag{a}[l][][\scl]{(a)} \psfrag{b}[l][][\scl]{(b)}
\psfrag{c}[l][][\scl]{(c)} \psfrag{d}[l][][\scl]{(d)}
\psfrag{x}[][][\scl]{$x$} \psfrag{k}[][][\scl]{$k$} 
\psfrag{Vk}[][][\scl]{$V_k$} \psfrag{u}[b][][\scl]{$u_k(T_f)$} 
\psfrag{normt}[b][][\scc]{$\|u-u_k\|$} \psfrag{norm}[b][][\scc]{$\|u-u_k\|, \|u_k\|$} \psfrag{t}[][][\scl]{$t$}
	\includegraphics[height = \size]{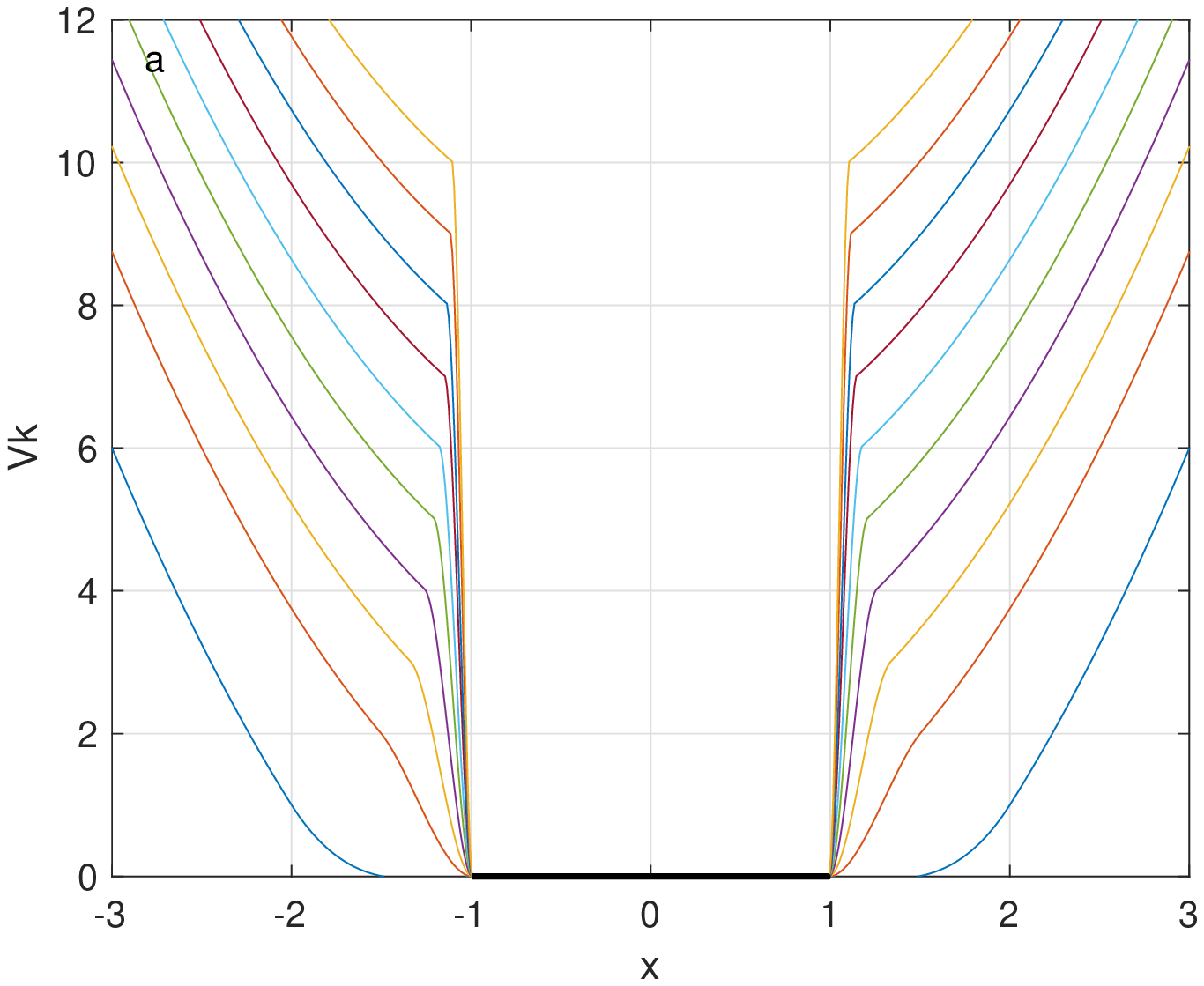} \qquad 
	\includegraphics[height = \size]{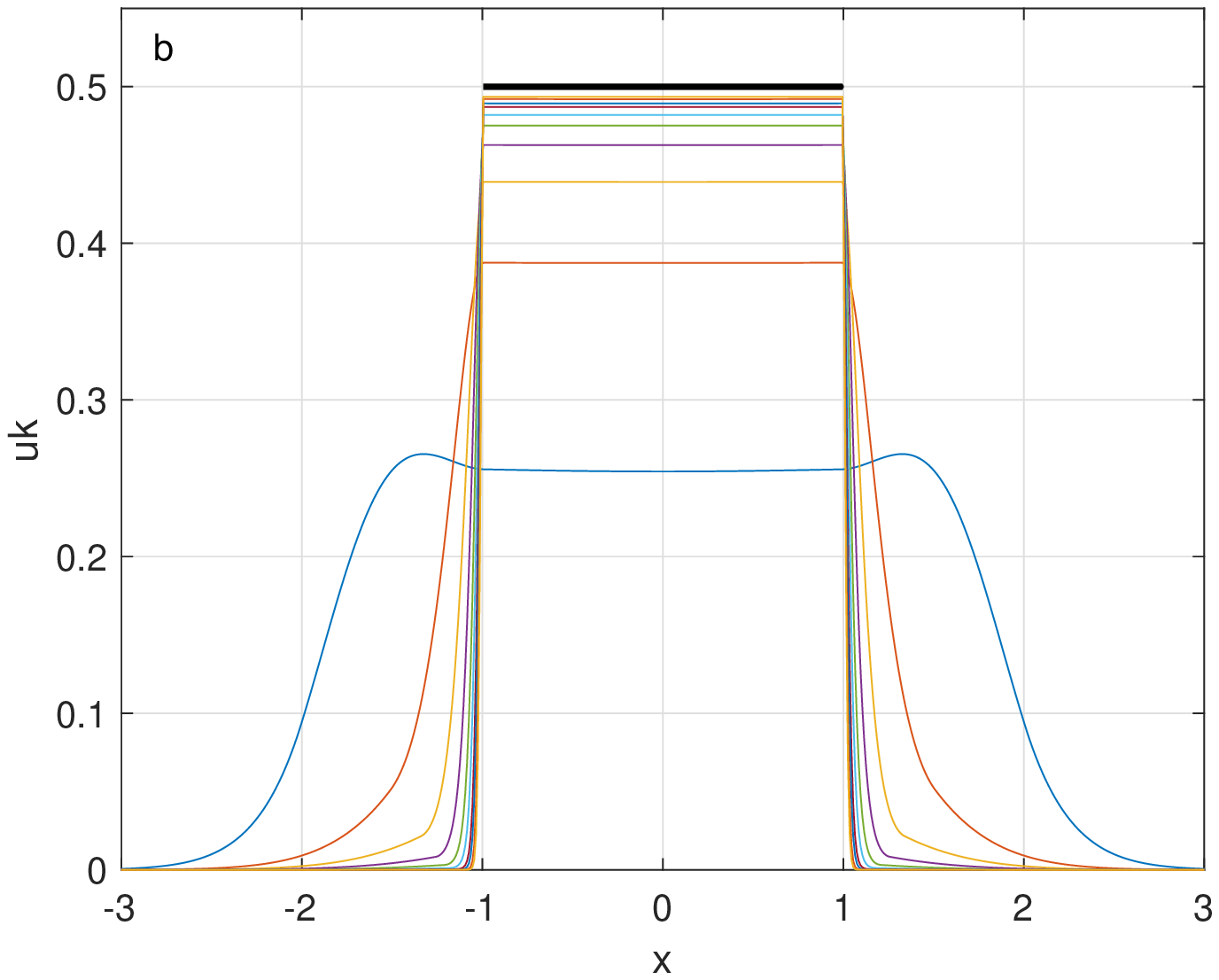}\\ \vspace{5mm}
	\includegraphics[height = \size]{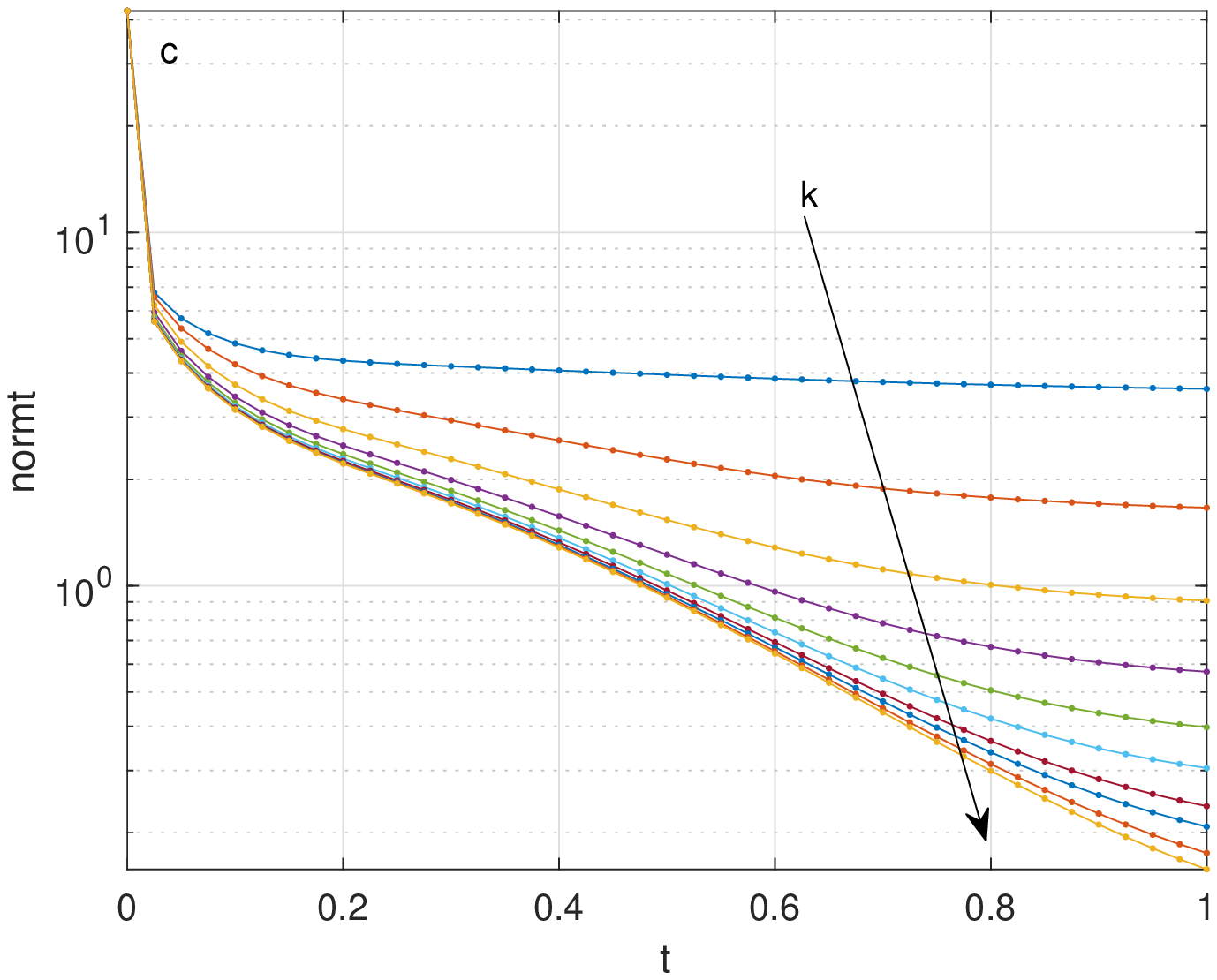} \qquad
	\includegraphics[height = \size]{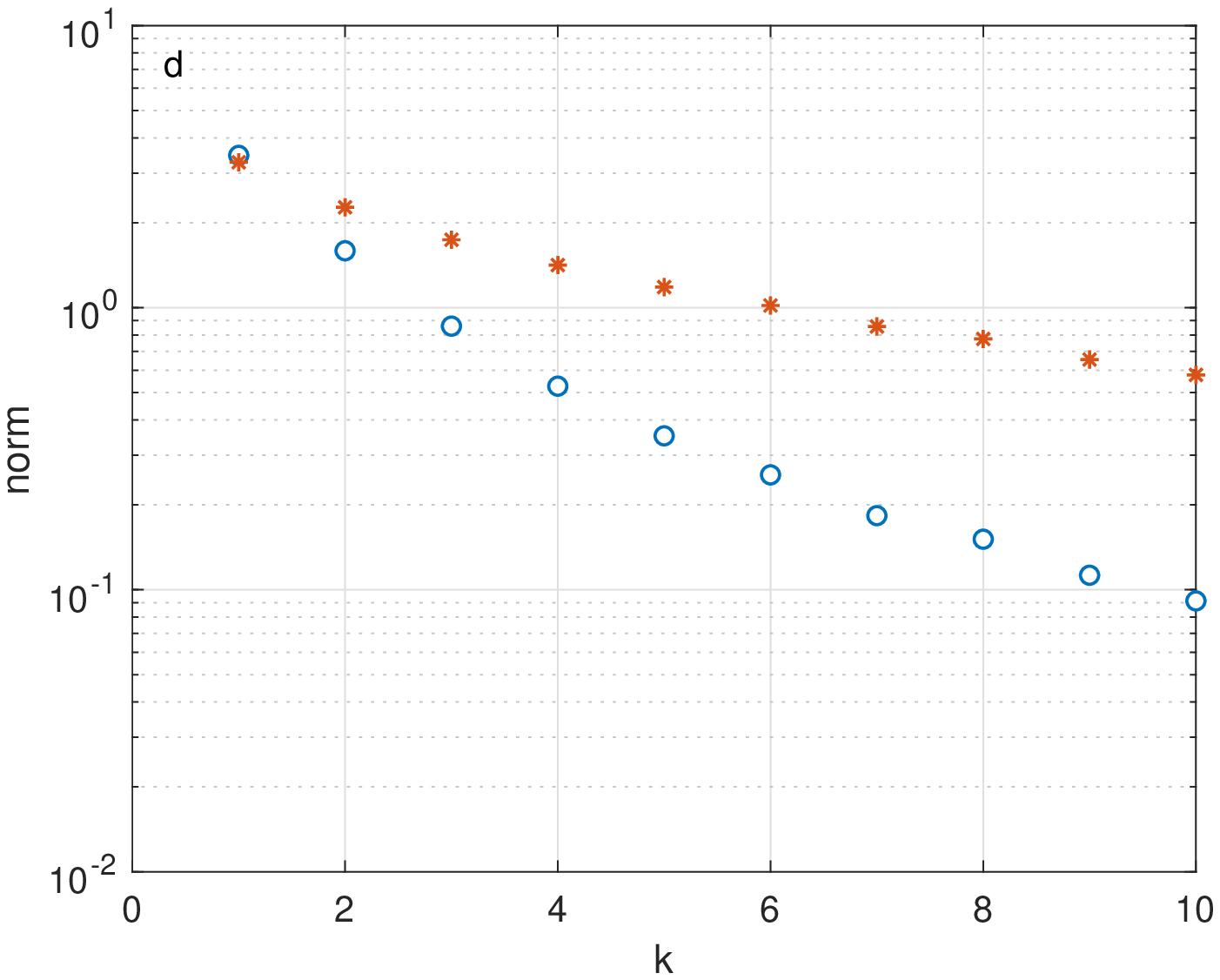} 
\caption{Example with an initial condition with support outside $\Omega$, $u_0$ Gaussian with $\mu = 0$ and $\sigma = 2$. Linear diffusion, $V_0 = 0$ and $T_f = 2$. Simulation with $\Omega = [-1, 1]$, computational domain for $u_k$ is $B = [-4, 4]$ ($k \ge 1$). Grid spacing $\Delta x = 0.01$ and initial time-step $\Delta t = 10^{-5}$. (a) Confining potential $V_k$ for $k = 1, \dots, 10$ (potential $V$ shown in thick black line). (b) Solutions $u$ (thick black line) and $u_k$ (colored thin lines) at $t = T_f$. (c) Norm between $u$ and $u_k$ in $\Omega$ at times $t \in [0, 0.1]$. (d) Norm between $u$ and $u_k$ in $\Omega$ at $t = T_f$ (circles) and norm of $u_k$ in $B \setminus \Omega$ at $t = T_f$ (asterisks).}
\label{fig:cas6}
  \end{center}
\end{figure}

\begin{figure}
\begin{center}
\includegraphics[width = .325\textwidth]{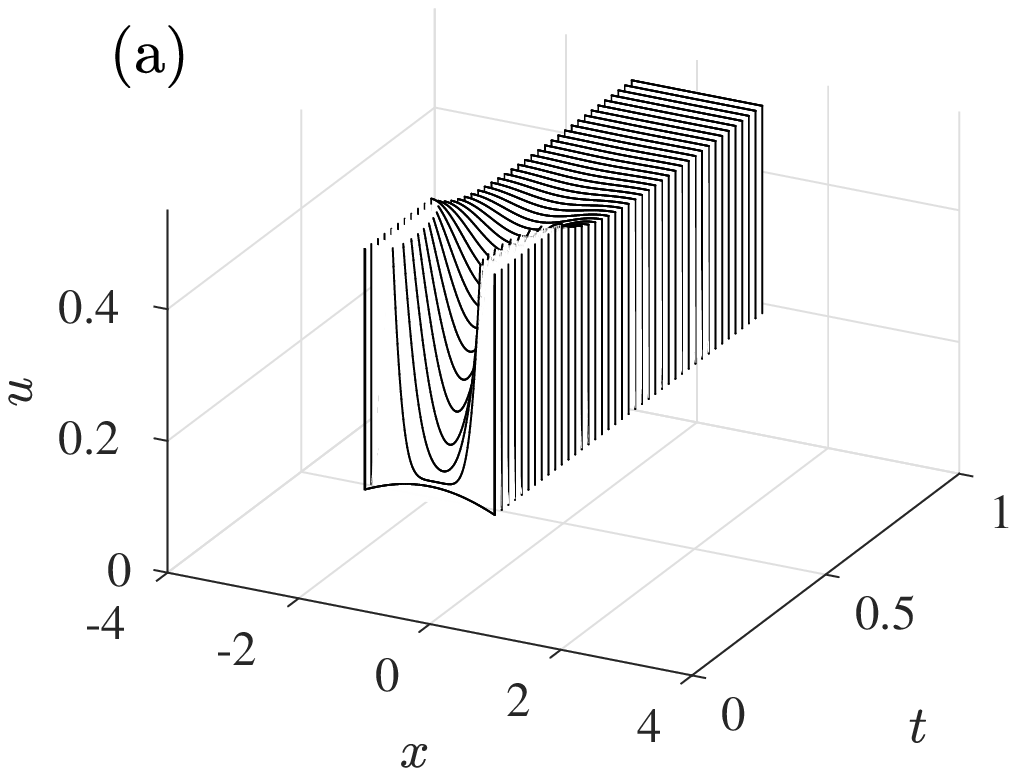} 
\includegraphics[width = .325\textwidth]{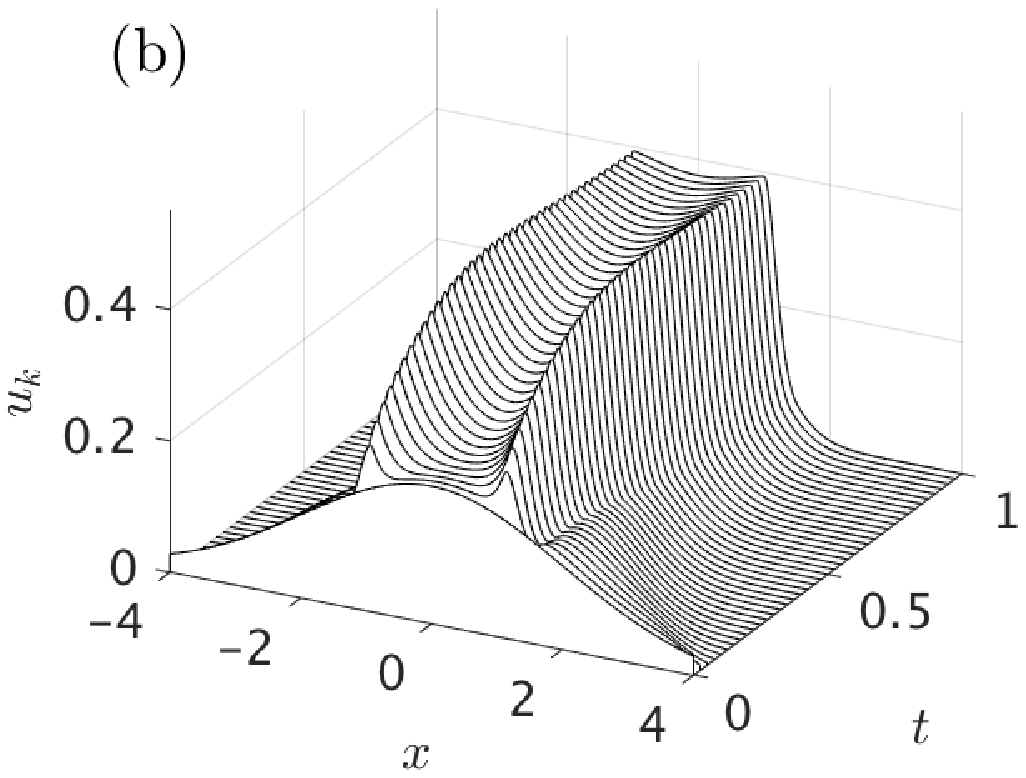} 
\includegraphics[width = .325\textwidth]{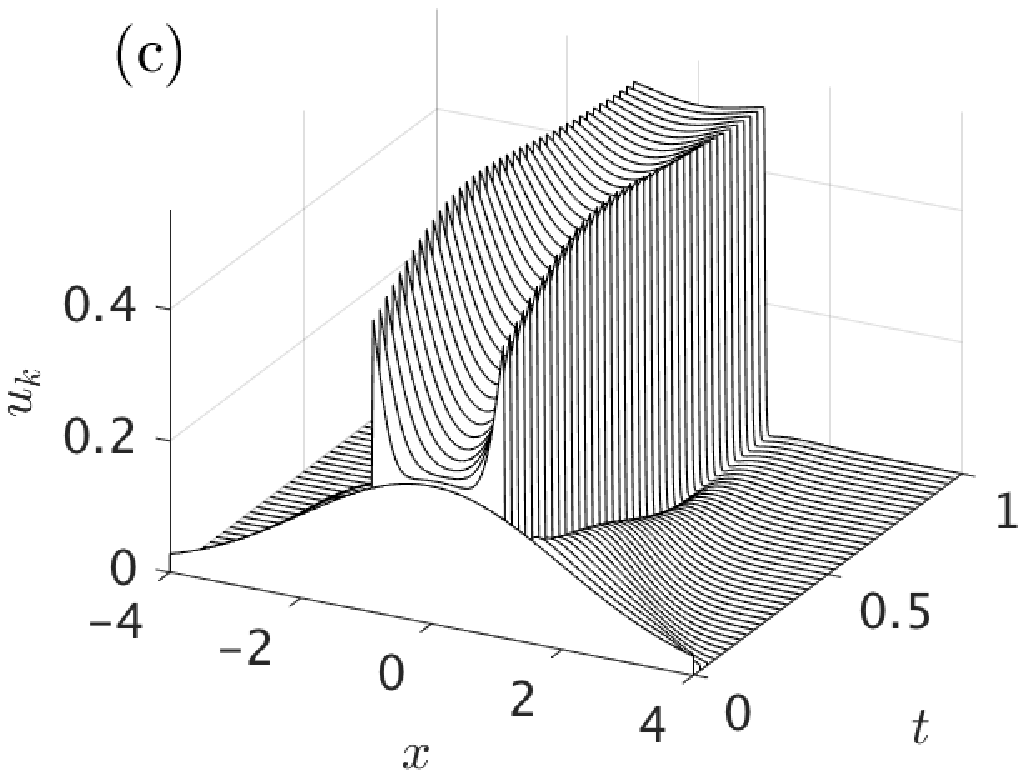} 
\caption{
The dynamics of the solutions shown in Figure \ref{fig:cas6} until time $t = 1$ for (a) in $\Omega$ ($k=0$), (b) a weak confining potential ($k = 2$), and (c) a strong confining potential ($k = 10$). In (a) the vertical axis is truncated (it'd go up to 30, mass of the delta functions at the edges).}
\label{fig:cas6time}
  \end{center}
\end{figure}

\newpage
\FloatBarrier

\subsection{Two-dimensional examples}

In one dimension it seems reasonable to say that \eqref{u0_shifted} is the only way to move the initial mass outside $\Omega$ towards $\partial \Omega$. This would imply the convergence towards a unique limit problem regardless of the confining potential $V_k$. However, it is not clear if the same would hold true in higher dimensions, where there are multiple ways of transporting mass from outside $\Omega$ to $\partial \Omega$. For example, suppose that $\Omega = \{x \in \mathbb R^2:  |x | \le 1 \}$. Then among many options, the mass could be sent to $|x | = 1$ radially or proportionally to the strength of $V_k$.

To explore what happens when the initial data has support outside $\Omega$ in two dimensions, we consider the square domain $\Omega = [-1, 1]^2$ and $B = [-4, 4]^2$. Again for simplicity we work with the linear problem, setting $\phi = u$ and $W = 0$, and also $V_0 =0$. We choose $\Omega_k = [-L_k, L_k]^2$, with $L_k = 1 + 1/k$. We consider four scenarios, combining the cases when the initial datum $u_0$ and/or the confinement potential $V_k$ are radially symmetric or not. As a radially symmetric initial data we use the following volcano-shaped function (see Figure \ref{fig:cas2d_u0}(a))
\begin{equation}
	\label{u0_sym}
	u_0 = C e^{-5(r-1)^2},
\end{equation}
where $C$ is a normalization constant (so that $u_0$ has unit mass in $B$) and $r = \sqrt{x_1^2 + x_2^2}$ with $(x_1,x_2) \in B = [-4, 4]^2$. As a non-radially symmetric initial data we use (see Figure \ref{fig:cas2d_u0}(b))
\begin{equation}
	\label{u0_nonsym}
	u_0 =  C \left[1+ \frac{1}{2}\sin\left(\frac{x_2}{r}\right)\right] e^{-5(r-1)^2},
\end{equation}
where $C$ is again the normalization constant. 

\begin{figure}
	\includegraphics[width = .49\textwidth]{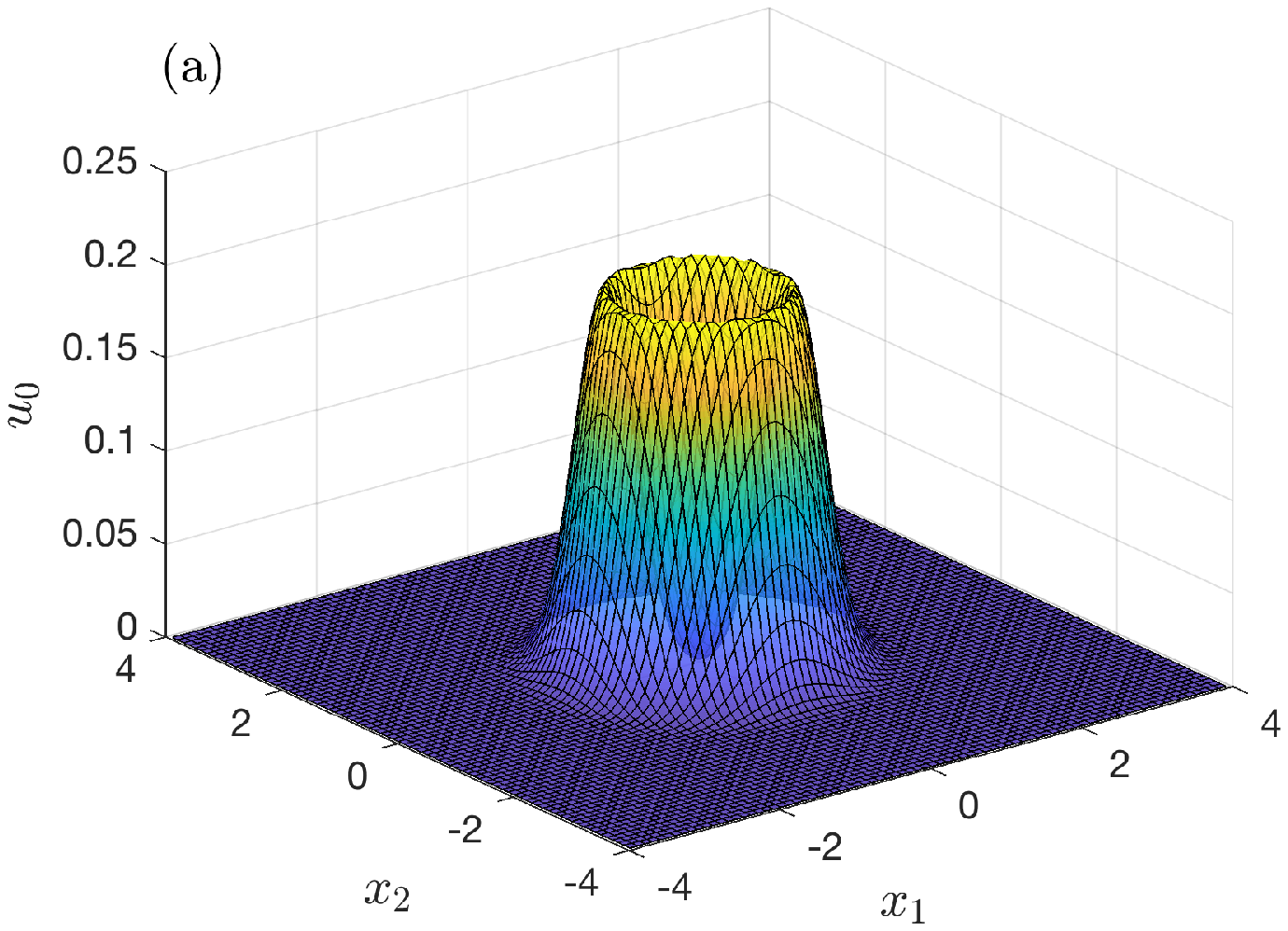}  \quad 	\includegraphics[width = .49\textwidth]{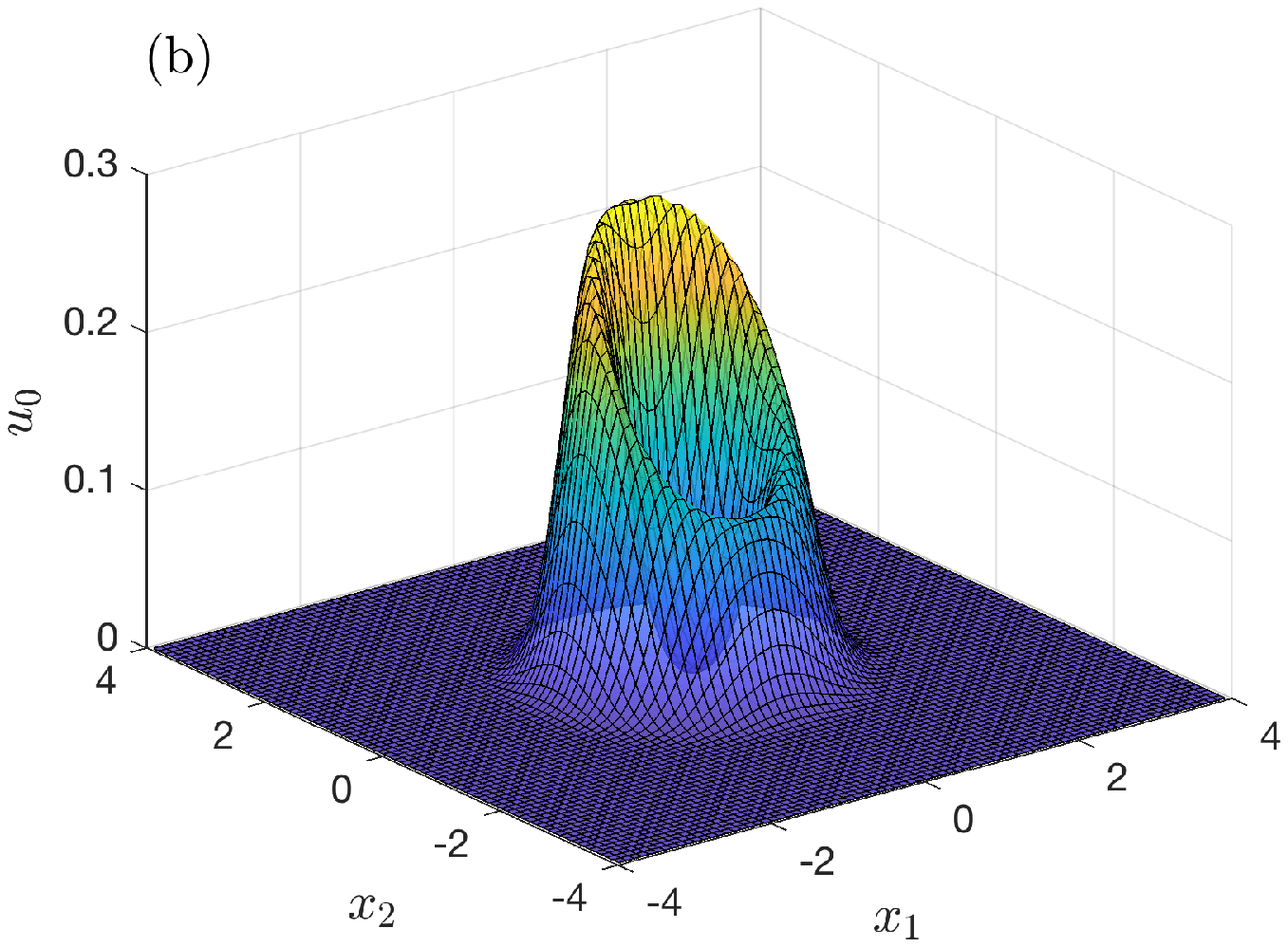} \\ \includegraphics[width = .49\textwidth]{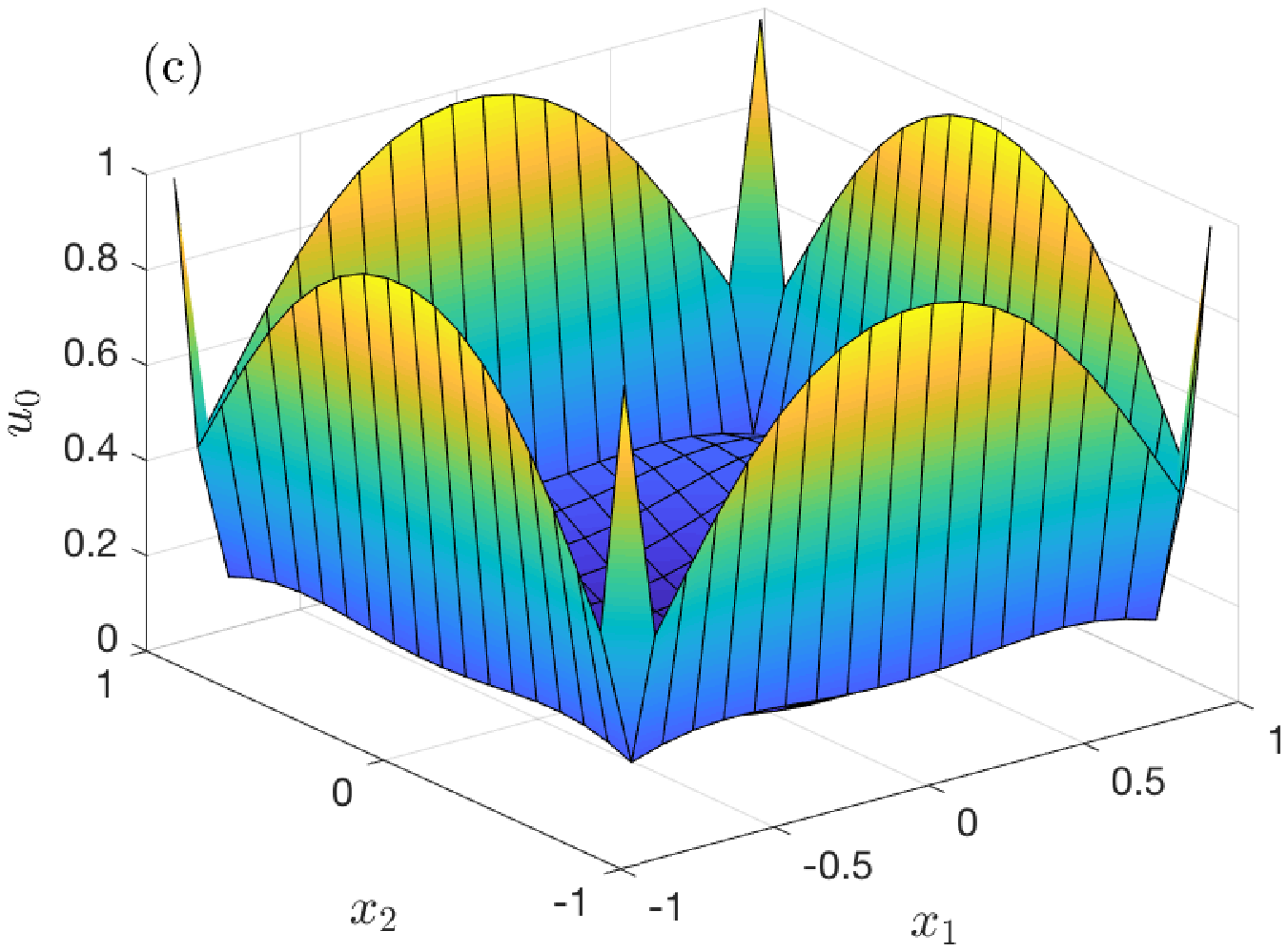} \quad
\includegraphics[width = .49\textwidth]{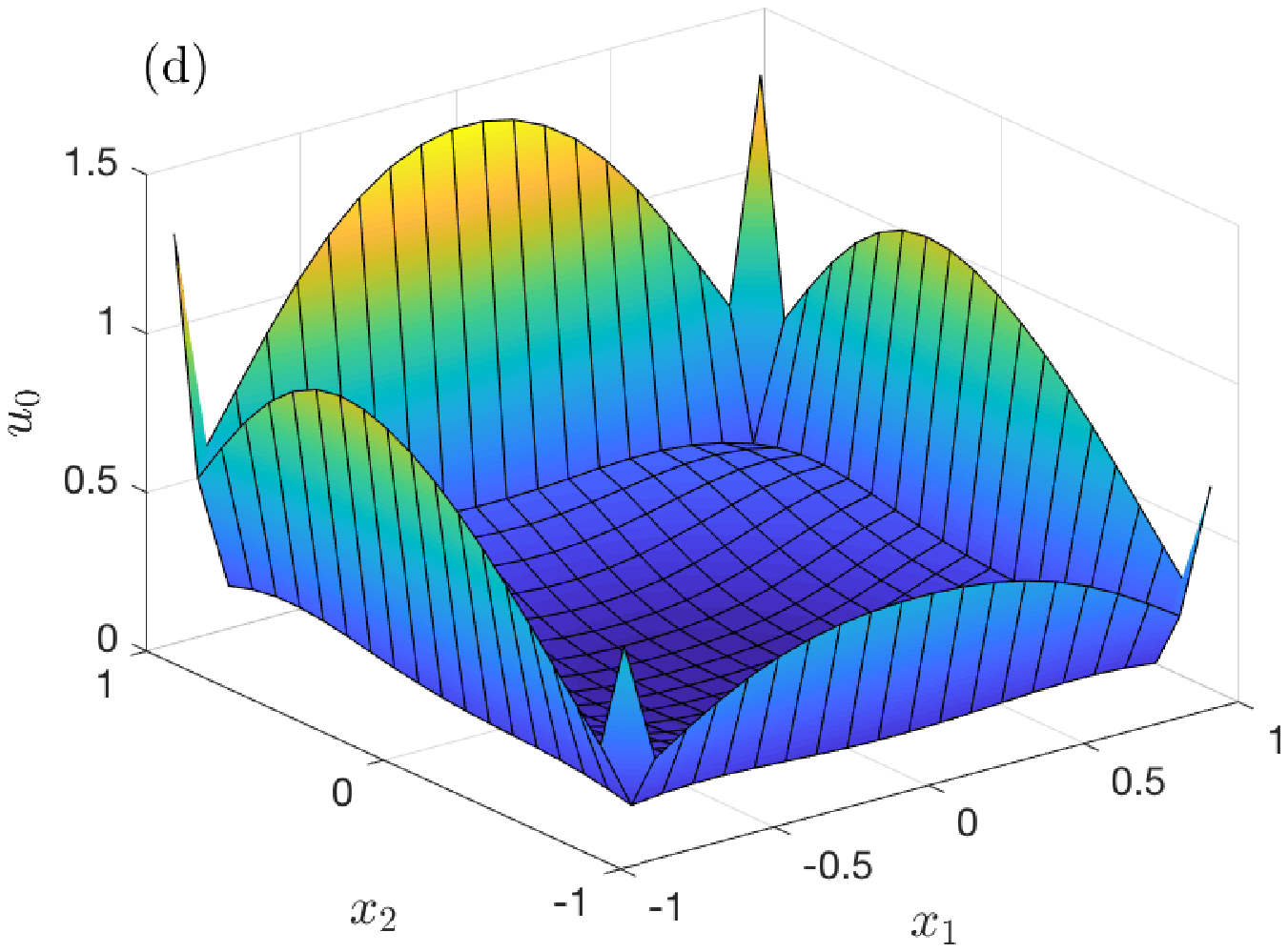} \\ \includegraphics[width = .49\textwidth]{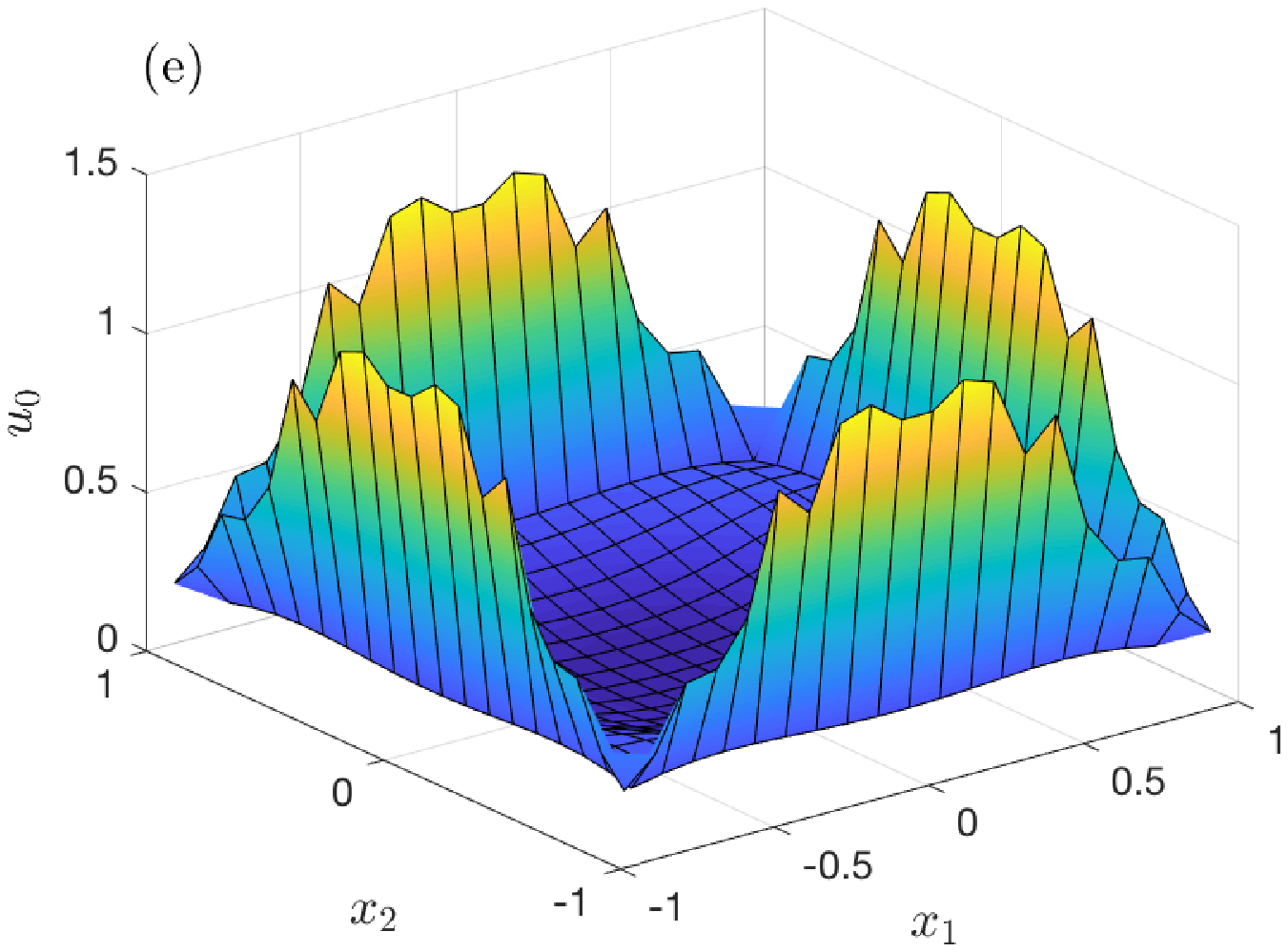} \quad
\includegraphics[width = .49\textwidth]{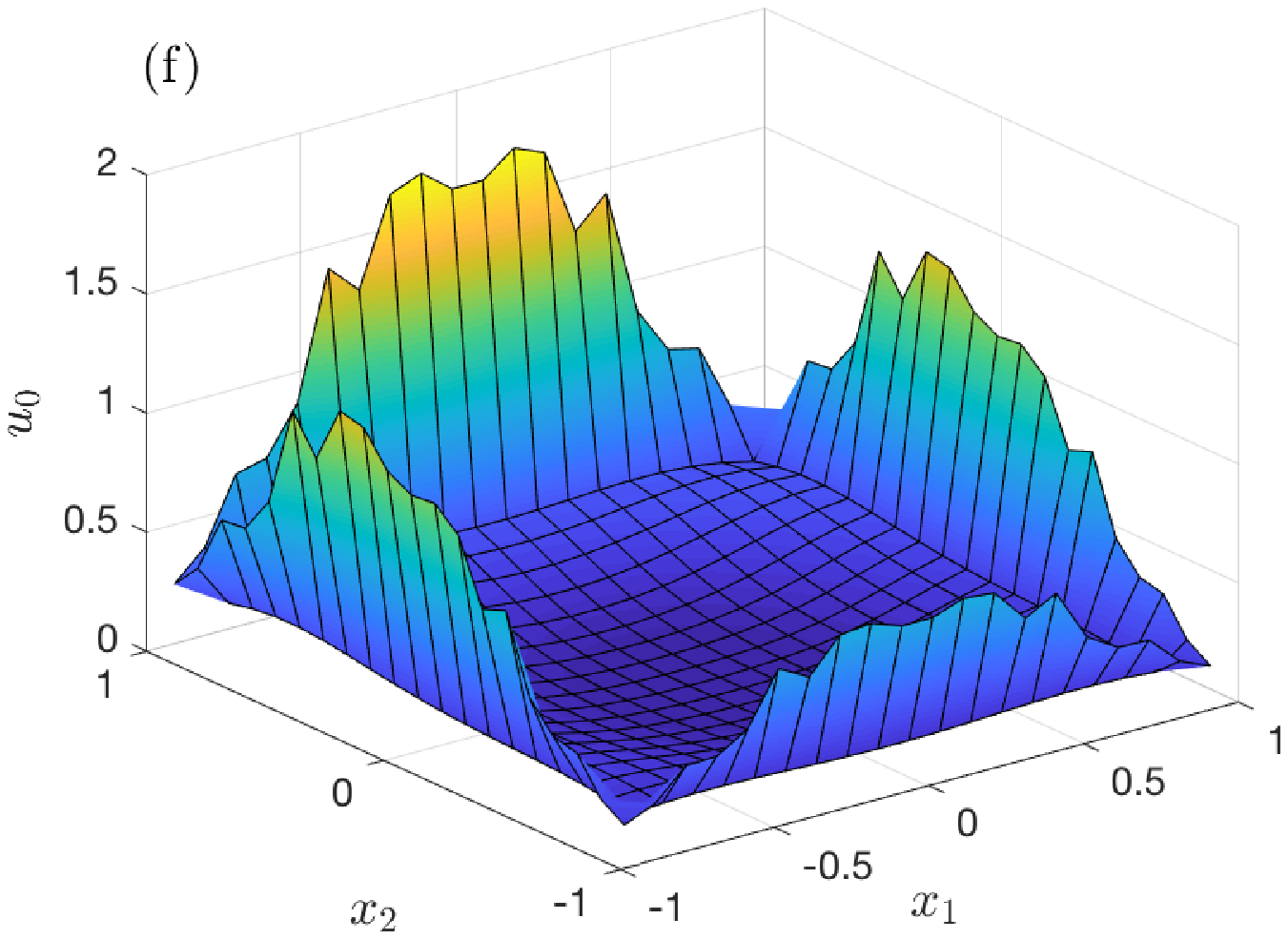} 
	\caption{Initial conditions used for the two-dimensional examples in $B$ (a-b) and in $\Omega$ (c-f). (a) Radially symmetric $u_0$. (b) Non-radially symmetric $u_0$. (c) Radially symmetric $\bar u_0$ obtained by the first mass transport method. (d) Non-radially symmetric $\bar u_0$ obtained by the first mass transport method. (e) Radially symmetric $\bar u_0$ obtained by the second mass transport method. (f) Non-radially symmetric $\bar u_0$ obtained by the second mass transport method. }
\label{fig:cas2d_u0}
\end{figure}

We use two methods to transport the initial mass from outside $\Omega$ to its boundary. The first method consists of sending the mass perpendicular to $\partial \Omega$ and accumulate at the corners of $\partial \Omega$ all the mass in the regions $\{|x_1|>1$ and $|x_2| > 1\}$. 
The resulting initial conditions $\bar u_0$ for the limit problem corresponding to  \eqref{u0_sym} and \eqref{u0_nonsym} are shown
in Figures \ref{fig:cas2d_u0}(c) and \ref{fig:cas2d_u0}(d) respectively. 
This method leads to a concentration of mass at the four corners of $\Omega$. The second method we consider consists of transporting the mass radially from the origin. 
Specifically, if $U_0^{ij} = u_0(x_i, x_j)$ with $(x_i, x_j) \notin \Omega$, then we compute $\theta^{ij} = \tan^{-1}(x_j,x_i)$ (we use the four-quadrant inverse tangent, \texttt{atan2} in Matlab) and send the mass $U_0^{ij}$ to the grid point on $\partial \Omega$ whose angle is closest to $\theta^{ij}$. 
The initial conditions $\bar u_0$ for the limit problem corresponding to \eqref{u0_sym} and \eqref{u0_nonsym} and computed using the second method  are shown in Figures \ref{fig:cas2d_u0}(e) and \ref{fig:cas2d_u0}(f) respectively.

For the confinement potentials, we use $V_k (x) = k + r^2 - L_k^2$ for $x \notin \Omega$ for the radially symmetric case
and $V_k (x) = k + \tilde r^2 - L_k^2$ with $\tilde r = [1 + \sin(x_2/r)/2] r$ for the non-radially symmetric case.
We run simulations for a short time ($T_f = 0.01$) comparing the solution $u_k$ in $B$ for $k = 25$ with either symmetric or asymmetric initial data (see Figures \ref{fig:cas2d_u0}(a-b)) and using either a symmetric or asymmetric confinement potential $V_k$ with the solution $u$ to the limit problem in $\Omega$ with initial data $\bar u_0$ prescribed using one of the two methods described above (see Figures \ref{fig:cas2d_u0}(c-f)). 

In all possible combinations, we find  that the simulation results (not shown) are not particularly sensitive to the shape of the confinement potential, and little sensitive to the method of transporting the mass of $u_0$ from outside $\Omega$ to $\partial \Omega$ (see Figure \ref{fig:cas2d_u0}(c-d) for method 1 and (e-f) for  method 2).
The first transportation method (moving mass perpendicular to $\partial \Omega$) leads to a smaller difference at all times, but this could be determined by the choice of $\Omega$ and its discretization.
The choice of a non-symmetric initial datum affects the shape of the solutions but it doesn't seem to affect the previous considerations concerning mass transportation particularly.
Such insensitivity is somewhat counter-intuitive since, in dimension greater that one, we would expect the limit problem to vary depending on the choice of the confinement potential, and in particular on the way it diverges to infinity outside $\Omega$.

We next try a more extreme  two-dimensional case, still with $\Omega = [-1,1]^2$. 
In particular, we compare a simple confinement potential (without buffer zone, $\Omega_k\equiv \Omega$, and quadratic outside $\Omega$, see Figure \ref{fig:moses_rhoV}(a)) and the following potential ( see Figure \ref{fig:moses_rhoV}(b)):
\begin{equation} \label{moses_potential}
	V_k = \begin{cases}
		0, & \text{if } x \in \Omega \text{ or } |x_1| < 1/\sqrt{k},\\
		k + x_1^2 + x_2^2, & \text{otherwise}. 
	\end{cases}
\end{equation}
We work with the linear Fokker--Planck equation in $B = [-4, 4]^2$, initial condition $u_0 = C \exp(-|x|/2)$ with $C$ normalisation condition and zero potential inside $\Omega$, $V_0 = 0$. We look at the behaviour as $k$ increases (we solve for $k=5,10, 15, 20$) and very short times ($t = 0, 0.001, \dots, 0.01$). 

\begin{figure}
\begin{center}
	\includegraphics[width=.45\textwidth]{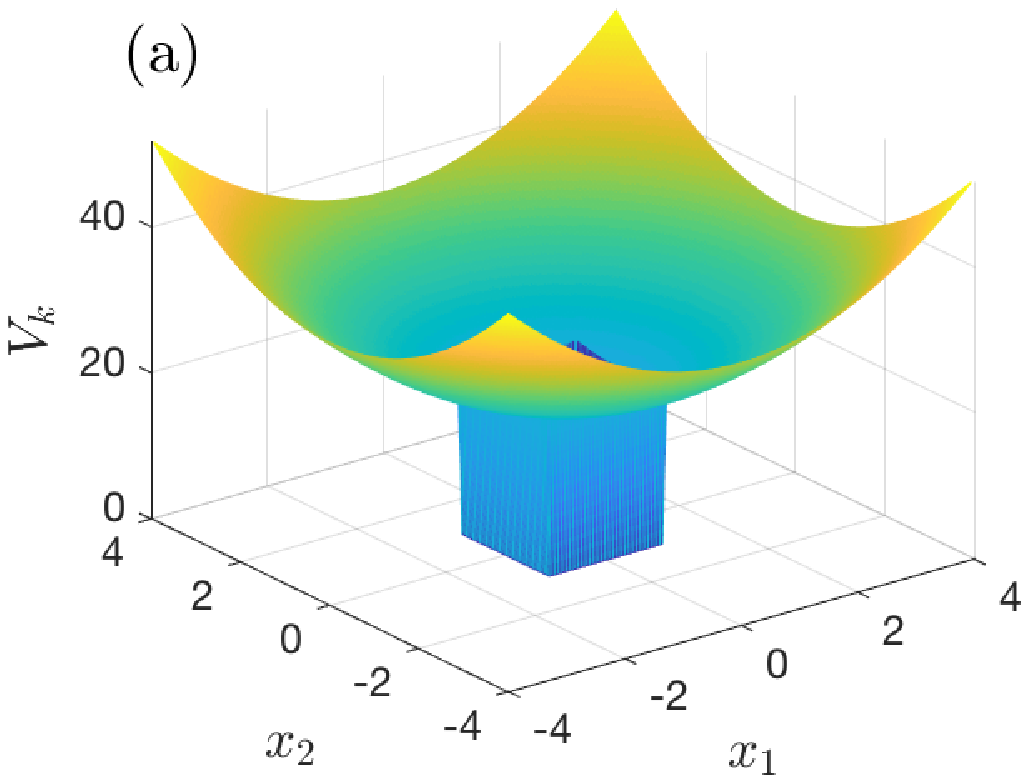} \
	\includegraphics[width=.45\textwidth]{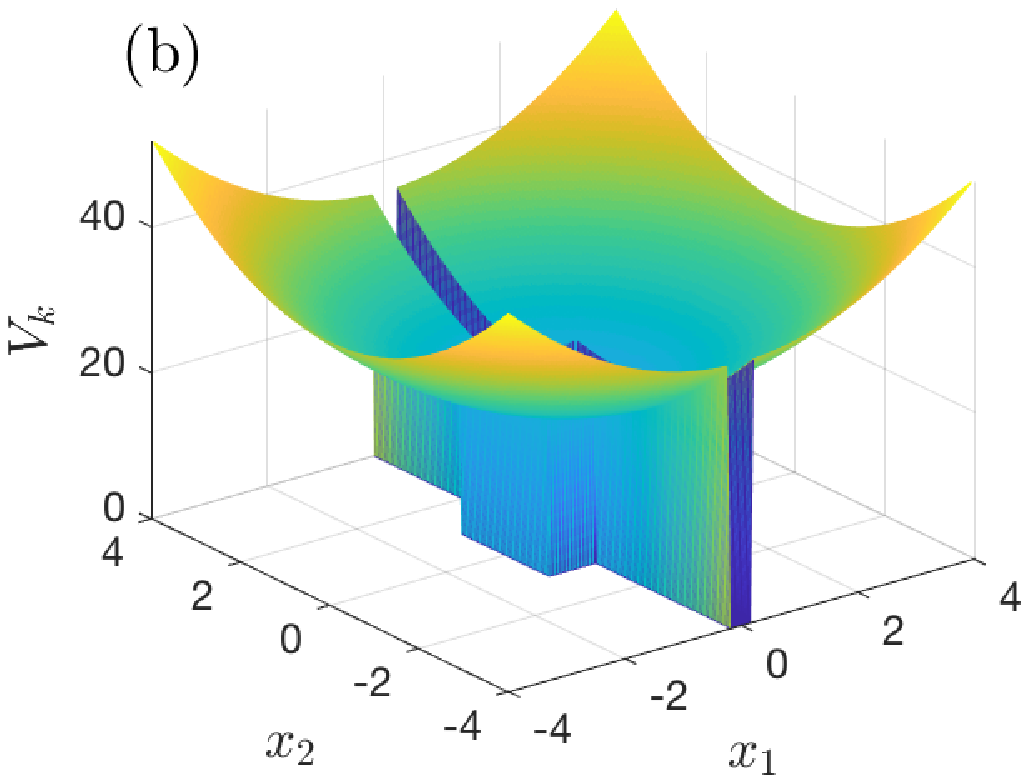} \\
	\includegraphics[width=.45\textwidth]{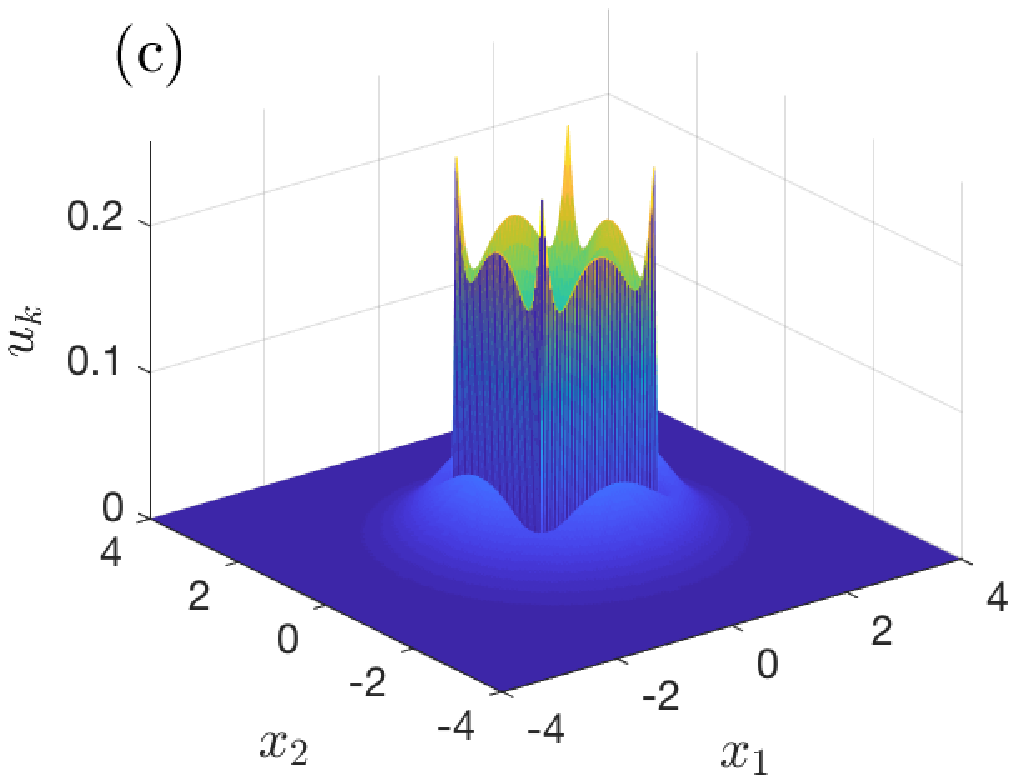}\
	\includegraphics[width=.45\textwidth]{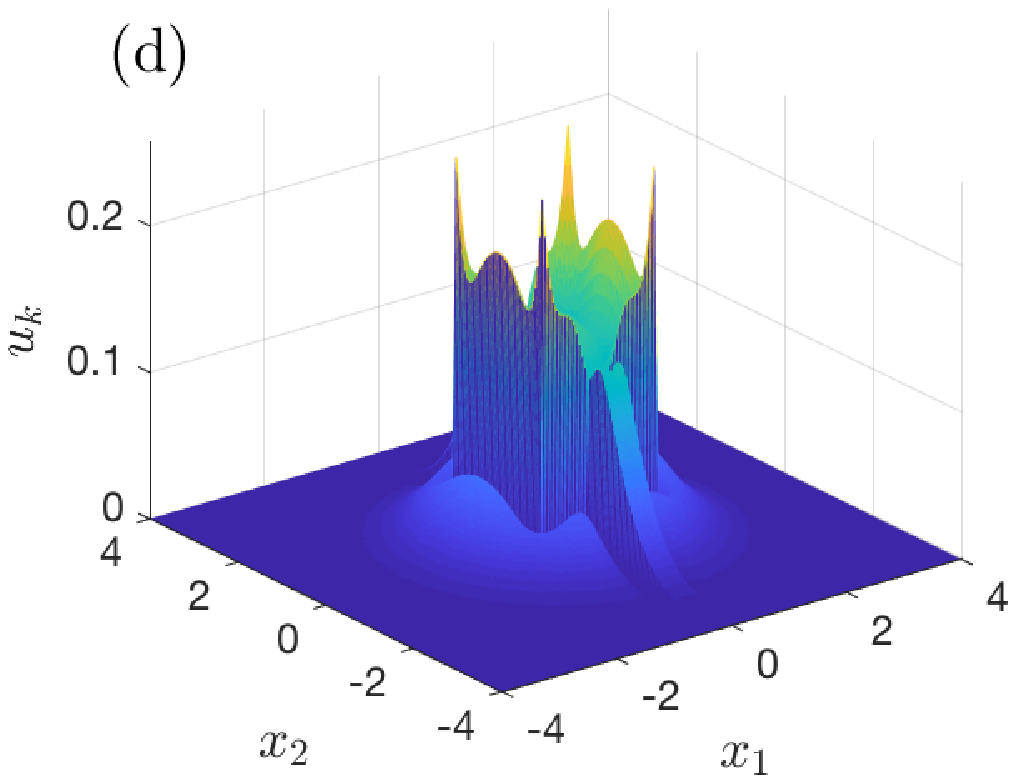}
	\caption{(a) Simple confinement potential $V_k = k + x_1^2 + x_2^2$ outside $\Omega$ and zero otherwise, with $k=20$. (b) Moses confinement potential \eqref{moses_potential} with $k=20$. (c) and (d) Solutions $u_k$ of the linear Fokker--Planck equation at $t = 0.01$ corresponding to the potentials in (a) and (b), respectively. Initial data is $u_0 = C \exp(-|x|/2)$ with $C$ normalisation condition.}
\label{fig:moses_rhoV}
\end{center}
\end{figure}
In Figure \ref{fig:moses_rhoV} we show the two confinement potentials for $k=20$ and their respective solutions at $t = 0.01$. 
In Figure \ref{fig:moses_k} we plot the  value of the solutions at the various times for the different $k$ along $\partial \Omega$, parameterized with arc length $s$ (starting from $x = (1,1)$ and going round clockwise). 
The lines corresponding to the two potentials seem to slowly converge to the same profile as the parameter $k$ increases.

\def \scc {0.7}
\def \scl {.9}
\begin{figure}
\unitlength=1cm
\begin{center}
\vspace{3mm}
\psfrag{a}[][][\scl]{\ (a)} \psfrag{b}[][][\scl]{\ (b)}
\psfrag{c}[][][\scl]{\ (c)} \psfrag{d}[][][\scl]{\ (d)}
\psfrag{rho}[][][\scl][-90]{$u_k$}
\psfrag{s}[][][\scl]{$s$} \psfrag{t}[][][\scl]{$t$}
	\includegraphics[width=.4\textwidth]{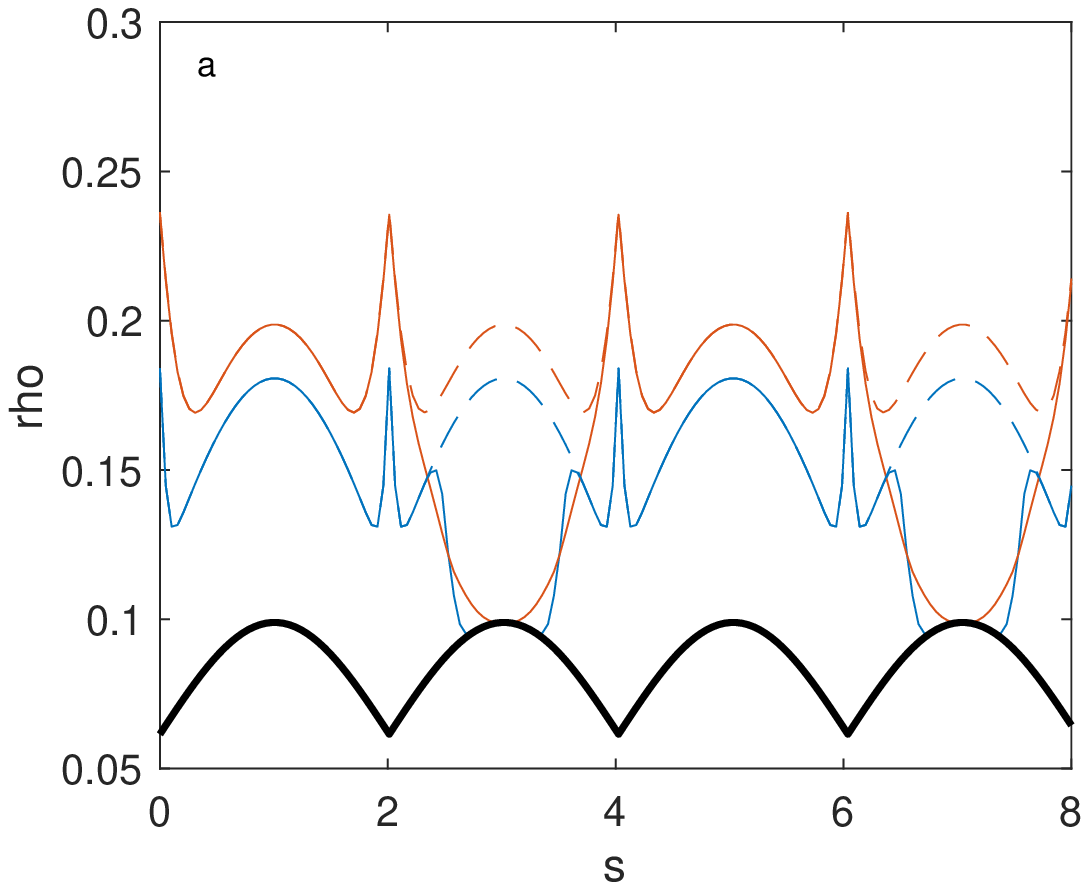} \qquad
	\includegraphics[width=.4\textwidth]{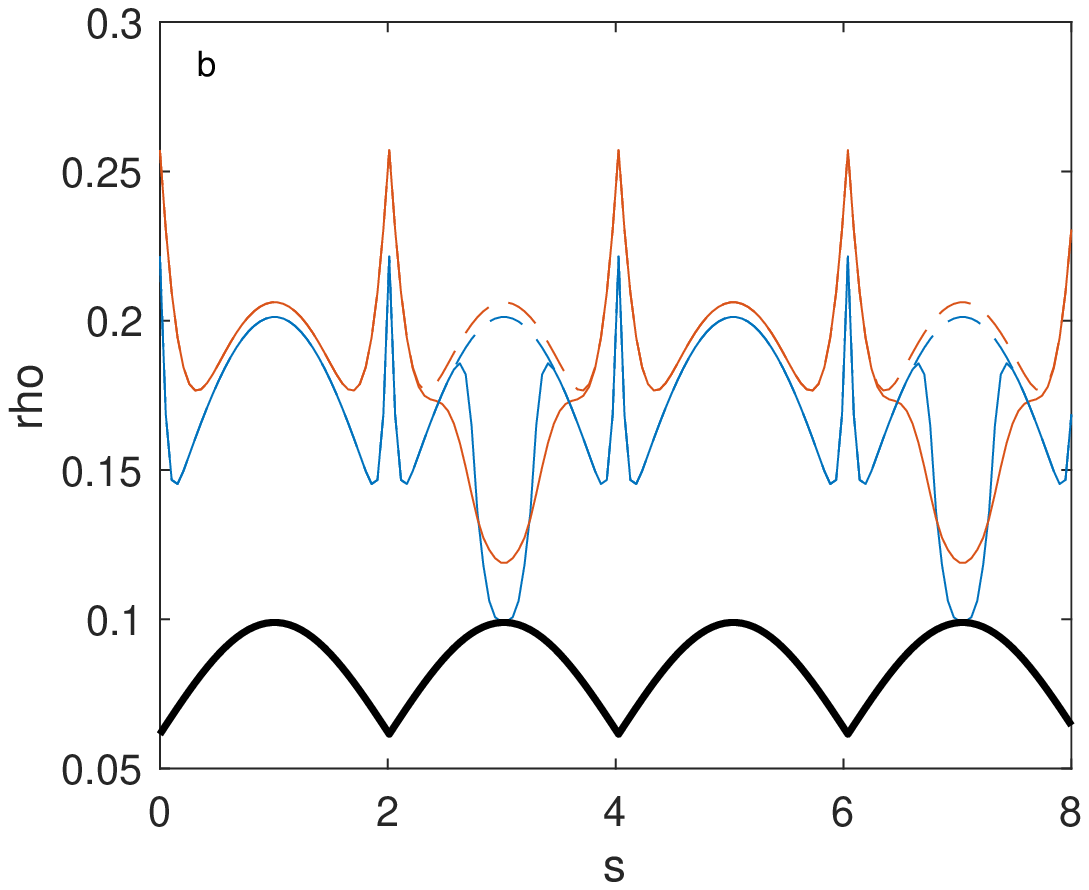} \\ \vspace{.5cm}
		\includegraphics[width=.4\textwidth]{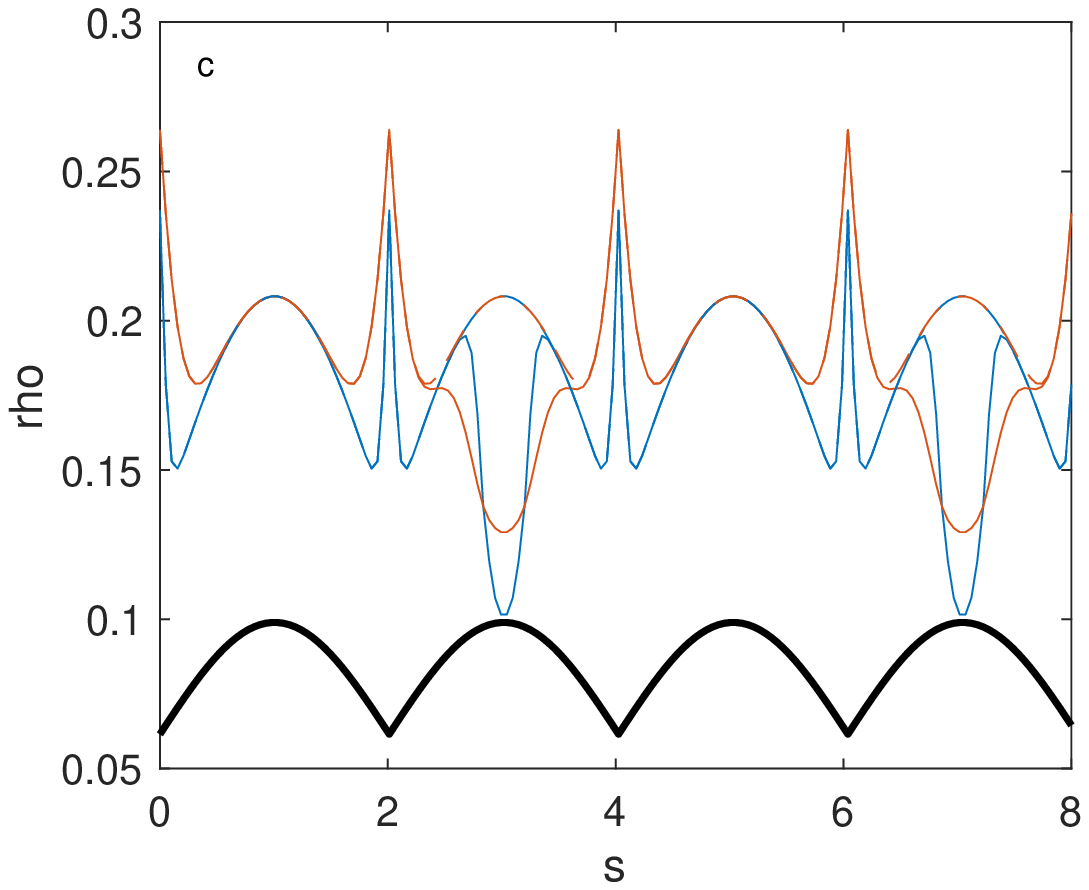} \qquad
	\includegraphics[width=.4\textwidth]{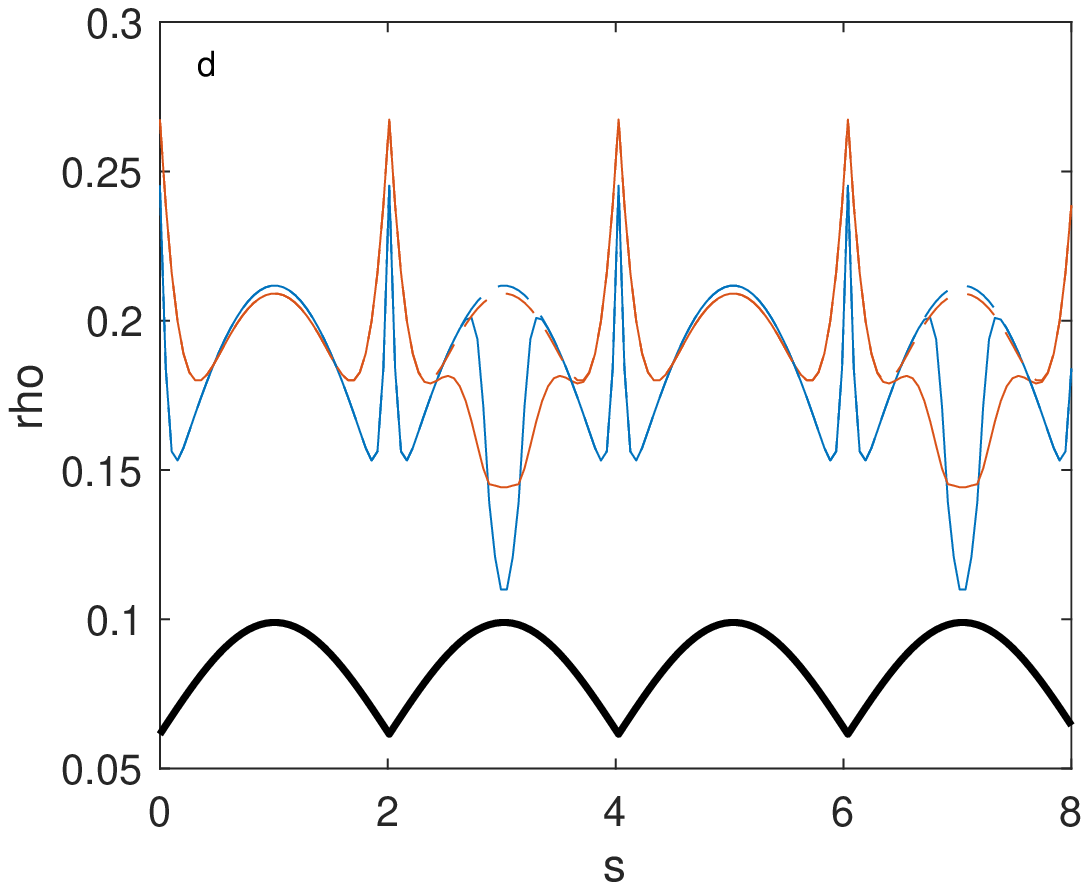}
	\caption{Solution $u_k$ along $\partial \Omega$ at times $t=0.001$ (shown in blue) and $t = 0.01$ (shown in red) using either the simple quadratic potential (dashed lines) or the Moses potential (solid lines) for (a) $k=5$, (b) $k=20$, (c) $k=50$, and (d) $k=100$. The initial condition is the same in all cases and shown as a black thick line. 
	The horizontal axis represents $\partial \Omega$, parameterized with arc length $s$ (starting from $x = (1,1)$ and going round clockwise).}
\label{fig:moses_k}
\end{center}
\end{figure}

\FloatBarrier

In conclusion, our numerical simulations in two dimensions did not give a clear indication in terms of (non-)uniqueness of the limit problem for $k\to\infty$. We formulate the following conjecture:
\begin{conj*}
Let $d\geq 2$ and suppose that the hypotheses of Theorem \ref{thm:main1} 
(resp. Theorem \ref{thm:main2}) 
are satisfied, but assume that $\supp(u_0) \nsubseteq \Omega$. 
Consider a solution $u_k$ of the Cauchy problem \eqref{eq:main} in the sense of Definition \ref{def:weaksol1} 
(resp. Definition \ref{def:weaksol2}) 
with $V=V_k$ satisfying the conditions in Definition \ref{defi:vk}
(resp. Definition \ref{defi:vk-ent}).
Then the sequence $u_k$ does not converge to the solution of a unique limit problem for $k\to\infty$. In fact, as $k\to\infty$, the mass outside $\Omega$ (namely $\int_{\Omega^c} u_0 \d x$) accumulates on the boundary $\de\Omega$, resulting in a singular, measure-valued initial datum of the form $u_0\big|_{\Omega} + \mathcal{M}\big|_{\de\Omega}$, where $\mathcal{M}$ is a non-negative measure concentrated on $\de\Omega$. The measure $\mathcal{M}$ is not uniquely determined and it can vary depending on  the properties of $\Omega$, $u_0$ and the sequence $V_k$.
\end{conj*}

%
%
%
%
\section*{Acknowledgements}
L.~Alasio was partially supported by the EPSRC grant number EP/L015811/1. M.~Bruna was funded by a  Junior Research Fellowship from St John's College, Oxford. J.~A.~Carrillo was partially supported by the EPSRC grant number EP/P031587/1.

%
%
%
%

\end{document}